\documentclass[11pt]{amsart}

\usepackage[
paper=a4paper,
text={147mm,230mm},centering
]{geometry}

\iftrue
\makeatletter
\def\@settitle{%
  \vspace*{-20pt}
  \begin{flushleft}%
    \baselineskip14\p@\relax
    \normalfont\bfseries\LARGE
    \@title
  \end{flushleft}%
}
\def\@setauthors{%t
  \begingroup
  \def\thanks{\protect\thanks@warning}%
  \trivlist
  \large \@topsep30\p@\relax
  \advance\@topsep by -\baselineskip
  \item\relax
  \author@andify\authors
  \def\\{\protect\linebreak}%

  \authors
  \ifx\@empty\contribs
  \else
    ,\penalty-3 \space \@setcontribs
    \@closetoccontribs
  \fi
  \normalfont
  \endtrivlist
  \endgroup
}
\def\@setabstracta{%
    \ifvoid\abstractbox
  \else
    \skip@25\p@ \advance\skip@-\lastskip
    \advance\skip@-\baselineskip \vskip\skip@
    \box\abstractbox
    \prevdepth\z@ % because \abstractbox is a vtop
    \vskip-10pt
  \fi
}
\renewenvironment{abstract}{%
  \ifx\maketitle\relax
    \ClassWarning{\@classname}{Abstract should precede
      \protect\maketitle\space in AMS document classes; reported}%
  \fi
  \global\setbox\abstractbox=\vtop \bgroup
    \normalfont\small
    \list{}{\labelwidth\z@
      \leftmargin0pc \rightmargin\leftmargin
      \listparindent\normalparindent \itemindent\z@
      \parsep\z@ \@plus\p@
      
    }%
    \item[\hskip\labelsep\bfseries\abstractname.]%
}{%
  \endlist\egroup
  \ifx\@setabstract\relax \@setabstracta \fi
}
\def\section{\@startsection{section}{1}%
  \z@{-1.2\linespacing\@plus-.5\linespacing}{.8\linespacing}%
  {\normalfont\bfseries\large}}
\def\subsection{\@startsection{subsection}{2}%
  \z@{-.8\linespacing\@plus-.3\linespacing}{.3\linespacing\@plus.2\linespacing}%
  {\normalfont\bfseries}}
\def\subsubsection{\@startsection{subsubsection}{3}%
  \z@{.7\linespacing\@plus.1\linespacing}{-1.5ex}%
  {\normalfont\itshape}}
\def\@secnumfont{\bfseries}
\makeatother
\fi %\iftrue/false for amsart.cls hack

\usepackage{amssymb}
\usepackage{mathrsfs}
\usepackage[all]{xy}
\usepackage{graphicx,color,float}
\usepackage[bookmarks]{hyperref}
\usepackage{pinlabel}
\usepackage[textwidth=1.3in,color=yellow]{todonotes}
\usepackage{amsmath}
\usepackage{pb-diagram,pb-xy}
\usepackage{amsmath}
\usepackage{amsfonts}
\usepackage{graphicx}
\usepackage{amscd}
\usepackage{mathtools}
\usepackage{url}
\usepackage{caption}
\usepackage{subcaption}
\usepackage{fancybox}
\usepackage{wrapfig}
\usepackage{color}
\usepackage{multirow, url}
\usepackage{tikz}
\usetikzlibrary{shapes.geometric, arrows}
\usetikzlibrary{shapes}
\usepackage{xcolor}
\usepackage[normalem]{ulem}  %sout

\usetikzlibrary{calc}
\textwidth=\paperwidth \advance\textwidth by-2\oddsidemargin
\advance\oddsidemargin by-1in
\evensidemargin=\oddsidemargin
\calclayout

\renewcommand{\bar}{\overline}
\renewcommand{\tilde}{\widetilde}

\newcommand{\C}{\mathbb{C}}

\newcommand{\Q}{\mathbb{Q}}
\newcommand{\R}{\mathbb{R}}

\newcommand{\Z}{\mathbb{Z}}

\newcommand{\tf}{\tilde{f}}

\newcommand{\te}{\tilde{e}}

\tikzstyle{startstop} = [rectangle, rounded corners, 
minimum width=3cm, 
minimum height=1cm,
text centered, 
draw=black]

\tikzstyle{io} = [trapezium, 
trapezium stretches=true, % A later addition
trapezium left angle=70, 
trapezium right angle=110, 
minimum width=3cm, 
minimum height=1cm, text centered, 
draw=black, fill=blue!30]

\tikzstyle{process} = [rectangle, rounded corners, 
minimum width=4cm, 
minimum height=1cm,
text centered, 
draw=black]

\tikzstyle{decision} = [diamond, 
minimum width=3cm, 
minimum height=1cm, 
text centered, 
draw=black, 
fill=green!30]
\tikzstyle{arrow} = [thick,->,>=stealth]

\def\mcal{\mathcal}
\def\frak{\mathfrak}
\def\scr{\mathscr}

\numberwithin{equation}{section}
\theoremstyle{plain}
\newtheorem{theorem}{Theorem}[section]
\newtheorem{thmx}{Theorem}
\newtheorem{propx}[thmx]{Proposition}

\newtheorem{cor}[thmx]{Corollary}
\newtheorem{proposition}[theorem]{Proposition}
\newtheorem{corollary}[theorem]{Corollary}

\newtheorem{lemma}[theorem]{Lemma}

\theoremstyle{definition}
\newtheorem{definition}[theorem]{Definition}

\newtheorem{example}[theorem]{Example}
\newtheorem{remark}[theorem]{Remark}

\def\to{\mathchoice{\longrightarrow}{\rightarrow}{\rightarrow}{\rightarrow}}
\makeatletter
\newcommand{\shortxra}[2][]{\ext@arrow 0359\rightarrowfill@{#1}{#2}}
\def\longrightarrowfill@{\arrowfill@\relbar\relbar\longrightarrow}
\newcommand{\longxra}[2][]{\ext@arrow 0359\longrightarrowfill@{#1}{#2}}
\renewcommand{\xrightarrow}[2][]{\mathchoice{\longxra[#1]{#2}}%
  {\shortxra[#1]{#2}}{\shortxra[#1]{#2}}{\shortxra[#1]{#2}}}
\makeatother
\numberwithin{equation}{section}

\newcommand{\g}{\mathfrak{g}}

\newcommand{\rxw}{\underline{w}}
\newcommand{\fr}{\mathrm{fr}}
\newcommand{\uf}{\mathrm{uf}}
\newcommand{\cmA}{\mathsf{A}}
\newcommand{\vep}{\varepsilon}
\newcommand{\seed}{\mathsf{s}}
\newcommand{\bfa}{\mathbf{a}}
\newcommand{\bfb}{\mathbf{b}}

\newcommand{\Lb}{\mathcal{L}}

\usepackage{marginnote}

\newenvironment{Red}
{\relax\color{red}}
{\hspace*{.5ex}\relax}

\newenvironment{Yel}
{\relax\color{yellow}}
{\hspace*{.5ex}\relax}

\newcommand{\ber}{\begin{Red}}
\newcommand{\er}{\end{Red}}

\newcommand{\beb}{\begin{Yel}}
\newcommand{\eb}{\end{Yel}}
	
\newcommand{\berE}{\begin{Red}{}\marginnote{\fbox{\scshape\lowercase{E}}}{}}

\newcommand{\berMH}{\begin{Red}{}\marginnote{\fbox{\scshape\lowercase{MH}}}{}}

\newcommand{\berYS}{\begin{Red}{}\marginnote{\fbox{\scshape\lowercase{YS}}}{}}

\newcommand{\berYH}{\begin{Red}{}\marginnote{\fbox{\scshape\lowercase{YH}}}{}}

\begin{document}

\title[Newton--Okounkov bodies of partial flag varieties via cluster algebras]{Newton--Okounkov bodies of partial flag varieties via cluster algebras}

\author{Yunhyung Cho}
\address{Department of Mathematics Education, Sungkyunkwan University, Seoul, Republic of Korea}
\email{yunhyung@skku.edu}
\thanks{The research of Y.\ Cho was supported by the National Research Foundation of Korea(NRF) grant funded by the Korea government (MSIT) (RS-2025-00572980 and RS-2020-NR049535) and Samsung Science and Technology Foundation (SSTF-BA2402-01).}

\author{Myungho Kim}
\address{Department of Mathematics, Kyung Hee University, Seoul, Republic of Korea}
\email{mkim@khu.ac.kr}
\thanks{The research of M.  \ Kim was supported by the National Research Foundation of Korea(NRF) grant funded by the Korea government (MSIT) (NRF-2020R1A5A1016126).}

\author{Yoosik Kim}
\address{Department of Mathematics and Institute of Mathematical Science, Pusan National University, Busan, Republic of Korea}
\email{yoosik@pusan.ac.kr}
\thanks{The research of Y.  \ Kim was supported by the National Research Foundation of Korea(NRF) grant funded by the Korea government (MSIT) (RS-2025-16069532 and NRF-2020R1A5A1016126).}

\author{Euiyong Park}
\address{Department of Mathematics, University of Seoul, Seoul, Republic of Korea}
\email{epark@uos.ac.kr}
\thanks{The research of E.\ Park was supported by the National Research Foundation of Korea(NRF) grant funded by the Korea government (MSIT) (RS-2023-00273425 and NRF-2020R1A5A1016126).}

%\thanks{This work was supported by}
%\date{\today}
%\subjclass[2010]{11F33, 11F80, 11G18}
%\keywords{Cuspidal group, Eisenstein ideals}

\begin{abstract}
We construct Newton--Okounkov polytopes of Schubert varieties in partial flag varieties of arbitrary type using the cluster structure on a unipotent cell. When the governing cluster algebra is of infinite type, we prove that for any very ample homogeneous line bundle over a simply laced partial flag variety, the resulting family of Newton--Okounkov polytopes contains infinitely many pairwise nonequivalent polytopes up to integral affine transformation. As an application to symplectic geometry, we construct infinitely many distinct monotone Lagrangian tori in a broad class of simply laced partial flag varieties.
\end{abstract}
\maketitle
\setcounter{tocdepth}{1} %if you want to show more sections on the contents, change this number 1 to 2 or 3.
\tableofcontents

%---------------------------------------------------
\section{Introduction}\label{secIntroduction}

Newton--Okounkov bodies, introduced by Lazarsfeld--Musta\c{t}\u{a} \cite{LM09} and Kaveh--Khovanskii \cite{KK12}, provide a powerful framework for studying linear system on projective varieties via convex geometry. These convex bodies associated to a polarized projective variety $(X,\mathcal{L})$ effectively encode its geometric information. For a suitable choice of valuation, this body becomes a rational polytope $\Delta$, which in turn yields a flat degeneration of $X$ to the toric variety associated with $\Delta$ in \cite{And13}. It has been useful to study the geometry of $X$ via its toric degeneration in a wide range of contexts.

Many foundational examples of Newton--Okounkov bodies of homogeneous spaces were initially studied via combinatorial models such as triangulations, ladder diagrams, plabic graphs, tropicalizations, and wiring diagrams, see \cite{GonciuleaLakshmibai, GP00b, Cal02, AB04, KoganMiller, SS04, Kav15, RW19, CMZ24} for instance. While these constructions are explicit and combinatorially accessible, they typically produce only finitely many Newton--Okounkov polytopes of a single variety.

The theory of Newton--Okounkov bodies has been further developed using the machinery of cluster algebras, see \cite{GHKK18, RW19, BFMN20, FH21, EH22, BCMN24, FO25} for instance. Cluster algebras provide a systematic method for constructing families of Newton--Okounkov bodies and recurrence relations describing how they are related. In many important cases, it is of particular interest to understand how previously known constructions can be unified with those arising from cluster algebras.

More interestingly, when the governing cluster algebra is of infinite type, this method yields infinitely many Newton--Okounkov bodies together with their associated geometric objects. Consequently, it provides a systematic and controlled way for constructing infinitely many geometric objects with nice properties on a single projective variety.

In this vein, a natural question would be whether these objects are genuinely distinct, up to a meaningful notion of equivalence. To answer this question, the authors in \cite{CKKP25} developed a practical criterion for detecting infinitely many pairwise nonequivalent Newton--Okounkov bodies arising from a cluster algebra, up to integral affine transformation. Using this criterion, they showed that there are infinitely many distinct Newton--Okounkov bodies, toric degenerations, and monotone Lagrangian tori for a broad class of complete flag varieties.

The main purpose of this paper is to construct and distinguish families of Newton--Okounkov polytopes of Schubert varieties in partial flag varieties governed by a cluster algebra.  More specifically, we investigate whether the resulting family contains infinitely many pairwise distinct combinatorial and geometric objects for Schubert varieties when the underlying cluster algebra is of infinite type.

Let $G$ be a simply connected semisimple algebraic group over $\C$ with the index set $I$ of simple roots of its Lie algebra $\mathfrak{g}$. Let $\textsf{A} = (a_{i,j})_{i,j\in I}$ be the Cartan matrix. For a subset $K \subseteq I$, we denote by $P$ the parabolic subgroup associated to $K$. Let $W$ be the Weyl group and $W_K$ the subgroup generated by the simple reflections indexed by $K$. We write $W^K$ for the set of minimal length representatives of the cosets in $W / W_K$. For each $w \in W^K$, let $X^K_w$ be the associated Schubert variety in $G/P$. 

Following \cite{BFZ05, Williams13, GLS11, GY17, FO25}, we consider the cluster algebra governing the combinatorial and geometric objects associated with $X_w^K$. Fix a reduced expression of $w$,
$$
\rxw = s_{i_1} s_{i_2} \cdots s_{i_m}.
$$
For each $k \in \{1, \dots, m\}$, set 
$$
k^+ \coloneqq \min \bigl(  \{ m+1 \} \cup \{  k+1 \leq j \leq m \mid i_j = i_k \} \bigr)
$$ 
and define index sets
$$
J \coloneqq \{ 1,2, \cdots, m \}, \, J_\mathrm{fz} \coloneqq \{ j \in J \mid j^+ = m+1 \}, \mbox{ and } J_\uf \coloneqq J \setminus J_\fr.
$$
We introduce a collection of algebraically independent variables $\{ A_{j} \mid j \in J \}$. Using the Cartan matrix $\textsf{A}$, we define an extended exchange matrix $\varepsilon_{0} \coloneqq (\varepsilon_{r,s})_{r \in J_{\uf},\, s \in J}$ by
\begin{equation}\label{equ_GLSseed}
\varepsilon_{r,s} =
\begin{cases}
-1 & \text{if } r = s^+, \\
1 & \text{if } r^+ = s, \\
-a_{i_s,i_r} & \text{if } s < r < s^+ < r^+, \\
a_{i_s,i_r} & \text{if } r < s < r^+ < s^+, \\
0 & \text{otherwise.}
\end{cases}
\end{equation}
The pair $\seed_0 \coloneqq (\{ A_{j} \mid j \in J \}, \varepsilon_{0})$ forms the initial seed. We denote by $\mathcal{A}(\seed_0)$ the cluster algebra generated by this initial seed, that is, the subalgebra of $\C( A_{j} \mid j \in J )$ generated by all the cluster variables of all the seeds mutation equivalent to $\seed_0$. Although the construction depends on the choice of reduced expression, different choices yield isomorphic cluster algebras. Accordingly, we write $\mathcal{A}(w) = \mathcal{A}(\seed_0)$ when we need to emphasize the choice $w$.

For each seed $\seed$ of the cluster algebra $\mathcal{A}(w)$, we associate to $\seed$ a Newton--Okounkov body of the Schubert variety $X_w^K$ with a choice of line bundle. Let $\lambda$ be a dominant integral weight with respect to $K$ and $\mathcal{L}^K_\lambda$ the corresponding homogeneous line bundle on $X_w^K$. This line bundle is globally generated and it is very ample if $\lambda$ is regular. To each seed $\seed$ of $\mathcal{A}(w)$, one can construct a valuation $v_\seed$ arising from a dominance order and a Newton–Okounkov body $\Delta_\seed$ of $(X_w^K, \mathcal{L}^K_\lambda)$ by following \cite{Qin17, FO25}. The resulting family of convex bodies reflects the underlying cluster structure. Accordingly, it has a tropicalized cluster structure, which we now recall. 

For an unfrozen index $k \in J_\mathrm{uf}$ and a pair $(\seed, \seed^\prime)$ of seeds with $\mu_k (\seed) = \seed^\prime$, the \emph{tropicalized cluster mutation in the $k$-th direction}
$$
\mu^{{T}}_k \colon \R^{J} \to \R^{J}
$$
is defined to be the piecewise-linear transformation given by 
\begin{equation}\label{equ_tropical11}
\mathbf{u} \coloneqq (u_j) \mapsto \mathbf{u}^\prime \coloneqq (u^\prime_j), \quad 
u_j^\prime = 
\begin{cases}
- u_j &\mbox{ if $j = k$,} \\
u_j + \left[ \varepsilon_{k,j} \right]_+ u_k  &\mbox{ if $j \neq k$ and $u_k \geq 0$,} \\
u_j + \left[ -\varepsilon_{k,j} \right]_+ u_k &\mbox{ if $j \neq k$ and $u_k \leq 0$}
\end{cases}
\end{equation}
where $[a]_+ \coloneqq \max \{a,0\}$. The family $\{\Delta_\seed \mid \mbox{$\seed$ is a seed of $\mathcal{A}(w)$}\}$ is said to have a \emph{tropicalized cluster structure} if $\Delta_{\seed^\prime} = \mu_k^T (\Delta_\seed)$ whenever $\seed^\prime$ is obtained from $\seed$ by mutation in the $k$-th direction.

In the case where $K = \emptyset$, we consider the Schubert variety $X_w \coloneqq X_w^\emptyset$ in the complete flag variety $G/B$. In \cite{FO25}, Fujita--Oya constructed a family of Newton--Okounkov polytopes of $(X_w, \mathcal{L}_\lambda)$ using valuations determined by seeds of $\mathcal{A}(w)$ and showed that this family has the tropicalized cluster structure. Building on their work, we extends their result to Schubert varieties of partial flag varieties. 

\begin{propx}[Theorem~\ref{theorem_NOpolyopteforXWK}]\label{propxA}
Let $G$ be a simply connected semisimple algebraic group over $\C$ with the index set $I$ for simple roots of its Lie algebra $\frak{g}$. For a subset $K \subseteq I$, choose $w \in W^K$ and a dominant integral weight $\lambda$ with respect to $K$. Consider the Schubert variety $X_w^K \subseteq G/P$ together with the homogeneous line bundle $\mathcal{L}^K_\lambda$. Then there exists a family 
$$
\bigl\{\Delta_\seed \mid \mbox{ $\seed$ is a seed of the cluster algebra $\mathcal{A}(w)$} \bigr\}
$$
of Newton--Okounkov bodies of $\bigl(X_w^K, \mathcal{L}^K_\lambda \bigr)$ indexed by the set of seeds such that
\begin{enumerate}
\item each underlying value semigroup is saturated and finitely generated,
\item each $\Delta_\seed$ is a rational polytope, and
\item the family $\{ \Delta_\seed \}$ has a tropicalized cluster structure.
\end{enumerate}
\end{propx}

Combining this construction with \cite{And13, HK15}, we obtain families of toric degenerations and completely integrable systems on Schubert varieties.

\begin{cor}\label{cor_B}
\begin{enumerate}
\item
If $\lambda$ is a regular, then there exists a family of toric degenerations $\{\mathcal{X}^\seed\}$ of Schubert variety $X_w^K$ indexed by the set of seeds of $\mathcal{A}(w)$ such that the central fiber $\mathcal{X}^\seed_0$ is the normal toric variety associated  with $\Delta_\seed$. 
\item
If $X_w^K$ is smooth in addition, then there exists a family of completely integrable systems $\Phi_\seed$ on $X_w^K$ indexed by the set of seeds of $\mathcal{A}(w)$ such that $\Phi_\seed \colon X_w^K \to \Delta_\seed$.
\end{enumerate}
\end{cor}

\begin{remark}
In \cite{BCMN24}, Bossinger et al. constructed a family of Newton--Okounkov bodies of $G/P$ using the cluster algebra parametrizing the $\theta$-basis in \cite{GHKK18}. In contrast, the construction of Proposition~\ref{propxA} relies on the cluster algebra parametrizing the dual canonical basis \cite{Lus90, Lus91, Kas90, Kas91, Kas93}.

In our setting, we derive and exploit certain non-negativity properties of the dual canonical basis to distinguish Newton--Okounkov polytopes. For this purpose, it is essential to work with Newton--Okounkov bodies whose integral points correspond to elements of the dual canonical basis.
\end{remark}

If a cluster algebra is of infinite type, then it admits infinitely many distinct seeds and hence gives rise to infinitely many associated Newton--Okounkov bodies, toric degenerations, and completely integrable systems of $(X^K_w, \mathcal{L}^K_\lambda)$. We consider the equivalence relation on the set of Newton--Okounkov polytopes given by integral affine transformations (a.k.a unimodular equivalences). Recall that if two polytope are related in this way, then the associated toric varieties are isomorphic. A natural question is whether these objects are pairwise distinct up to integral affine transformation.

In \cite{CKKP25}, the authors introduced a criterion to distinguish a family of Newton--Okounkov polytopes equipped with a tropicalized cluster structure under the additional assumption that each polytope is $\Q$-Gorenstein. One advantage of this criterion is that it can be applied with minimal reliance on an explicit polyhedral description. Recall that extended $g$-vectors or valuations, provide points of these polytopes. Determining an explicit polyhedral description of Newton--Okounkov bodies has been an important and challenging problem, see \cite{BZ01, RW19} for instance. In this paper, we refine this criterion so that it applies to full dimensional Newton--Okounkov polytopes having a tropicalized cluster structure, see Theorem~\ref{theorem_maincriterion}.

To apply the criterion, we assume that $G$ is of simply laced type. Let $w_K \in W^K$ be a (unique) minimal length representative of the left coset $w_0 W_{K}$ where $w_0$ is the longest element of $W$. We then have $X^K_{w_K} = G/P$.

\begin{thmx}[Theorem~\ref{theorem_distinguishGP}]\label{thmxB}
Assume that $G$ is a simply connected semisimple algebraic group over $\C$ of simply laced type. For a subset $K \subseteq I$, let $w_K$ be a minimal length representative of $w_0 W_{K}$ in $W$ and $\lambda$ a dominant integral weight with respect to $K$. Let $\Delta_\seed$ be the Newton--Okounkov polytope of $(G/P, \mathcal{L}^K_\lambda)$ associated to a seed $\seed$ of the cluster algebra $\mathcal{A}(w_K)$ in Proposition~\ref{propxA}. If $\mcal{A}(w_K)$ is of infinite type, then the family 
$$
\bigl\{\Delta_\seed \mid \mbox{ $\seed$ is a seed of the cluster algebra $\mathcal{A}(w_K)$} \bigr\}
$$ 
contains infinitely many Newton--Okounkov polytopes, no two of which are related by any integral affine transformation.
\end{thmx}

\begin{remark}
We expect that the result analogous to Theorem~\ref{thmxB} holds in the symmetrizable case as well. The main obstacle to extending our argument lies that the monoidal categorification of the quantum cluster algebra structure of a unipotent quantum coordinate ring in \cite{KKKO18, KKOP21, KKOP23} is currently available only in the symmetric case. 
\end{remark}

We now turn to applications to symplectic geometry. A growing body of work indicates that cluster algebras govern interesting symplecto-geometric objects, see \cite{Aurouxinfinitely, Via16, Via17, CG22, CG24, ABL25, ABR25, CKKP25} and many others. This paper provides a new class of monotone Lagrangian tori in partial flag varieties arising from cluster algebras.

When the line bundle $\mathcal{L}^K_\lambda$ over $G/P$ is anticanonical, one can construct a family of monotone Lagrangian tori in $G/P$ indexed by the set of seeds of the cluster algebra $\mathcal{A}(w_K)$ by using \cite[Propositions A and B]{CKKP25}. In this paper, we extend their result, which treats complete flag varieties, to the setting of partial flag varieties.

\begin{thmx}[Theorem~\ref{theorem_distinguishGPmonotone}]\label{thmxD}
Assume that a simply connected semisimple algebraic group $G$ over $\C$ is simply laced type and $\mathcal{L}^K_\lambda$ is the anticanonical line bundle on $G/P$. For each seed $\seed$ of the cluster algebra $\mathcal{A}(w_K)$, there is a monotone Lagrangian torus $L_\seed$ of $G/P$ that is the fiber over the center of $\Delta_\seed$ under the completely integrable system $\Phi_\seed$ in Corollary~\ref{cor_B}. If $\mathcal{A}(w_K)$ is of infinite type, then the family 
$$
\bigl\{L_\seed \mid \mbox{ $\seed$ is a seed of the cluster algebra $\mathcal{A}(w_K)$} \bigr\}
$$ 
contains infinitely many monotone Lagrangian tori, no two of which are related by any symplectomorphism on $G/P$.
\end{thmx}

\begin{remark}
In the case where $\mathrm{SL}_n(\C)/P$ is a Grassmannian in \cite{RW19}, plabic graphs provide combinatorial models for constructing Newton--Okounkov bodies and describing them via the mirror superpotentials. This framework can be used to produce monotone Lagrangian tori, see \cite{NU14, NU20, Cas23}. However, for a fixed $\mathrm{SL}_n(\C)/P$, there are only finitely many plabic graphs and hence this approach produces only finitely many such tori. In contrast, Proposition~\ref{propxA} produces monotone Lagrangian torus for every seed and therefore yields infinitely many tori when the associated cluster algebra is of infinite type.

In the setting of Proposition~\ref{propxA}, the valuations arise from $A$-cluster mutations, while the mirror superpotentials are related via $X$-cluster mutations. This contrasts with the situation in \cite{RW19} where the valuations are given by $X$-cluster mutations and the mirror superpotentials transform via $A$-cluster mutations. We refer to \cite[Section~7]{BCMN24} for a discussion of the relationship between them.
\end{remark}

To prove Theorems~\ref{thmxB} and~\ref{thmxD}, we apply the criterion established in Theorem~\ref{theorem_maincriterion}. Roughly speaking, this requires verifying two conditions: non-negativity and multiplicity.

The non-negativity condition asserts that there exists a frozen index $j \in J_\mathrm{fz}$ such that every Newton--Okounkov polytope $\Delta_\seed$ is contained in the half space given by $u_j \geq 0$. Since the lattice points of each $\Delta_\seed$ parametrize elements of the dual canonical basis, this condition admits the following representation-theoretic interpretation$;$ the $j$-th component of the extended $g$-vector of every dual canonical basis elements is non-negative.

To prove this statement, there are two inputs. 
The cluster algebra structure on the unipotent group in \cite{GLS11, KO21} and  the monoidal categorification of the quantum cluster algebra structure of a unipotent quantum coordinate ring and its localizations in \cite{KKKO18, KKOP21, KKOP23}. Indeed, the cluster algebra $\mathcal{A}(\seed_0)$ in~\eqref{equ_GLSseed} is isomorphic to the coordinate ring of the unipotent cell. According to \cite{GLS11}, the algebra $\mathcal{A}(\seed_0)$ is obtained by localizing the coordinate ring of the unipotent group and the corresponding cluster structures are compatible via the inversion of frozen variables.

A subtle point arises from this localization process. Additional factors appear in the denominator, and one must ensure that these do not affect the relevant components of dual canonical basis elements. To address this, we pass it to a categorical framework. Using structural properties of simple modules and determinantial modules in Kashiwara--Nakashima \cite{KN25}, we deduce the non-negativity of the extended $g$-vectors. Conceptually, this approach replaces intricate ring-theoretic computations with more tractable module-theoretic arguments.

In more concrete terms, the result can be stated at the level of Newton--Okounkov bodies, which may be of independent interest.

\begin{propx}[Corollary~\ref{cor_propertyofclustercone}]\label{thmx_E}
For $w \in W$, set 
\begin{equation}\label{equ_Iwcoloneqq}
I(w) \coloneqq \{i\in {\rm supp}(w) \, \vert \, ws_i>w\}.
\end{equation}
For each $k \in {\rm supp}(w) \setminus I(w)$, let $j_{k} \in J_\mathrm{fz}$ be the index corresponding to the frozen variable corresponding to $k$. Then for each seed $\seed$, the Newton--Okounkov polytope $\Delta_\seed$ in Theorem~\ref{thmxB} satisfies
\begin{equation}\label{equ_Iwcoloneqq22}
\Delta_\seed \subseteq \{ \mathbf{u} \in \R^J \mid u_{j_k} \geq 0 \}.
\end{equation}
\end{propx}

\begin{remark}
The mutation rule in~\eqref{equ_tropical11} is obtained by tropicalizing the $X$-cluster mutations. Under this transformation, frozen variables are affected indirectly through mutations at unfrozen variables, and hence their values may change. Nonetheless, Theorem~\ref{thmx_E} claims that the non-negativity is preserved under any iterated mutations. At present, the authors do not know how to prove Theorem~\ref{thmx_E} by purely combinatorial arguments.
\end{remark}

\begin{remark}
For the longest element $w_0 \in W$,~\eqref{equ_Iwcoloneqq22} holds for every frozen index. In contrast, for a general element $w$, not every frozen component has the non-negativity property, see Remark~\ref{remark_nonnegativeexamples}.
\end{remark}

Finally, we address the multiplicity condition. It is worth mentioning that distinct seeds may still yield equivalent Newton--Okounkov polytopes, see Remark~\ref{remark_sameobjdifseed}. Thus, even when the governing cluster algebra is of infinite type, additional input is necessary to distinguish the constructed polytopes. The key input is the existence of a sequence of seeds whose exchange quivers have arrows with arbitrarily large multiplicity emanating from a frozen variable corresponding to an index not in~\eqref{equ_Iwcoloneqq}.

Motivated by \cite{GLS08}, we introduce extended exchange matrices equipped with makings reflecting~\eqref{equ_Iwcoloneqq} together with a partial order on them. We then analyze their relationship with the weak Bruhat order and compare exchange matrices across different Weyl groups. We reduce the number of cases that we examine. We use the classification of cluster algebras of finite mutation type in \cite{FST12} to observe that $\mathcal{A}(w_K)$ is of infinite type if and only if it is of infinite mutation type. By identifying suitable subquivers of finite mutation type and determining the threshold at which the addition of an appropriate marked point yields infinite mutation type, we construct a sequence of seeds in which the multiplicities of arrows from the designated frozen variables become arbitrarily large.

\subsection*{Acknowledgement} 
The second named author would like to thank Bernard Leclerc for kindly answering our question.

\section{Partial flag varieties and Schubert varieties}\label{sec_partialSchvar}

In this section, we fix notation from Lie theory and review basic definitions and properties of partial flag varieties and their Schubert varieties.

Let $G$ be a simply connected semisimple algebraic group over $\C$ with Lie algebra $\g$. Denote by $\cmA = (a_{i,j})_{i,j\in I}$ the Cartan matrix of $\mathfrak{g}$ where $I = \{1, 2, \cdots, n = \mathrm{rank}(\mathfrak{g}) \}$ is the index set of simple roots. Let $B$ be a Borel subgroup of $G$, $H$ a maximal torus in $B$, $B^-$ the opposite Borel subgroup, and $U$ (resp. $U^-$) the unipotent radical of $B$ (resp. $B^-$).

Let $\mathfrak{h}$ be the Lie algebra of $H$, and set $\mathfrak{h}^* \coloneqq \mathrm{Hom}_\mathbb{C}(\mathfrak{h}, \mathbb{C})$ for its dual. We denote by $\langle- , -  \rangle$ the natural pairing between $\mathfrak{h}$ and $\mathfrak{h}^*$, and often write $\lambda(x) \coloneqq \langle x, \lambda \rangle$ for $x \in \mathfrak{h}$ and $\lambda \in \mathfrak{h}^*$. Let $\Phi = \Phi(G,H)$ be the root system and let $\Phi^+ = \Phi^+(B,H)$ denote the set of positive roots. The Weyl group $W$ of $\mathfrak{g}$ is given by $N(H)/H$ where $N(H)$ is the normalizer of $H$ in $G$. For each root $\alpha \in \Phi$, we denote by $U_\alpha$ the corresponding root subgroup of $(G,H)$.

For each $i \in I$, let $\alpha_i \in \mathfrak{h}^*$ be the $i$-th simple root, $s_i \in W$ the corresponding simple reflection, and $h_i \in \mathfrak{h}$ the $i$-th simple coroot, defined by $\langle h_i, \alpha_j \rangle = a_{i,j}$ for all $i,j \in I$. The $i$-th fundamental weight $\varpi_i \in \mathfrak{h}^*$ is defined by $\langle h_j, \varpi_i \rangle = \delta_{j,i}$ for all $j \in I$. The \emph{weight lattice} $\mathsf{P}$ is
$$
\mathsf{P} \coloneqq \bigoplus_{i\in I} \Z \varpi_i
$$
and the set of \emph{dominant integral weights} is
$$
\mathsf{P}^+ \coloneqq \{ \lambda \in \mathsf{P} \mid \lambda (h_i) \ge 0 \text{ for all } i\in I \}.
$$

Every Weyl group element $w \in W$ can be expressed as a product of simple reflections. The length function $\ell \colon W \to \mathbb{Z}_{\ge 0}$ is defined by the number of simple reflections of a reduced expression of $w$. For a subset $K \subseteq I$, the subgroup
$$
W_K \coloneqq \langle s_k \mid k \in K \rangle
$$ 
is called a \emph{parabolic subgroup} of $W$ associated with $K$. Define
\begin{equation}\label{equ_W^K}
W^K \coloneqq \{ w \in W \mid \ell(wv) \geq \ell(w) \mbox{ for all $v \in W_K$} \} \simeq W/W_K,
\end{equation}
so that $W^K$ is the set of minimal length representatives of cosets in the factor group $W/W_K$. We also set 
$$
\Phi_K \coloneqq \Phi \cap \bigoplus_{k \in K} \mathbb{Z} \alpha_k.
$$ 
The  \emph{parabolic subgroup} $P^K$ associated with $K$ is defined by
$$
P^K \coloneqq \langle H, U_\alpha \mid \alpha \in \Phi^+ \cup \Phi_K \rangle.
$$
The quotient $G/P^K$ is a smooth projective variety, called the \emph{partial flag variety}. If $K = \emptyset$, then $P^K$ is the Borel subgroup $B$ and hence $G/P^K = G/B$ is the complete flag variety. There is the canonical projection given by
\begin{equation}
\pi_K \colon G/B \to G/P^K, \quad gB \mapsto gP^K.
\end{equation}

Now fix a subset $K \subseteq I$. For $w \in W^K$, we choose a representative of $w$ in $N(H)$, and by abuse of notation, we denote this representative again by $w$ when the choice does not affect the discussion. The \emph{Schubert cell} is defined by
$$
C^K_w \coloneqq B{w}P^K/P^K
$$ 
and the \emph{Schubert variety} $X^K_{w}$ is its Zariski closure in $G/P^K$, which is a subvariety of the partial flag variety $G/P^K$. The Schubert variety admits the following Schubert cell decomposition
\begin{equation}\label{equ_decompositionofSchu}
X^K_w = \bigsqcup_{\substack{v \in W^K; \, v \leq w}} B v P^K / P^K
\end{equation}
where $v \leq w$ denotes the (strong) Bruhat order. For an arbitrary $w \in W$, there exists a unique decomposition 
$$
w = w^\prime w^{\prime \prime} \quad \mbox{with $w^\prime \in W^K$ and $w^{\prime \prime} \in W_K$}
$$ and we set $X^K_{w} \coloneqq X^K_{w^\prime}$. In particular, if $K = \emptyset$, then $W^K = W$, and we simply write $X_w$ instead of $X_w^\emptyset$.

\section{Newton--Okounkov polytopes of Schubert varieties of $G/B$}\label{Sec_NOp}

The goal of this section is to recall the construction of Newton--Okounkov bodies and the cluster algebra structure on unipotent cells. We then review the construction of Newton--Okounkov polytopes of Schubert varieties of the complete flag variety $G/B$ using this cluster structure and summarize their key properties in Fujita--Oya \cite{FO25}.

\subsection{Newton--Okounkov bodies}\label{sec_NObodiesreview}

Let $X$ be an irreducible normal projective variety over $\C$ of complex dimension $m$ together with a globally generated line bundle $\mathcal{L}$ over $X$. Fix a total order $\geq$ on $\Z^m$ compatible with the addition. To construct a Newton--Okounkov body, we choose a valuation $v$ and a reference section $h$ of $\mathcal{L}$. 

A \emph{valuation} $v$ \emph{with one-dimensional leaves} is a function $v \colon \C(X) \backslash \{0\} \to \Z^m$ on the function field such that for all rational functions $f, g \in \C(X) \setminus \{0\}$ and $c \in \C \setminus \{0\}$,
\begin{enumerate}
\item $v(fg) = v(f) + v(g)$,
\item $v(f+g) \geq \min (v(f), v(g))$,
\item $v(cf) = v(f)$, and
\item for all $f, g \in \C(X) \setminus \{0\} $ with $v(f) = v(g)$, there exists $c \in \C \setminus \{0\}$ such that $$
v(g - cf) > v(g) \mbox{ or } g - cf = 0.
$$ 
\end{enumerate}

Consider the section ring of $X$ given by
$$
R = \bigoplus_{k \geq 0} R_k \,\, \mbox{ where $R_0 \coloneqq \C$ and $R_k \coloneqq H^0(X, \mcal{L}^{\otimes k})$}. 
$$
Choose a reference section $h \in H^0(X, \mathcal{L})$ with $h \neq 0$ and define the value semigroup 
\begin{equation}\label{equ_SR}
S(X, \mathcal{L},v,h) = \bigcup_{k \geq 1} \bigl\{ (k, v(f/h^k)) \mid f \in R_k \backslash \{0\} \bigr\}  \subseteq \Z_{\geq 0} \times \Z^m \subseteq \R_{\geq 0} \times \R^m. 
\end{equation}
The \emph{Newton--Okounkov body} associated with the quadruple $(X, \mathcal{L}, v, h)$ is defined by 
\begin{equation}\label{equ_Deltav}
\Delta(X, \mathcal{L}, v, h) \coloneqq \overline{\operatorname{conv} \left( \bigcup_{k \geq 1} \{ x / k \mid (k,x) \in S(X, \mathcal{L},v,h) \} \right) } \subseteq \R^m.
\end{equation}
We say that a convex body $\Delta \subseteq \mathbb{R}^m$ is a \emph{Newton--Okounkov body of} $(X, \mathcal{L})$ if there exist a valuation $v$ and a nonzero section $h \in H^0(X,\mathcal{L})$ such that $\Delta = \Delta(X, \mathcal{L}, v, h)$.

\subsection{Cluster algebra structures on unipotent cells}\label{subsec_clusunicell}

In this subsection, we recall the cluster algebra structures on unipotent cells, which will be used to construct Newton--Okounkov bodies of Schubert varieties in partial flag varieties. For $w \in W$, the \emph{unipotent cell} is defined by 
$$
U_w^- = U^- \cap BwB.
$$ 
It can be regarded as an open subset of $X_w$ via an open embedding $U^-_w \hookrightarrow X_w$ induced by the natural inclusion  $U^- \hookrightarrow G/B$. 

The cluster algebra structure on the coordinate ring $\C[U^-_w]$ was constructed in \cite{BFZ05, Williams13, GLS11, GY17, FO25}.  Fix a reduced expression of $w \in W$
$$
\rxw = s_{i_1} s_{i_2} \cdots s_{i_m}.
$$
The \emph{support} of $w$ is defined as $\{i_1, i_2, \ldots, i_m\} \subseteq I$. This set is independent of the choice of reduced expression. By replacing $G$ with the subgroup corresponding to $\operatorname{supp}(w)$ if necessary, we may assume that $\operatorname{supp}(w)=I$.

For $\lambda \in \mathsf{P}^+$ and $u,v \in W$, let $\Delta_{u \lambda, v \lambda} \in \C[G]$ denote the \emph{generalized minor} associated with $u,v$, and $\lambda$. Its restriction to $U_w^-$ defines a regular function
$$
D_{u \lambda, v \lambda} \coloneqq \Delta_{u \lambda, v \lambda} |_{U^-_w} 
$$
In type $A$, the functions $\Delta_{u \varpi_j, v \varpi_j}(g)$ coincide with the usual matrix minors. 

For $k = 1, \dots, m$, set
$$
w_{\leq k} \coloneqq s_{i_1} s_{i_2} \cdots s_{i_k} \, \mbox{ and } \, k^+ \coloneqq \min \bigl( \{ m+1 \} \cup \{  k+1 \le j \le m \mid i_j = i_k \} \bigr). 
$$
Define index sets
$$
J \coloneqq \{ 1,2, \dots, m \}, \, J_\mathrm{fz} \coloneqq \{ j \in J \mid j^+ = m+1 \}, \mbox{ and } J_\uf \coloneqq J \setminus J_\mathrm{fz}.
$$
The \emph{unipotent minors} are defined by
$$
D_j \coloneqq D_{w_{\leq j} \varpi_{i_j}, \varpi_{i_j}} \qquad \text{ for $1 \leq j \leq m$}.
$$
Next, define the extended exchange matrix $\varepsilon_{t_0} \coloneqq (\varepsilon_{r,s})_{r \in J_{\uf},\, s \in J}$ by
\begin{equation}\label{equ_GLSseed}
\varepsilon_{r,s} =
\begin{cases}
-1 & \text{if } r = s^+, \\
1 & \text{if } r^+ = s, \\
-a_{i_s,i_r} & \text{if } s < r < s^+ < r^+, \\
a_{i_s,i_r} & \text{if } r < s < r^+ < s^+, \\
0 & \text{otherwise,}
\end{cases}
\end{equation}
where $(a_{i,j})_{i,j \in I}$ is the Cartan matrix $\textsf{A}$ of $G$. The principal submatrix $(\varepsilon_{r,s})_{r,s \in J_\uf}$ is skew-symmetrizable, and it is skew-symmetric if and only if $\textsf{A}$ is symmetric.  

Consider the field of rational functions
\begin{equation}\label{equ_functionfields}
\mathcal{F} = \C \left( A_{j,t_0} \mid j = 1, \dots, m \right). 
\end{equation}
A \emph{seed} $\seed = (\mathcal{A}, \varepsilon)$ consists of an algebraically independent set $\mathcal{A}$ generating $\mathcal{F}$ and an exchange matrix $\varepsilon$. We take the initial seed as
\begin{equation}\label{equ_initialseed}
\seed_0 = \left( \mathcal{A}_{\seed_0} \coloneqq \{ A_{j,t_0} \mid j \in J \}, \, \varepsilon_{t_0}  = (\varepsilon_{r,s})_{r \in J_{\uf},\, s \in J} \right)
\end{equation}
All other seeds are obtained from $\seed_0$ by iterated mutations. For $k \in J_{\mathrm{uf}}$, the mutation $\seed' = \mu_k(\seed)$ is defined as follows. Set $[a]_+ \coloneqq \max \{a,0\}$. The cluster variables mutate by the rule
\begin{equation}\label{equ_defatjvar}
\mu_k \colon \mcal{A}_\seed \dashrightarrow \mcal{A}_{\seed'}, \quad
\mu_k^*(A^\prime_j) =
\begin{cases}
A_k^{-1} \left( \prod_{i \in J} A_i^{[\varepsilon_{k,i}]_+} + \prod_{i \in J} A_i^{[-\varepsilon_{k,i}]_+} \right) & \text{if } j = k, \\
A_j & \text{if } j \neq k,
\end{cases}
\end{equation}
and the exchange matrix mutates by the rule 
\begin{equation}\label{equ_mutationexchangematri}
\varepsilon^\prime = (\varepsilon_{r,s}^\prime)_{r,s \in J}, \qquad
\varepsilon_{r,s}^\prime =
\begin{cases}
- \varepsilon_{r,s} & \text{if } r = k \text{ or } s = k, \\
\varepsilon_{r,s} + \mathrm{sgn}(\varepsilon_{k,s}) [\varepsilon_{r,k}\varepsilon_{k,s}]_+ & \text{otherwise.}
\end{cases}
\end{equation}

Let $\mathbb{T}$ be the exchange graph with initial vertex $t_0$ corresponding to $\seed_0$. For each $t \in \mathbb{T}$, let $(A_{j,t})_{j \in J}$ denote the corresponding cluster variables defined by the rule~\eqref{equ_defatjvar}. We define the cluster algebras associated with $\seed_0$ as follows$\colon$

\begin{definition}\label{definition_clusteralg}
Let $\mathcal{F}$ be given as in~\eqref{equ_functionfields}.
\begin{itemize}
\item The (\emph{ordinary}) \emph{cluster algebra} is the subalgebra $\mathcal{A}(\seed_0)$ of $\mathcal{F}$ generated by
$$
\left\{A_{j,t} \mid t \in \mathbb{T}, j \in J_{\mathrm{uf}} \right\} \cup \left\{ A_{j,t}^\pm \mid t \in \mathbb{T}, j \in J_\mathrm{fz} \right\}.
$$
\item The \emph{upper cluster algebra} is the subalgebra of $\mathcal{F}$ given by
$$
\mathcal{U}(\seed_0) = \bigcap_{t \in \mathbb{T}} \C \bigl[ A_{j,t}^{\pm} \mid j \in J \bigr].
$$
\end{itemize}
\end{definition}

\begin{remark}\label{remark_notationAw}
Different choices of reduced expression of $w$ yield isomorphic cluster algebras. Accordingly, we write $\mathcal{A}(w)$ in place of $\mathcal{A}(\seed_0)$ when we want to emphasize the dependence on $w$.
\end{remark}

\begin{theorem}[Appendix B in \cite{FO25}]\label{theorem_clustunitcell}
The ordinary cluster algebra agrees with the upper cluster algebra, that is, 
$$
\mathcal{A}(\seed_0) = \mathcal{U}(\seed_0)
$$
Moreover, there is a $\C$-algebra isomorphism 
$$
\mathcal{A}(\seed_0) = \mathcal{U}(\seed_0)  \to \C[U_w^-], \quad A_{j,t_0} \mapsto D_j.
$$
In this sense, $\C[U_w^-]$ is endowed with the cluster algebra structure.
\end{theorem}

\subsection{Newton--Okounkov polytopes of Schubert varieties of $G/B$}

By utilizing the above cluster algebra structure, we associate to each seed $\seed$ a valuation $v_\seed$ on the function field $\C(U^-_w)$, which is isomorphic to $\C(X_w)$. Let $ \seed = \left((A_{j,t})_{j \in J}, \vep_t \right)$ be the seed of $\C[ U^-_w]$ associated with a vertex $t \in \mathbb{T}$. For $\bfa, \bfb \in \Z^{J}$, we define 
$$
\bfa \preceq_{\vep_t} \bfb \quad \Longleftrightarrow \quad \bfa = \bfb + \textbf{v} \vep_t \text{ for some $\textbf{v} \in \Z_{\ge0}^{J_\uf}$.}
$$
This order $ \preceq_{\vep_t}$ on $\Z^J$ is called the \emph{dominance order} associated with $\vep_t$ in \cite{Qin17}. 

Consider the Laurent polynomial ring $\mathcal{F} \coloneqq \C[ A_{j,t}^{\pm 1} \mid j \in J ]$. Let $<_t$ be a total order refining the opposite dominance order $\preceq^{\mathrm{op}}_{\varepsilon_t}$ associated with the exchange matrix $\varepsilon_t$ corresponding to $t$. By identifying a Laurent monomial $ \prod_{j \in J} A_{j,t}^{a_j}$ with its exponent vector $\bfa = (a_j)_{j \in J} \in \Z^{J}$, the order $<_t$ induces an order on the set of Laurent monomials in $\mathcal{F}$. We denote by $ v_{\seed}$ the lowest term valuation on $\mathcal{F}$ corresponding to $<_t$. Namely, for $f/g \in \mathcal{F}$ with $f,g \in \C[A_{j,t} \mid j \in J]$, we define
$$
v_{\seed}(f/g) \coloneqq (a_j)_{j \in J} - (b_j)_{j \in J} \in \Z^{J}
$$
where 
$$
f = \prod_{j=1}^{m} c A_{j,t}^{a_j} + (\mbox{terms of higher order}), \,\, g = \prod_{j=1}^{m} c^\prime A_{j,t}^{b_j} + (\mbox{terms of higher order})
$$
for some $c, c^\prime \in \mathbb{C} \setminus \{0\}$.
We write $v_t$ for $v_{\seed}$ if no confusion arises.

For a certain class of elements in $\mathcal{F}$, the valuation can be calculated by the extended $g$-vector, which we now recall. Following \cite{FominZelevinsky4, FockGoncharov}, we set
$$
X_{i,t} \coloneqq \prod_{j \in J} A_{j,t}^{\varepsilon_{i,j}}.
$$
An element $f \in \mathcal{F}$ is said to be \emph{weakly pointed} at $(g_j)_{j \in J} \in \Z^J$ if $f$ can be expressed as
\begin{equation}\label{equ_extgvector}
f = \left( \prod_{j \in J} A^{g_j}_{j,t} \right) \left( \sum_{\mathbf{a} = (a_{j}) \in \mathbb{Z}^{J_\mathrm{uf}}_{\geq 0} } c_\mathbf{a} \prod_{j \in J_\mathrm{uf}} X_{j,t}^{a_j} \right)
\end{equation}
for some coefficients $c_{\mathbf{a}} \in \mathbb{C}$ with $c_{\mathbf{0}} \neq 0$. In this case, 
$$
g_{\seed}(f) \coloneqq (g_j)_{j \in J}
$$ 
is called the \emph{extended g-vector} of $f$. We sometimes write $g_t$ in place of $g_{\seed}$. By \cite[Proposition 3.9]{FO25}, for every weakly pointed element $f \in \mathcal{F}$, we have
\begin{equation}\label{equ_val=g}
v_\seed(f) = g_\seed(f).    
\end{equation}

For a dominant integral weight $\lambda \in \mathsf{P}^+$, we define a line bundle $\Lb_\lambda \coloneqq (G \times \C)/B$ over the flag manifold $G/B$, where $B$ acts on $G\times \C$ from the right as $ (g,c) \cdot b \coloneqq (gb, \chi_\lambda(b) c)$ for $g\in G, c\in \C$, and $b\in B$. Restricting to $X_w$, we obtain the line bundle on $X_w$ which is also denoted by $\Lb_\lambda$. If $\lambda$ is regular in addition, that is, $\langle h_i, \lambda \rangle >0$ for every $i\in I$, then the line bundle $\mcal{L}_\lambda$ is very ample. We fix a lowest weight vector $\tau_\lambda$ in $H^0(G/B, \Lb_\lambda)$ and restrict it to $X_w$. Following Section~\ref{sec_NObodiesreview} and using the valuation $v_\seed$ on the function field $\mathcal{F} \simeq \C(X_w)$, we produce the semigroup $S(X_w, \mcal{L}_\lambda,v_\seed,\tau_\lambda)$ defined in \eqref{equ_SR} and the Newton--Okounkov body $\Delta(X_w, \mcal{L}_\lambda,v_\seed,\tau_\lambda)$ defined by \eqref{equ_Deltav}. In sum, we have the quadruple
\begin{equation}\label{equ_quadrupleonxw}
(X_w,\mcal{L}_\lambda,v_\seed,\tau_\lambda)
\end{equation}
By following the procedure in Section~\ref{Sec_NOp}, we obtain  the family of semigroups and Newton--Okounkov bodies 
\begin{equation}\label{equ_familyofseminewto}
\left\{ \left( S(X_w, \mcal{L}_\lambda, v_{\seed}, \tau_\lambda), \Delta(X_w, \mcal{L}_\lambda, v_{\seed}, \tau_\lambda) \right) \mid \mbox{$\seed$ is a seed} \right\}
\end{equation}

Let $C_{\seed}(w) = C_t(w)$ be the smallest closed convex cone containing $v_\seed ( \C[ U^- \cap X_w] \setminus \{ 0 \} )$ in $\R^{J}$, which is called the \emph{cluster cone} of $X_w$ associated with the seed $\seed$. 

\begin{definition}\label{definition_tropicalclusterstr}
Consider the family $\{ S_t \subseteq \R_{\geq 0} \times \R^J \mid t \in \mathbb{T} \}$ of semigroups and the family $\{ \Delta_t \subseteq \R^J \mid t \in \mathbb{T} \}$ of associated Newton--Okounkov bodies, which are parametrized by $\mathbb{T}$. 

\begin{enumerate}
\item For an unfrozen index $k \in J_\mathrm{uf}$ and a pair $(t, t^\prime)$ of vertices with $\mu_k (t) = t^\prime$ in  $\mathbb{T}$, a \emph{tropicalized cluster mutation in the $k$-th direction}
$$
\mu^{{T}}_k \colon \R^{J} \to \R^{J}
$$
is defined to be the piecewise-linear transformation defined by 
\begin{equation}\label{equ_tropical1}
\mathbf{u} \coloneqq (u_j) \mapsto \mathbf{u}^\prime \coloneqq (u^\prime_j), \quad 
u_j^\prime = 
\begin{cases}
- u_j &\mbox{ if $j = k$} \\
u_j + \left[ \varepsilon_{k,j} \right]_+ u_k  &\mbox{ if $j \neq k$ and $u_k \geq 0$} \\
u_j + \left[ -\varepsilon_{k,j} \right]_+ u_k &\mbox{ if $j \neq k$ and $u_k \leq 0$}
\end{cases}
\end{equation}
where ${\varepsilon}_t = (\varepsilon_{i,j})_{i \in J_{uf}, j \in J}$ is the exchange matrix associated with $t$. 
\item We also define 
\begin{equation}\label{equ_hattropical1}
\widehat{\mu}^{{T}}_k \colon \R_{\geq 0} \times \R^{J} \to \R_{\geq 0} \times \R^{J}
\end{equation}
by $\widehat{\mu}^{{T}}_k (r, \mathbf{u}) \coloneqq (r, \mu^{{T}}_k (\mathbf{u}))$. This map $\widehat{\mu}^{{T}}_k$ is also called a \emph{tropicalized cluster mutation in the $k$-th direction}.

\item The family of semigroups parametrized by $\mathbb{T}$ is said to have a \emph{tropical cluster structure} if for every pair $(t, t^\prime)$ of vertices with $t^\prime = \mu_k (t)$, the associated semigroups $S_t$ and $S_{t^\prime}$ are related by the tropicalized cluster mutation $\widehat{\mu}^T_k$ in the $k$-th direction, that is, $S_{t^\prime} = \widehat{\mu}^T_k (S_t)$.

\item The family of Newton--Okounkov bodies parametrized by $\mathbb{T}$ is said to have a \emph{tropical cluster structure} if for every pair $(t, t^\prime)$ of vertices with $t^\prime = \mu_k (t)$, the associated Newton--Okounkov bodies $\Delta_t$ and $\Delta_{t^\prime}$ are related by the tropicalized cluster mutation $\mu^T_k$ in the $k$-thdirection, that is, $\Delta_{t^\prime} = \mu^T_k (\Delta_t)$.
\end{enumerate}
\end{definition}

\begin{theorem}[Corollary 5.8, Theorem 6.8, Corollary 6.9, Corollary 6.10 in \cite{FO25}]\label{Thm: FO1}
The family in~\eqref{equ_familyofseminewto} satisfies the following$\colon$
\begin{enumerate}
\item each semigroup in~\eqref{equ_familyofseminewto} is finitely generated and saturated. 
\item each Newton--Okounkov body in~\eqref{equ_familyofseminewto} is a rational polytope whose integral points coincide with the extended g-vectors of elements of the dual canonical basis,
\item the family of semigroups in~\eqref{equ_familyofseminewto} has a tropicalized cluster structure,
\item the family of Newton--Okounkov bodies in~\eqref{equ_familyofseminewto} has a tropicalized cluster structure,
\item the cluster cone $C_{t}(w) = \bigcup_{\lambda \in \mathsf{P}^+} \Delta(X_w,\mcal{L}_\lambda,v_t,\tau_\lambda)$.
\end{enumerate}	
\end{theorem}	

Combining this with \cite{And13}, we obtain the following corollary.

\begin{corollary}
If $\lambda$ is a regular dominant integral weight, then for each seed $\seed$, there exists a toric degeneration of $X_w$ into the normal toric variety associated with $\Delta(X_w,\mcal{L}_\lambda,v_\seed,\tau_\lambda)$.
\end{corollary}

\section{Newton--Okounkov polytopes of Schubert varieties in $G/P$}\label{sec_NOpolyoptesforpartial}

The aim of this section is to construct Newton–Okounkov polytopes of Schubert varieties in partial flag varieties $G/P$ using the cluster algebra structure on unipotent cells, and to investigate their applications. We also state the main result, the existence of infinitely many pairwise non-equivalent Newton--Okounkov bodies of $G/P$ constructed from cluster algebras of infinite type.

\subsection{Construction of Newton--Okounkov polytopes of Schubert varieties in $G/P$}\label{subsection_constNO}

We begin with a general setup. Let $X$ and $Y$ be irreducible, normal, projective varieties of dimension $m$ over $\C$, and suppose that 
$$
\varphi \colon X \to Y
$$
is a birational morphism. Assume that $X$ is equipped with a valuation 
$$
v_X \colon \C(X) \setminus \{0\} \to \Z^m
$$ 
with one-dimensional leaves, with respect to a fixed total order on the lattice $\Z^m$. Since the induced pullback on function fields $\varphi^* \colon \C(Y) \xrightarrow{\sim} \C(X)$ is an isomorphism, we obtain a valuation with one-dimensional leaves on $Y$ by
\begin{equation}\label{equ_vY}
v_Y \coloneqq v_X \circ \varphi^* \colon \C(Y) \setminus \{0\}  \to \C(X) \setminus \{0\}  \to \Z^m.
\end{equation}

Let $\mcal{L}_Y$ be a basepoint-free line bundle on $Y$ and set 
$$
\mcal{L}_X \coloneqq \varphi^* \mcal{L}_Y,
$$  
which is also basepoint-free. Fix a reference section $\tau_Y \in H^0(Y, \mcal{L}_Y)$, and define 
$$
\tau_X \coloneqq \varphi^* \tau_Y = \tau_Y \circ \varphi \in H^0(X, \mcal{L}_X)
$$ 
as the reference section for $\mcal{L}_X$. 

In this way, we obtain two quadruples 
$$
(X, \mcal{L}_X,  v_X, \tau_X) \mbox{ and } (Y, \mcal{L}_Y,  v_Y, \tau_Y),
$$
which are input data for defining Newton--Okounkov bodies. The following proposition compares the Newton--Okounkov bodies associated with these two quadruples.

\begin{lemma}\label{proposition_NewtonOkounkovbodycoincide}
Let $X$ and $Y$ be irreducible, normal, projective varieties of dimension $m$ over $\C$, and let $\varphi \colon X \to Y$ be a birational morphism. Fix a total order on $\Z^m$ and a valuation $v_X \colon \C(X) \to \Z^m$ with one-dimensional leaves. Let $\mcal{L}_Y$ be a basepoint-free line bundle on $Y$ with a reference section $\tau_Y \in H^0(Y, \mcal{L}_Y)$ and define
$$
\mathcal{L}_X \coloneqq \varphi^* \mathcal{L}_Y, \tau_X \coloneqq \varphi^* \tau_Y, \mbox{ and } v_Y \coloneqq v_X \circ \varphi
$$
Then the corresponding Newton--Okounkov bodies coincide, i.e., 
\begin{equation}\label{equ_twoNewtonOkounkovbodies}
\Delta(X, \mcal{L}_X,  v_X, \tau_X) = \Delta(Y, \mcal{L}_Y,  v_Y, \tau_Y).
\end{equation}
\end{lemma}

\begin{proof}
We claim that if $\psi \colon X \to Y$ is a projective birational morphism and $Y$ is normal, then for any bundle $\mathcal{L}$ of $Y$ and $m \in \mathbb{N}$, we have an isomorphism
$$
\psi^* \colon H^0(Y, \mathcal{L}^{\otimes m}) \to H^0(Y, \psi^* \mathcal{L}^{\otimes m}), \quad f \mapsto f \circ \psi.
$$
Recall that if $\psi : X \to Y$ is a projective birational morphism and $Y$ is normal, one has that
\[
    \psi_*\mathcal O_X \cong \mathcal O_Y
\]
by Zariski's main theorem \cite[Corollary III.11.4]{Hartshorne}. Hence, by the projection formula, 
\[
    \psi_*\psi^*{\mcal{L}} \cong {\mcal{L}} \otimes_{\mcal{O}_Y} \psi_*\mathcal O_X \cong {\mcal{L}}.
\]
Moreover, taking global sections of $\mcal{L}^{\otimes m}$, we obtain
a desired isomorphism 
\[
    H^0(X,\psi^*{\mcal{L}}^{\otimes m}) \cong H^0(Y,{\mcal{L}}^{\otimes m}).
\]
Therefore, we have a graded ring isomorphism between section rings
\[
    R(Y,{\mcal{L}}) \coloneqq \bigoplus_{m\ge0} H^0(Y,{\mcal{L}}^{\otimes m}) \cong \bigoplus_{m\ge0} H^0(X,\psi^*{\mcal{L}}^{\otimes m})
    \eqqcolon R(X,\psi^*{\mcal{L}}).
\]

Now, we return to our situation. Since $\varphi \colon X \rightarrow Y$ is a projective birational morphism and $Y$ is normal, we have a graded ring isomorphism
\[
    \varphi^* \colon R(Y,\mcal{L}_Y) \rightarrow R(X,\mcal{L}_X), \quad \varphi^*(f) := f \circ \varphi.
\]
For each $f \in H^0(Y, \mcal{L}_Y^{\otimes k})$, by definition of $v_Y$, we have
\begin{equation}\label{equ_valuat}
v_Y\!\left(\frac{f}{\tau_Y^{k}}\right)
= (v_X \circ \varphi^*)\!\left(\frac{f}{\tau_Y^{k}}\right)
= v_X\!\left(\frac{f}{\tau_Y^{k}} \circ \varphi\right)
= v_X\!\left(\frac{\varphi^*f}{\tau_X^{k}}\right).
\end{equation}
Hence, it yields that
$$
S(Y, \mcal{L}_Y, v_Y, \tau_Y) \subseteq S(X, \mcal{L}_X, v_X, \tau_X).
$$
Conversely, since $\varphi^*$ is an isomorphism, each $g \in H^0(X, \mcal{L}_X^{\otimes k})$ can be written as $g = \varphi^* f$ for some $f \in H^0(Y, \mcal{L}_Y^{\otimes k})$. Applying~\eqref{equ_valuat} yields the opposite inclusion, and hence the two semigroups coincide, and~\eqref{equ_twoNewtonOkounkovbodies} follows.
\end{proof}

We now return to the setting of Section~\ref{sec_partialSchvar}. Using Lemma~\ref{proposition_NewtonOkounkovbodycoincide}, we construct Newton--Okounkov bodies of Schubert varieties in $G/P$. Fix a subset $K \subseteq I$ and consider the associated parabolic subgroup $P^K$. For brevity, we write $P \coloneqq P^K$. Let 
$$
\pi \colon G/B \to G/P
$$ 
be the natural projection. For each $v \in W^K$, the morphism $\pi$ maps the Schubert cell $BvB/B$ onto $BvP/P$. For each $w \in W^K$, we obtain two Schubert varieties$\colon$ 
\begin{itemize}
\item $X_w \subseteq G/B$, consisting of Schubert cells indexed by $v \in W$ with $v \leq w$, and
\item $X_w^K \subseteq G/P$, consisting of Schubert cells indexed by $v \in W^K$ with $v \leq w$
\end{itemize}
by the decomposition~\eqref{equ_decompositionofSchu}. Thus, the restriction of $\pi$ to $X_w$ yields a surjective morphism 
$$
\pi |_{X_w} \colon X_w (\subseteq G/B) \to X_w^K (\subseteq G/P).
$$ 
Moreover, since $w \in W^K$ is the minimal length representative in its coset, the restriction of $\pi$ to the Schubert cell $BwB/B$ is injective. Since $BwB$ is dense in $X_w$, the morphism $\pi |_{X_w}$ is birational. Let $\iota \colon X_w \to G/B$ denote the closed embedding, and define
\begin{equation}\label{equ_varphidef}
\varphi \coloneqq \pi \circ \iota \colon X_w \to G/B \to G/P. 
\end{equation}
We summarize the above discussion as follows.

\begin{proposition}[Lemma 11.3.3 in \cite{Kum02}]\label{prop_morphismsujbiration}
For $w \in W^K$, the morphism $\varphi \colon X_w  \to X_w^K$ 
is surjective and birational.
\end{proposition}

A weight $\lambda \in \mathsf{P}^+$ is called a \emph{dominant integral weight} (resp. \emph{regular dominant integral weight}) \emph{with respect to} $K$ if it admits an expression of the form
$$
\lambda = \sum_{i \in I \setminus K} \lambda_i \varpi_i, \quad \lambda_i \in \Z_{\geq 0}  \quad (\mbox{resp. }  \lambda_i \in \Z_{> 0}). 
$$

Fix a dominant integral weight $\lambda \in \mathsf{P}^+$ with respect to $K$. It determines a character $\chi_\lambda$ of $P$ and the associated homogeneous line bundle
$$
\mcal{L}^K_\lambda = G \times_{P} \C
$$
where $P$ acts on $G \times \C$ by $(g,z) \cdot p = (gp, \chi_\lambda({p})z)$ for $g \in G, z \in \C$, and $p \in P$. In particular, when $K = \emptyset$, we write $\mcal{L}_\lambda \coloneqq \mcal{L}_\lambda^\emptyset$ for the line bundle on $G/B$. This line bundle is basepoint-free. Moreover, $\mcal{L}_\lambda$ is isomorphic to the pullback $\pi^* \mcal{L}_\lambda^K$ on $G/B$ and the induced bundle $\iota^* \mcal{L}_\lambda$ is isomorphic to $\varphi^* \mcal{L}_\lambda^K$ on $X_w$. 

For $K \subseteq I$ and $w \in W^K$, consider the cluster algebra structure $\mcal{A}(w)$ on $\C[U^-_w]$ given in Theorem~\ref{theorem_clustunitcell}. For each seed $\seed$ of $\mcal{A}(w)$, we obtain the quadruple $(X_w, \mcal{L}_\lambda, v_\seed, \tau_\lambda)$ as in~\eqref{equ_quadrupleonxw}, which defines a Newton--Okounkov body for $X_w \subseteq G/B$. By Proposition~\ref{prop_morphismsujbiration}, there exists a birational morphism
$$
\varphi \colon X_w (\subseteq G/B) \to X_w^K (\subseteq G/P).
$$
Choose a reference section $\tau^K_\lambda \in H^0(X_w^K, \mcal{L}^K_\lambda)$ such that $\varphi^* \tau^K_\lambda = \tau_\lambda$, and define $v_\seed^K \coloneqq v_\seed \circ \varphi^*$. This construction yields the quadruple $(X_w^K, \mcal{L}^K_\lambda, v^K_\seed, \tau^K_\lambda)$, which defines a Newton--Okounkov body for $X_w^K \subseteq G/P$. 

The next theorem compares the Newton--Okounkov bodies of $X_w^K$ and $X_w$. 

\begin{theorem}[Proposition~\ref{propxA}]\label{theorem_NOpolyopteforXWK}
Let $G$ be a simply connected semisimple algebraic group over $\C$ with the index set $I$ for simple roots of its Lie algebra $\frak{g}$. For a subset $K \subseteq I$, let $\lambda$ be a dominant integral weight with respect to $K$. For each $w \in W^K$ and each seed $\seed$ of the cluster algebra $\mcal{A}(w)$ in Definition~\ref{definition_clusteralg}, we have 
$$
S(X_w^K, \mcal{L}^K_\lambda, v^K_\seed, \tau^K_\lambda) =  S(X_w, \mcal{L}_\lambda, v_\seed, \tau_\lambda),
$$
where the right-hand side is the semigroup from the quadruple~\eqref{equ_quadrupleonxw}.
In particular, the Schubert variety $X_w^K$ admits a Newton--Okounkov body $\Delta(X_w^K, \mcal{L}^K_\lambda, v^K_\seed, \tau^K_\lambda)$ satisfying 
$$
\Delta(X_w^K, \mcal{L}^K_\lambda, v^K_\seed, \tau^K_\lambda) =  \Delta(X_w, \mcal{L}_\lambda, v_\seed, \tau_\lambda).
$$
\end{theorem}

\begin{proof}
By Proposition~\ref{prop_morphismsujbiration}, the map $\varphi$ is a birational morphism. Recall that every Schubert variety is a normal variety, see \cite[Theorem 2.1.1]{Bri04} for instance. The assertion then follows from Lemma~\ref{proposition_NewtonOkounkovbodycoincide}.
\end{proof}

For each seed $\seed$, we obtain a semigroup $S(X_w^K, \mcal{L}^K_\lambda, v^K_{\seed}, \tau^K_\lambda)$ and a Newton--Okounkov body $\Delta(X_w^K, \mcal{L}^K_\lambda, v^K_{\seed}, \tau^K_\lambda)$. By Theorem~\ref{theorem_NOpolyopteforXWK}, the family of semigroups and Newton--Okounkov bodies 
\begin{equation}\label{equ_familyofnewtonxwk}
\left\{ \left( S(X_w^K, \mcal{L}^K_\lambda, v^K_{\seed}, \tau^K_\lambda), \Delta(X_w^K, \mcal{L}^K_\lambda, v^K_{\seed}, \tau^K_\lambda) \right) \mid \mbox{$\seed$ is a seed} \right\}
\end{equation}
inherits all structural properties from the corresponding family $$
\left\{ \left( S(X_w, \mcal{L}_\lambda, v_\seed, \tau_\lambda), \Delta(X_w, \mcal{L}_\lambda, v_\seed, \tau_\lambda) \right) \mid \mbox{$\seed$ is a seed}  \right\}.
$$

\begin{corollary} \label{cor_theoremAstruct}
The family in~\eqref{equ_familyofnewtonxwk} satisfies the following$\colon$
\begin{enumerate}
\item each semigroup in~\eqref{equ_familyofnewtonxwk} is finitely generated and saturated,
\item each Newton--Okounkov body in~\eqref{equ_familyofnewtonxwk} is a rational polytope whose integral points coincide with the extended $g$-vectors of elements of the dual canonical basis,
\item the family of semigroups in~\eqref{equ_familyofnewtonxwk} has a tropicalized cluster structure,
\item the family of Newton--Okounkov polytopes in~\eqref{equ_familyofnewtonxwk} has a tropicalized cluster structure,
\end{enumerate}
\end{corollary}

\begin{remark}\label{def_clusteralgforSchpartialflag}
For $w \in W^K$, the coordinate ring $\C[U^-_{w}]$ carries a cluster algebra structure that is isomorphic to $\mcal{A}(\seed_0)$ defined in~\eqref{equ_initialseed}. By Theorem~\ref{theorem_NOpolyopteforXWK}, we may denote by $\mcal{A}(w)$ this cluster algebra which we use to construct the Newton--Okounkov bodies of both $X_w^K \subseteq G/P$ and $X_w \subseteq G/B$, cf. Remark~\ref{remark_notationAw}. 
\end{remark}

We now discuss several applications of Theorem~\ref{theorem_NOpolyopteforXWK} and Corollary~\ref{cor_theoremAstruct}. Assume in addition that the line bundle $\mcal{L}^K_\lambda$ is very ample (equivalently, $\lambda$ is a regular dominant integral weight with respect to $K$). For each seed $\seed$, there exists a flat degeneration of $X_w^K$ to a toric variety whose normalization is the toric variety associated with $\Delta(X_w^K, \mcal{L}^K_\lambda, v^K_{\seed}, \tau^K_\lambda)$ by \cite{And13}. Since the semigroup $S(X_w, \mcal{L}_\lambda, v_\seed, \tau_\lambda)$ is saturated, it follows that $S(X_w^K, \mcal{L}^K_\lambda, v^K_{\seed}, \tau^K_\lambda)$ is also saturated. Hence, the central fiber of this flat degeneration is a normal toric variety. Moreover, by \cite{HK15}, there exists a completely integrable system 
$$
\Phi_{\seed} \colon X_w^K \to \R^m
$$ 
whose image is  $\Delta(X_w^K, \mcal{L}^K_\lambda, v^K_{\seed}, \tau^K_\lambda)$, where $m = \dim_\C X_w^K$.

\begin{corollary}\label{cor_constructiontoric}
Assume in addition that $\mcal{L}^K_\lambda$ is a very ample line bundle over $X_w^K$. For each seed $\seed$ of the cluster algebra $\mcal{A}(w)$ in Theorem~\ref{theorem_clustunitcell}, there exist
\begin{enumerate}
\item a flat degeneration of $X_w^K$ into the normal toric variety associated with the Newton--Okounkov polytope $\Delta(X_w^K, \mcal{L}^K_\lambda, v^K_{\seed}, \tau^K_\lambda)$ and
\item a completely integrable system $\Phi_\seed \colon X_w^K \to \R^{m}$ whose image is $\Delta(X_w^K, \mcal{L}^K_\lambda, v^K_{\seed}, \tau^K_\lambda)$.
\end{enumerate}
\end{corollary}

We now specialize to the case where $X_w^K = G/P$ by choosing the minimal length representative $w_K$ of the coset $w_0 W_K$ in $W$ where $w_0$ is the longest element of $W$. Let $\lambda$ be chosen so that $\mathcal{L}^K_\lambda$ is the anticanonical line bundle, that is, a positive integer multiple of the anticanonical weight 
$$
\lambda = \sum_{i \in I \setminus K} \varpi_i.
$$ 
Then the K\"{a}hler form induced by the projective embedding associated with $\mcal{L}^K_\lambda$ is \emph{monotone}, which means that the first Chern class is positively proportional to the K\"{a}hler form. By \cite[Proposition A]{CKKP25}, the associated completely integrable system carries a monotone Lagrangian torus located at the ``center" of $\Delta(G/P, \mcal{L}^K_\lambda, v^K_{\seed}, \tau^K_\lambda)$. Here, a Lagrangian submanifold $L$ is called \emph{monotone} if for every relative homotopy class $\beta \in \pi_2(X,L)$, the Maslov index of $\beta$ is positively proportional to its symplectic area, see \cite{Oh93, BC09}.

Recall that these Newton--Okounkov polytopes are in one-to-one correspondence with the seeds of the cluster algebra structure on $\C[U^-_{w_K}]$. For later use, we introduce the following notation for the cluster algebra $\mcal{A}(\seed_0)$.

\begin{definition}\label{def_clusteralgforpartialflag}
For a parabolic subgroup $P$, the coordinate ring $\C[U^-_{w_K}]$ carries a cluster algebra structure isomorphic to $\mcal{A}(w_K)$ defined in~\eqref{equ_initialseed}. We denote this cluster algebra by $\mcal{A}_P$. 
\end{definition}

\begin{corollary}\label{cor_monotoneLagtoriconst}
Assume that $\mcal{L}^K_\lambda$ is an anticanonical line bundle on $G/P$. Then, for each seed $\seed$ of the cluster algebra $\mcal{A}_P$, the completely integrable system $\Phi_\seed \colon G/P \to \R^{m}$ admits a monotone Lagrangian torus $L_\seed$ in $G/P$ located at the center of $\Delta(G/P, \mcal{L}^K_\lambda, v^K_{\seed}, \tau^K_\lambda)$.
\end{corollary}

\subsection{Comparison of Newton--Okounkov polytopes of Schubert varieties in $G/P$}

The goal of this subsection is to compare the Newton--Okounkov polytopes, toric degenerations, and monotone Lagrangian tori constructed in Section~\ref{subsection_constNO} in the simply laced case. If the associated cluster algebra contains infinitely many distinct seeds, then each of these families also contains infinitely many pairwise distinct members.

We first introduce an equivalence relation on Newton--Okounkov polytopes.

\begin{definition}
An \emph{integral affine transformation} is a map $\phi \colon \R^m \to \R^m$ of the form
$$
\phi (\mathbf{u} ) = A \mathbf{u} + \mathbf{u}_0 \,\, \mbox{ where $A \in \mathrm{GL}(m, \Z)$ and $\mathbf{u}_0 \in \Z^m$}.
$$

Let $\{ \Delta_\seed \}$ be a family of polytopes in $\R^m$. We say that
$\Delta_\seed$ and $\Delta_{\seed^\prime}$ are \emph{equivalent} and write $\Delta_\seed \sim \Delta_{\seed^\prime}$ if there exists an integral affine transformation $\phi$ such that $\Delta_{\seed^\prime} = \phi (\Delta_\seed)$.
\end{definition}

We also introduce an equivalence relation on Lagrangian submanifolds.

\begin{definition}
Let $\{ L_\seed \}$ be a family of Lagrangian submanifolds in a symplectic manifold $(X, \omega)$. We say that $L_\seed$ and $L_{\seed^\prime}$ are \emph{equivalent} and write $L_\seed \sim L_{\seed^\prime}$ if there exists a symplectomorphism $\phi$ on $X$ such that $L_{\seed^\prime} = \phi (L_\seed)$. \end{definition}

Recall that a cluster algebra $\mcal{A}$ is said to be of \emph{infinite type} if it has infinitely many distinct seeds, see Definition~\ref{def_clufininfin}. We now state the main results. 

\begin{theorem}[Theorem~\ref{thmxB}]\label{theorem_distinguishGP}
Let $G$ be a simply connected semisimple algebraic group of simply laced type over $\C$. For a subset $K \subseteq I$, let $P$ be the parabolic subgroup corresponding to $K$ and let $\lambda$ be a regular dominant integral weight with respect to $K$ so that the line bundle $\mcal{L}^K_\lambda$ over $G/P$ is very ample. For each seed $\seed$ of the cluster algebra $\mcal{A}_{P}$, we set 
$$
\Delta_\seed \coloneqq \Delta(X_w^K, \mcal{L}^K_\lambda, v^K_\seed, \tau^K_\lambda),  
$$
the Newton--Okounkov polytope constructed in Theorem~\ref{theorem_NOpolyopteforXWK}. If $\mcal{A}_P$ is of infinite type, then the set of equivalence classes 
\begin{equation}\label{equ_NOpolyseed}
\bigl\{ [\Delta_\seed] \mid\mbox{$\seed$ is a seed\,} \bigr\}
\end{equation}
is infinite.
\end{theorem}

By Corollary~\ref{cor_constructiontoric}, we obtain a family of toric degenerations of $G/P$. In particular, Theorem~\ref{theorem_distinguishGP} implies that this family contains infinitely many pairwise distinct toric limits.

\begin{theorem}[Theorem~\ref{thmxD}]\label{theorem_distinguishGPmonotone}
Assume in addition that the line bundle $\mcal{L}^K_\lambda$ over $G/P$ is anticanonical. For each seed $\seed$ of the cluster algebra $\mcal{A}_P$, let $L_\seed$ be the monotone Lagrangian torus constructed in Corollary~\ref{cor_monotoneLagtoriconst}. If $\mcal{A}_P$ is of infinite type, then the set of equivalence classes 
$$
\bigl\{ [L_\seed] \mid \mbox{$\seed$ is a seed\,} \bigr\}
$$
is infinite.
\end{theorem}

\begin{remark}\label{remark_sameobjdifseed}
Even if we have two different seeds, the associated Newton--Okounkov polytopes and monotone Lagrangian tori may coincide, see \cite{CKLP21,Kim24}. One cannot expect that there is a one-to-one correspondence between the set of equivalence classes and the set of distinct seeds. 
\end{remark}

To prove Theorem~\ref{theorem_distinguishGP}, we require a refinement of \cite[Theorem C]{CKKP25}.

\begin{theorem}[cf.\ Theorem C in \cite{CKKP25}]\label{theorem_maincriterion}
Let $X$ be an irreducible projective variety and let $\mathcal{L}$ be a globally generated and big line bundle over $X$. Suppose that $(X, \mcal{L})$ admits a family 
$$
\{ \Delta_\seed \mid \mbox{$\seed$ is a seed\,} \}
$$
of Newton--Okounkov bodies containing the origin, arising from a family of finitely generated saturated semigroups endowed with a tropical cluster structure. Assume that there exists a polytope $\Delta_{\seed_0}$ in this family satisfying 
\begin{enumerate}
\item (Fixed point) there exists a point $\mathbf{u}_0 \in \mathrm{Int}(\Delta_{\seed_0}) \cap \Q^m$ that is fixed under all iterated tropicalized cluster mutations. 
\end{enumerate}
Suppose further that there exist a sequence $\left( \seed_\ell \right)_{\ell \in \mathbb{N}}$ of seeds and indices $(s, r) \in J_\mathrm{uf} \times J$ such that
\begin{enumerate}
\setcounter{enumi}{1}
\item (Multiplicity) the sequence $( \varepsilon^\ell_{s, r} )_{\ell \in \mathbb{N}}$ of the $(s, r)$-entries in the extended exchange matrices diverges to $- \infty$ as $\ell \to \infty$ and
\item (Non-negativity) for every seed $\seed$, the Newton--Okounkov polytope $\Delta_{\seed}$ is contained in the half-space $\{ \mathbf{u} \in \R^m \mid u_s \geq 0 \}$.
\end{enumerate}
Then the set~\eqref{equ_NOpolyseed} of equivalence classes is infinite.
\end{theorem}

The proof of Theorem~\ref{theorem_maincriterion} follows that of \cite[Theorem C]{CKKP25} with minor modifications. We briefly indicate the necessary changes and then explain how to apply it to our setting. 

The main difference is that the polytopes are assumed to be $\Q$-Gorenstein in \cite[Theorem C]{CKKP25} and hence admit a distinguished interior lattice point. This point serves as a center whose affine distance to each facet is equal. In our setting, such a canonical lattice point need not exist since we consider globally generated and big line bundles over $X$. These conditions in Theorem~\ref{theorem_maincriterion} imply that each Newton--Okounkov body is a full dimensional rational polytope.

By condition (1), there exists a rational point $\mathbf{u}_{0}$ in the interior of $\Delta_{\seed_0}$ that is fixed under iterated tropicalized cluster mutations. 
Such fixed points can be characterized as follows.

\begin{lemma}[Corollary 4.9. in \cite{CKKP25}]\label{lemma_characfixed}
The following are equivalent. 
\begin{enumerate}
\item A point $\mathbf{u} \in \R^{J}$ is fixed under all iterated tropicalized cluster mutations,
\item All $J_\mathrm{uf}$-components of $\mathbf{u}$ are zero.
\end{enumerate}
\end{lemma} 

Fix such a point $\mathbf{u}_{0}$ in $\mathrm{Int}(\Delta_{\seed_0}) \cap \Q^m$. Since $\mathbf{u}_0$ lies in the interior of $\Delta_{\seed_0}$ and $\Delta_{\seed_0}$ is contained in the half-space $\{ \mathbf{u} \in \R^m \mid u_s \geq 0 \}$ by condition (3), the $s$-th component of $\mathbf{u}_0$ is a strictly positive rational number. Thus, it can be written as $p / q$ where $(p,q) \in \mathbb{N} \times  \mathbb{N} $ and $\mathrm{gcd}(p,q) = 1$. For $\mathbf{v} = \mathbf{e}_s$ or $\mathbf{e}_s - \varepsilon^\ell_{r,s} \mathbf{e}_r$, Lemma~\ref{lemma_characfixed} implies that for all $\ell$, we have 
$$
\langle \mathbf{u}_0, \mathbf{v} \rangle = \frac{p}{q} \in \Q_{>0}.
$$
We emphasize that $p$ and $q$ are independent of $\ell$. 

To show that the family contains infinitely many distinct polytopes, we introduce an invariant under integral affine transformations. Let $\Delta \subseteq \R^m$ be a full dimensional polytope and let $\Delta^\circ$ be the polar dual of $\Delta$. For fixed $p \in \mathbb{N}$, define 
$$
n(\Delta) \coloneqq \max_{\mathbf{u}_i \in \mathscr{I}} \left[ \# \left( \frac{1}{p} \Z^m \cap (\Delta - \mathbf{u}_i)^\circ \right) \right]
$$
where $\mathscr{I} \coloneqq \frac{1}{p}\Z^m \cap \mathrm{Int}(\Delta)$. Note that $n(\Delta)$ is finite because $\Delta - \mathbf{u}_i$ is full dimensional and contains the origin.

We claim that $n(\Delta) $ is invariant under any integral affine transformation. 

\begin{proposition}
The number $n(\Delta)$ is invariant under integral affine transformations. Namely, if $\Delta \sim \Delta^\prime$, then $n(\Delta) = n(\Delta^\prime)$.
\end{proposition}

\begin{proof}
The claim follows from the facts that any integral affine transformation preserves the lattice $\frac{1}{p} \Z^m$ and that if $\Delta$ and $\Delta^\prime$ are related by the matrix multiplication by $A \in \mathrm{GL}(m, \Z)$, then their polar duals are related by its transpose. 
\end{proof}

For the family $\{ \Delta_\seed \mid \mbox{$\seed$ is a seed\,} \}$ in Theorem~\ref{theorem_maincriterion}, set
$$
n_\seed \coloneqq n(\Delta_\seed)
$$
This invariant can be used to detect the existence of infinitely many distinct members.

\begin{corollary}\label{cor_nslinf}
Let $(\Delta_{\seed_\ell})_{\ell \in \mathbb{N}}$ be a sequence of full dimensional polytopes in $\R^m$.
If $n(\Delta_{\seed_\ell}) \to \infty$ as $\ell \to \infty$, then the set~\eqref{equ_NOpolyseed} is infinite.
\end{corollary}

We are now ready to prove Theorem~\ref{theorem_maincriterion}.

\begin{proof}[Proof of Theorem~\ref{theorem_maincriterion}]
Recall from \cite[Lemma 4.16]{CKKP25} that the polytope $\Delta_{\seed_\ell}$ is supported by the half-spaces given by 
$$
\{ \mathbf{u} \in \R^m \mid \langle \mathbf{u}, \mathbf{e}_s \rangle \geq 0 \}, \,\, \{ \mathbf{u} \in \R^m \mid \langle \mathbf{u}, \mathbf{e}_s - \varepsilon^\ell_{r,s} \mathbf{e}_r \rangle \geq 0 \}.
$$ 
It follows that the translated polytope $\Delta_{\seed_\ell} - \mathbf{u}_0$ is supported by the half-spaces 
$$
\left\{ \mathbf{u} \in \R^m \mid \left\langle \mathbf{u}, \frac{q}{p} \mathbf{e}_s \right\rangle + 1 \geq 0 \right\}, \,\, \left\{ \mathbf{u} \in \R^m \mid \left\langle \mathbf{u}, \frac{q}{p} \bigl( \mathbf{e}_s - \varepsilon^\ell_{r,s} \mathbf{e}_r \bigr) \right\rangle + 1 \geq 0 \right\}.
$$ 
By \cite[Theorem 6.4]{Brn} (see also \cite[Lemma 4.15]{CKKP25}), the polar dual $(\Delta_{\seed_\ell} - \mathbf{u}_0)^\circ$ contains the line segment joining $ \frac{q}{p} \mathbf{e}_s$ and $\frac{q}{p} \bigl( \mathbf{e}_s - \varepsilon^\ell_{r,s} \mathbf{e}_r \bigr)$. Consequently, the polar dual contains at least $- q \varepsilon^\ell_{r,s} + 1$ lattice points in $\frac{1}{p} \Z^m$. Since $ - \varepsilon^\ell_{r,s} \to \infty$ as $\ell \to \infty$,we conclude that $n(\Delta_{\seed_\ell}) \to \infty$ as $\ell \to \infty$. Theorem~\ref{theorem_maincriterion} follows from Corollary~\ref{cor_nslinf}. 
\end{proof}

We have established the criterion, which will be used to prove Theorem~\ref{theorem_distinguishGP}. We now explain why condition~(1) in Theorem~\ref{theorem_maincriterion} holds in our setting. By \cite{FO17}, there exists a seed $\seed$ such that $\Delta_\seed$ is equivalent to a string polytope of a Schubert variety. Using the explicit description of string polytopes (see, for example, \cite{BZ01, Lit98}), one verifies the existence of a fixed point as in condition~(1) by Lemma~\ref{lemma_characfixed}.

In Section~\ref{sec_nonnegativity}, we verify the non-negativity condition in Theorem~\ref{theorem_maincriterion} by exploiting properties of elements in the dual canonical basis. Sections~\ref{Sec_labledextendedexch} and \ref{sec_distinguishing} are devoted to establishing the multiplicity condition.

\section{Non-negativity of the dual canonical basis}\label{sec_nonnegativity}

The goal of this section is to review the dual canonical basis and the cluster algebra structures on unipotent cells and unipotent groups. By exploiting the relationship between the cluster algebras on these spaces and their monoidal categorifications together with their localizations, we establish the non-negativity of certain components of the associated extended $g$-vectors.

\subsection{Dual canonical basis}

Let $\frak{g}$ be a complex semisimple Lie algebra and let $I$ be the index set for simple roots. Let $U(\mathfrak{g})$ be the \emph{universal enveloping algebra} of $\mathfrak{g}$ and  let $U_q(\mathfrak{g})$ be the $q$-deformation of $U(\mathfrak{g})$ over $\C(q)$, called the \emph{quantum group} associated with the Cartan matrix $\mathsf{A}$ where $q$ is an indeterminate, see \cite{LusztigBook}. The algebra $U(\mathfrak{g})$ is recovered by specializing $U_q(\mathfrak{g})$ at $q=1$. For each $i \in I$, we denote by $e_i$ and $f_i$ the \emph{Chevalley generators} of $U_q(\mathfrak{g})$ of weight $\alpha_i$ and $- \alpha_i$, respectively. We denote by $U_q^-(\mathfrak{g})$ the \emph{negative half} of $U_q(\mathfrak{g})$, which is the subalgebra of $U_q(\mathfrak{g})$ generated by $\{ f_i \mid i \in I\}$. 

Let $\mathbf{B}(\infty)$ be the \emph{infinite crystal} of $U_q^-(\mathfrak{g})$ and let $\mathsf{G}^{\mathrm{up}}(\infty)$ be the \emph{dual canonical basis} (or \emph{upper global basis}) of the restricted dual of $U_q^-(\mathfrak{g})$, see \cite{KashiwaraBook, LusztigBook} for instance. Since $\textbf{B}(\infty)$ is regarded as the specialization of $\mathsf{G}^{\mathrm{up}}(\infty)$ at $q=0$, we write 
$$
\mathsf{G}^{\mathrm{up}}(\infty) = \{ \mathsf{G}^{\mathrm{up}}(b) \mid b\in \mathbf{B}(\infty) \}.
$$
Let $\te_i$ and $\tf_i$ be the \emph{Kashiwara operators} of weight $\alpha_i$ and $- \alpha_i$, respectively. For a reduced expression $w = s_{i_1} \dots s_{i_\ell}$ of $w$ in the Weyl group $W$, define the subcrystal $\mathbf B_w(\infty)$ of $\mathbf{B}(\infty)$ by
$$\mathbf B_w(\infty)
= \bigl\{\tf_{i_1}^{a_1} \tf_{i_2}^{a_2} \dots \tf_{i_{\ell-1}}^{a_{\ell-1}} \tf_{i_\ell}^{a_\ell} {\mathbf 1}\, \vert \, (a_1, \dots, a_\ell) \in \Z_{\ge 0}^\ell \bigr\}.
$$
This set is independent of the choice of reduced expression.

Recall that for $w$ in the Weyl group $W$, the unipotent cell $U_w^{-} = U^- \cap BwB$ has the cluster algebra structure in Section~\ref{subsec_clusunicell}. The coordinate ring $\C[U_w^{-}]$ is isomorphic to the cluster algebra $\mcal{A}(\seed_0)$, whose initial seed $\seed_0$ is given in~\eqref{equ_initialseed}. The specialization of the restricted dual of $U_q^-(\mathfrak{g})$ at $q=1$ is isomorphic to $\C[U^{-}]$ and the specialization of the dual canonical basis yields a basis $\mathsf{G}^{\mathrm{up}}_{q=1}(\infty)$ of $\C[U^{-}]$. Let
\begin{itemize}
\item 
$\iota_w \colon U^- \cap X_w \hookrightarrow U^-$ be the natural inclusion and 
\item 
$\iota_w^* \colon \C[U^-] \twoheadrightarrow \C[U^- \cap X_w]$
be the corresponding ring homomorphism. 
\end{itemize}
Then we have
\begin{itemize}
\item $\{ \mathsf{G}^{\mathrm{up}}_{q=1}(b) \in \C[U^-] \, \vert \,b \in \mathbf B(\infty) \setminus \mathbf B_w(\infty)\}$ forms a $\C$-basis of the kernel of $\iota_w^*$ and
\item $\{\iota_w^* (\mathsf{G}^{\mathrm{up}}_{q=1}(b)) \in \C[U^-\cap X_w] \, \vert \, b \in  \mathbf B_w(\infty)\} $ forms a $\C$-basis of $\C[U^-\cap X_w]$.
\end{itemize}

Since $U^-_w$ is the open subvariety of $U^- \cap X_w$ defined by $0 \neq D_w \coloneqq \prod_{i\in I}D_{w\varpi_i,\varpi_i}$ by \cite[Lemma 2.17]{KO21}, the coordinate ring $\C[U^-_w] $ is isomorphic to the localized ring $\C[U^- \cap X_w]_{D_w}$. Through the isomorphism, we define
$$
\mathsf{G}^{\mathrm{up}}_w(\infty) \coloneqq  \left\{ \iota_w^* \mathsf{G}^{\mathrm{up}}_{q=1}(b) \cdot {\prod_{i\in I} \iota_w^*(D_{w \varpi_i, \varpi_i})^{-a_i}} \mid b \in \textbf{B}_w(\infty),  a_i \in \Z_{\geq 0} \mbox{ for $i \in I$} \right\},
$$
which forms a basis of $\C[U^- \cap X_w]_{D_w}\simeq \C[U^-_w]$, see \cite[Theorem C.1]{FO25}. We also refer to $\mathsf{G}^{\mathrm{up}}_w(\infty)$ as the \emph{dual canonical basis} of $\C[U^-_w]$. 

The following is the main theorem of this section, concerning the non-negativity of certain  components of the extended $g$-vectors of the dual canonical basis elements in $\mathsf{G}^{\mathrm{up}}_w(\infty)$ when $\frak{g}$ is of simply laced type.

\begin{theorem}\label{theorem_main6}
Let $\mathfrak{g}$ be a finite dimensional semisimple Lie algebra of simply laced type over $\C$ and let $I$ denote the index set of the simple roots of $\mathfrak{g}$. 
For $w \in W$, set 
\begin{equation}\label{equ_defIw}
I(w) \coloneqq \{i\in {\rm supp}(w) \, \vert \, ws_i>w\}.
\end{equation}
For each $k \in {\rm supp}(w) \setminus I(w)$, let $j_{k} \in J_\mathrm{fz}$ be the index corresponding to the frozen variable $D_{w \varpi_k, \varpi_k}$ with respect to the initial seed $\seed_0$ of the cluster algebra $\mcal{A}(\seed_0) \simeq \C[U^-_w]$. Then, for every dual canonical basis element $x \in \mathsf{G}^{\mathrm{up}}_w(\infty)$ which lies in $\C[U^-\cap X_w]$ and each seed $\seed$, the $j_k$-th component of its extended $g$-vector with respect to $\seed$ is nonnegative, i.e.,
$$
\bigl( g_\seed (x) \bigr)_{j_k} \geq 0.
$$
\end{theorem}

\begin{remark} We record the following remarks.
\begin{itemize}
\item For $K \subseteq I$, let $w_K$ be the minimal length representative of the coset $w_0W_K$ in $W$. Then $I(w_K)= I \setminus K$. 
\item Theorem \ref{theorem_main6} remains valid when $\mathfrak{g}$ is a symmetric Kac--Moody algebra.
\end{itemize}
\end{remark}

As a consequence of Theorem~\ref{theorem_main6}, we verify the non-negativity condition in Theorem~\ref{theorem_maincriterion}. For any dominant integral weight $\lambda \in \mathsf{P}^+$ and $w \in W$, there exists a subset called the \emph{Demazure crystal} $\mathbf{B}_w(\lambda) \subseteq \textbf{B}(\infty)$ such that the set 
$$
\mathsf{G}^{\mathrm{up}}_w(\lambda) \coloneqq \bigl\{\iota_w^* \mathsf{G}^{\mathrm{up}}_{q=1}(b) \mid b \in \mathbf{B}_w(\lambda) \bigr\} \subseteq \mathsf{G}^{\mathrm{up}}_w(\infty),
$$
generates the vector space
\begin{equation}\label{equ_sigmatau}
\mathsf{L}_\lambda \coloneqq \left\{ \sigma/ \tau_\lambda \mid \sigma \in H^0(X_w, \mcal{L}_\lambda) \right\} \subseteq \C[U^-\cap X_w],
\end{equation}
by the compatibility of the dual canonical basis with highest weight modules, see \cite[Corollary 3.20]{FO17}. We then have the following corollary.

\begin{corollary}[Proposition~\ref{thmx_E}]\label{cor_propertyofclustercone}
Let $\mathfrak{g}$ be a finite dimensional semisimple Lie algebra of simply laced type over $\C$. For $w \in W$ and a seed $\seed$ of the cluster algebra $\mcal{A}(w)$, let $C_\seed(w)$ be the cluster cone of $w$ with respect to $\seed$. Let $j_{k} \in J_\mathrm{fz}$ be the index corresponding to the frozen variable $D_{w \varpi_k, \varpi_k}$ with respect to the initial seed $\seed_0$. Then, for every seed $\seed$,
$$
C_\seed(w) \subseteq \{ \mathbf{u} \in \R^{J} \mid u_{j_k} \geq 0 \mbox { for all $j \in \mathrm{supp}(w) \setminus I(w).$} \}
$$
In particular, every Newton--Okounkov polytope $\Delta(X^K_w,\mcal{L}^K_\lambda,v^K_\seed,\tau^K_\lambda)$ satisfies 
$$
\Delta(X^K_w,\mcal{L}^K_\lambda,v^K_\seed,\tau^K_\lambda) \subseteq \{ \mathbf{u} \in \R^{J} \mid u_{j_k} \geq 0 \mbox { for all $j \in \mathrm{supp}(w) \setminus I(w).$} \}
$$
\end{corollary}

\begin{remark}\label{remark_nonnegativeexamples}
If $k \in I(w)$, then the $j_k$-th component of the extended $g$-vector of a dual canonical basis element may take negative values. As an example, let $G = \mathrm{SL}_4(\C)$ and $w = s_1 s_2 s_1 s_3 \in W$. Then $I(w) = \{2\}$. Let $\seed$ be the seed corresponding to the reduced expression $\underline{w} = s_1 s_2 s_1 s_3$. The indices $j_1 = 3, j_2 = 2$, and $j_3 = 4$ correspond to the frozen variables associated with $s_1, s_2$, and $s_3$, respectively. The cluster cone $C_\seed(w)$ is given by 
$$
C_\seed(w) = \{ \mathbf{u} \in \R^4 \mid u_1 + u_2 + u_4 \geq 0, u_2 + u_4 \geq 0, u_4 \geq 0, u_3 \geq 0\}.
$$
Observe that the components $u_3$ and $u_4$ of $\mathbf{u} \in C_\seed(w)$ are always nonnegative, whereas $u_2$ can take negative values.
\end{remark}

For a seed $\seed$, the valuation $v_\seed$ of an element in $\mathsf{L}_{\lambda}$ corresponds to a point in the semigroup $S(X_w, \mcal{L}_\lambda, v_\seed, \tau_\lambda)$. Since the valuation of a dual canonical basis element agrees with its extended $g$-vector, Theorem~\ref{theorem_main6} implies that the corresponding component of the valuation is nonnegative for every element of $S(X_w, \mcal{L}_\lambda, v_\seed, \tau_\lambda)$.

We close this section by outlining the strategy of the proof of Theorem~\ref{theorem_main6}. The argument is summarized by the commutative diagram~\eqref{equ_maindiagram11}. The construction of Newton–Okounkov polytopes relies on the cluster algebra structure on $\C[U^-_w]$ and the associated valuation is computed via extended $g$-vector on $\C[U^- \cap X_w] \setminus \{0\}$. 
\begin{equation}\label{equ_maindiagram11}
\xymatrix@C=5em{
\C[U^-(w)]\ar@{>->}[d] \ar@{>->}[rr] & & \C[U^-(w)]_{D_w} \ar[d]^{\cong} & \\
\C[U^-]  \ar@{->>}[r]_{\iota_w^*} \ar[urr]^{} 
& \C[U^-\cap X_w] \ar@{>->}[r] \ar[ur]
&\C[U^-_w]
}.
\end{equation}

To establish non-negativity, we employ two key ingredients. The first is the cluster algebra structure on the coordinate ring $\C[U^-(w)]$ of the unipotent group $U^-(w)$. Since the frozen variables are not inverted, their corresponding components remain nonnegative, and any element coming from $\C[U^-(w)]$ inherits this property. Each element in $\C[U^-\cap X_w]$ lies in a localization $\C[U^-(w)]_{D_w}$, so its expression may involve arbitrarily high powers of the element $D_w$ a priori. We need to show that the frozen variables indexed by ${\rm supp}(w) \setminus I(w)$ do not appear in the denominators of elements of the dual canonical basis.

Rather than performing ring-theoretic calculation, we pass the problem to the setting of the monoidal categorification. We realize the ring homomorphism $\C[U^- \cap X_w] \to \C[U^-(w)]_{D_w}$ via the localized functor between their categorifications~\eqref{equ_diagramincategory} and~\eqref{equ_maindiagram1}. Moreover, using structural properties of the simple modules corresponding to dual canonical basis elements in this monoidal categorification (Proposition~\ref{thm_dualcanonical}), we show that the frozen variables indexed by ${\rm supp}(w) \setminus I(w)$ do not appear in the denominator.

\begin{remark}
We expect that the statement remains valid in the symmetrizable case. The main obstacle to such a generalization is that the required monoidal categorification is available only in the symmetric case.
\end{remark}

\subsection{Cluster algebra structures on unipotent subgroups}

To prove Theorem~\ref{theorem_main6}, we make use of cluster algebra structure on the coordinate ring of the unipotent group $U^-(w)$, which we briefly recall below. Define 
$$
\Phi^+(w) \coloneqq \bigl\{ \alpha \in \Phi^+ \mid w^{-1} \alpha \in (- \Phi^+) \bigr\}
$$ 
where $\Phi^+$ is the set of positive roots. For a reduced expression $\underline{w} =  s_{i_1} s_{i_2} \cdots s_{i_m}$ of $w$, we have
$$
\Phi^+(w) = \{  s_{i_1} s_{i_2}\cdots s_{i_{j-1}}(\alpha_{i_j}) \mid j = 1, \dots, m\}.
$$
Let $U^-(w)$ be the \emph{unipotent group} corresponding to the nilpotent Lie algebra
$$
\frak{n}^-(w) \coloneqq \bigoplus_{\alpha \in \Phi^+(w)} \frak{g}_{-\alpha}.
$$
In \cite{GLS11}, Gei{\ss}--Leclerc--Schr{\"{o}}er constructed cluster algebra structures on the coordinate rings $\C[U^-(w)]$ and $\C[U^-_w]$ and explored their relationships in the case where $\frak{g}$ is a symmetric Kac--Moody algebra. Note that the cluster algebra structure on $\C[U^-_w]$ coincides with that in Theorem~\ref{theorem_clustunitcell} as they have the same generators with the same exchange matrix at a seed.

Let $\mathbb{T}$ be the exchange graph associated with the cluster algebra $\C[U_w^-]$ with the initial seed at the vertex $t_0$. We define the cluster algebra $\overline{\mcal{A}}(\seed_0)$ to be the subalgebra of $\mathcal{F} = \C(A_{j, t_0} \mid j \in J)$ generated by
$$
\left\{A_{j,t} \mid t \in \mathbb{T}, j \in J \right\}.
$$
By the Laurent phenomenon \cite{FZ02}, $\overline{\mcal{A}}(\seed_0)$ can be regarded as an algebra over the polynomial ring $\C[ A_{j,t} \mid j \in J_\mathrm{fz}]$. Unlike $\mathcal{A}(\seed_0)$, the frozen (coefficient) variables are \emph{not} inverted in $\overline{\mathcal{A}}(\seed_0)$. The algebraically independent variables of the initial seed are given by generalized minors. By restricting these minors to $U^-_w$ and $U^-(w)$, we obtain seeds for $\C[U^-(w)]$ and $\C[U^-_w]$, respectively.

\begin{theorem}[Theorem 3.3 in \cite{GLS11}] \label{theorem_clusterstructreonunisub}
The coordinate rings $\C[U^-_w]$ and $\C[U^-(w)]$ carry cluster algebra structures such that
\begin{enumerate}
\item ${\mathcal{A}}(\seed_0) \simeq \C[U^-_w]$ given by $A_{j,t_0} \mapsto \Delta_{w_{\leq j} \varpi_{i_j}, \varpi_{i_j}} |_{U^-_w}$ is a $\C$-algebra isomorphism,  
\item $\overline{\mcal{A}}(\seed_0)\simeq \C[U^-(w)]$ given by $A_{j,t_0} \mapsto \Delta_{w_{\leq j} \varpi_{i_j}, \varpi_{i_j}} |_{U^-(w)}$ is a $\C$-algebra isomorphism,  
\item ${\mcal{A}}(\seed_0)$ is obtained from $\overline{\mcal{A}}(\seed_0)$ by inverting all frozen variables.
\end{enumerate}
\end{theorem}

To apply Theorem~\ref{theorem_clusterstructreonunisub} to the proof of Theorem~\ref{theorem_main6}, we examine the compatibility of the valuation of dual canonical basis elements on $\C[U^- \cap X_w] \setminus \{0\}$ following \cite[Section 2.6]{KO21} and \cite[Section 8]{GLS11}. 

Consider the following commutative diagram. 
\begin{equation}
\xymatrix{
U^-(w)  & & U^-(w) \cap O_w  \ar@{_{(}->}[ll] & \\
U^-  \ar[u]_{\pi(w)} &  U^-\cap X_w \ar@{_{(}->}[l]^{\iota_w}  & U^-_w \ar@{_{(}->}[l]\ar^{\cong}[u]_{\bar\pi(w)}
}
\end{equation}
Recall that there is an isomorphism $U^-(w)^\prime \times U^-(w) \simeq U^-$ given by multiplication where $U^-(w)^\prime \coloneqq U^-\cap w U^- w^{-1}$. Let $\pi(w) \colon U^- \to U^-(w)$ denote the projection, which is a retraction of the inclusion $j(w) \colon U^-(w) \to U^-$, that is, $\pi(w) \circ j(w) = {\rm id}_{U^-(w)}$. Let
$O_w \coloneqq \{g \in G \, \vert \,  \Delta_{w\lambda,\lambda}(g) \neq 0 \,\, \text{for all } \lambda \in P^+ \}$ 
and the horizontal arrows are inclusions. In addition, the restriction of $\pi(w)$ to $U^-_w$ induces an isomorphism $\bar \pi(w) \colon U^-_w \cong U^-(w) \cap O_w$, see \cite[Proposition 2.20 (5)]{KO21}. 

Passing to coordinate rings, we obtain the following commutative diagram.
\begin{equation}\label{equ_maindiagram}
\xymatrix@C=5em{
\C[U^-(w)]\ar@{>->}[d]_{\pi(w)^*} \ar@{>->}[rr]^{\Omega_w} & & \C[U^-(w) \cap O_w ] \ar[d]_{\bar \pi(w)^*}^{\cong} & \\
\C[U^-]  \ar@{->>}[r]_{\iota_w^*} 
& \C[U^-\cap X_w] \ar@{>->}[r] 
& \C[U^-_w] 
}.
\end{equation}
Through the injective ring homomorphism $\pi(w)^* \colon \C[U^-(w)] \to \C[U^-]$, we identify $\C[U^-(w)]$ with the $\C$-subalgebra of $\C[U^-]$ generated by the \emph{dual PBW generators} of $\C[U^-(w)]$, see \cite[Proposition 8.2]{GLS11}. Moreover, 
the intersection of $\mathsf{G}^{\mathrm{up}}_{q=1}(\infty)$ with this subalgebra forms a basis.  

For notational simplicity, we denote the restriction $\Delta_{u\lambda,\lambda}\vert_{U^-}$ by $D_{u\lambda,\lambda}$ for $u\in W$.
For any reduced expression $\underline{w}=s_{i_1} \cdots s_{i_m}$,  the elements 
$$
\{D_{w_{\le j}\varpi_{i_j}, \varpi_{i_j}} \, |\, j=1, \dots, m\}
$$ 
are polynomials in the PBW generators of $\C[U^-]$ and hence they lie in the image of $\pi(w)^*$. Since $j(w)^* \circ \pi(w)^* = {\rm id}_{\C[U^-(w)]}$, it follows that $$D_{w_{\le j}\varpi_{i_j}, \varpi_{i_j}}  = \pi(w)^*(\Delta_{w_{\le j}\varpi_{i_j}, \varpi_{i_j}}\vert_{U^-(w)}).$$
Let 
$$
D_w \coloneqq \prod_{i \in I} D_{w\varpi_i, \varpi_i} \in  \C[U^-]. 
$$
Then the map $\Omega_w$ is given by localization
$$
  \C[U^-(w)] \to \C[U^-(w) \cap O_w ] \simeq \C[U^-(w)]_{D_w}.
$$

Note that the homomorphism $\C[U^-\cap X_w] \to \C[U^-_w] $ is also given by the localization at $\iota_w^*(D_w)$ (\cite[Lemma 2.17, Lemma 2.18]{KO21}), that is, 
$$
\C[U^-]/{\rm Ker}  (\iota^*_w) \simeq \C[U^- \cap X_w ]   \to \C[U^- \cap X_w ]_{\iota^*(D_w)} \simeq \left( \C[U^-]/{\rm Ker } (\iota^*_w) \right)_{\overline{D_w}} \simeq  \C[U_w^-].
$$
where $\overline{D_w}$ is the coset represented by $D_w$ in the quotient $\C[U^-]/{\rm Ker}  (\iota^*_w)$. Then the map 
$$
\bar \pi(w)^* \colon \C[U^-(w)]_{D_w} \simeq \C[U^-(w) \cap O_w ]  \to  \C[U^-_w]  \simeq \left( \C[U^-]/{\rm Ker } (\iota^*_w) \right)_{\overline{D_w}}
$$ 
is given by
\begin{equation}
\bar\pi(w)^*(x) = x \quad \text{for} \ x \in\C[U^-(w)] \, \text{ and } \, \bar\pi(w)^*( D_{w\lambda, \lambda})= \overline{D_{w\lambda, \lambda}} \ \text{for} \ \lambda \in P^+.  
\end{equation}

\subsection{Non-negativity of $g$-vectors of dual canonical basis}

To prove Theorem~\ref{theorem_main6}, recall that
$$
\{\iota_w^* (\mathsf{G}^{\mathrm{up}}_{q=1}(b)) \in \C[U^-\cap X_w] \, \vert \, b \in  \mathbf B_w(\infty)\}  
$$
is a $\C$-basis of $\C[U^-\cap X_w]$. We begin with the following proposition.

\begin{proposition}\label{thm_dualcanonical}
Let $b \in \mathbf{B}_w(\infty)$ and consider the corresponding dual canonical basis element $\mathsf{G}^{\mathrm{up}}_{q=1}(b) \in \C[U^-]$. Then there exist
\begin{equation}\label{equ_lambdag}
\lambda = \sum_{k \in I(w)} \lambda_k \varpi_k \in \mathsf{P}^+ \mbox{ and  \,} \mathsf{G}^{\mathrm{up}}_{q=1}(\widetilde{b}) \in \C[U^{-}(w)]\cap \mathsf{G}^{\mathrm{up}}_{q=1}(\infty)
\end{equation}
such that 
\begin{equation}\label{equ_choicemu}
\mathsf{G}^{\mathrm{up}}_{q=1}(b) = D_{w \lambda, \lambda}^{-1} \cdot \mathsf{G}^{\mathrm{up}}_{q=1}(\widetilde{b})  \,\, \mbox{ in $\C[U^-(w)]_{D_w}\simeq \C[U^-_w]$.}
\end{equation}
\end{proposition}

Note that the sum in~\eqref{equ_lambdag} is taken over $I(w)$ in~\eqref{equ_defIw}. In particular, $\lambda_k = 0$ for all $k \notin I(w)$.

Assuming Proposition~\ref{thm_dualcanonical}, we now prove Theorem~\ref{theorem_main6}.
\begin{proof}[Proof of Theorem~\ref{theorem_main6}]
Taking extended $g$-vectors of~\eqref{equ_choicemu} with respect to a seed $\seed$, we have
\begin{equation}\label{equ_extededgvectorsum}
g_{\seed}(\mathsf{G}^{\mathrm{up}}_{q=1}(b)) = g_{\seed} (D_{w \lambda, \lambda}^{-1}) + g_{\seed} (\mathsf{G}^{\mathrm{up}}_{q=1}(\widetilde{b})).
\end{equation}
For $k \in I$, let $j_k \in J_\mathrm{fz}$ denote the index corresponding to the frozen variable $D_{w \varpi_k, \varpi_k}$ with respect to the initial seed $\seed_0$ (and hence the above seed $\seed$).

First, we claim that
\begin{equation}\label{equ_gseedg}
g_{\seed} \bigl(\mathsf{G}^{\mathrm{up}}_{q=1}(\widetilde{b})\bigr)_{j_k} \geq 0 \quad \text{for all} \ k \in I.
\end{equation}
Since $\mathsf{G}^{\mathrm{up}}_{q=1}(\widetilde{b})$ lies in the cluster algebra $\C[U^-(w)]$ and frozen variables are not inverted in this cluster algebra by Theorem~\ref{theorem_clusterstructreonunisub}, no frozen variable appears in its Laurent expansion with respect to the cluster variables in $\seed$ in the denominators.

Next, by Proposition~\ref{thm_dualcanonical}, for $\lambda$ in~\eqref{equ_lambdag} we have
$$
g_{\seed} \bigl(D_{w \lambda, \lambda}^{-1} \bigr) =  g_{\seed} \left( \prod_{k\in I} D_{w \varpi_k, \varpi_k}^{-\lambda_k} \right) = \sum_{k\in I} -\lambda_k {\bold e}_{j_k} =\sum_{k\in  I(w)} -\lambda_k {\bold e}_{j_k} 
$$
so that 
\begin{equation}\label{equ_gseeddw}
g_{\seed} \bigl(D_{w \lambda, \lambda}^{-1} \bigr)_{j_k}=0 \quad \text{$k \in I \setminus I(w)$}.
\end{equation}

Combining~\eqref{equ_extededgvectorsum},~\eqref{equ_gseedg}, and~\eqref{equ_gseeddw}, we conclude that 
$$
g_{\seed}(\mathsf{G}^{\mathrm{up}}_{q=1}(b))_{j_k}\geq 0 \quad \text{$k \in I \setminus I(w)$}
$$
as desired.
\end{proof}

It remains to prove Proposition~\ref{thm_dualcanonical}. Our argument relies on the monoidal categorification of the cluster algebra structures on $\C[U^-(w)]$ and $\C[U^-_w]$ developed in \cite{KKKO18, KKOP21, KKOP23} together with a categorical analogue of Proposition~\ref{thm_dualcanonical} established by Kashiwara--Nakashima \cite{KN25}. 

To state \cite[Proposition 5.28]{KN25}, we briefly recall the necessary background. Let ${R}$ be a symmetric quiver Hecke algebra of type $\frak g$ over an algebraically closed field of characteristic zero, and let $\scr{C} \coloneqq R\text{-}\mathsf{gmod}$ be the category of finite-dimensional graded $R$-modules. The category $R\text{-}\mathsf{gmod}$ provides a monoidal categorification of the quantum unipotent coordinate ring $A_q(\mathfrak{n}^-)$, a $q$-deformation of $\mathbb{C}[U^-]$. Its Grothendieck ring $K(\mathscr{C})$ is isomorphic to the integral form $A_{\Z[q^{\pm 1}]}(\frak{n}^-)$. Under this identification, the convolution product $\circ$ in the category $\scr{C}$ corresponds to the multiplication in the algebra $A_{\Z[q^{\pm 1}]}(\frak{n}^-)$ and the isomorphism classes of simple modules correspond to dual canonical basis elements.

For each $w \in W$, there exists a full subcategory $\scr{C}_w $ of $\scr{C}$ that categorifies the quantum unipotent coordinate ring $A_q(\frak{n}^-(w))$, a $q$-deformation of $\C[U^-(w)]$. This categorification is compatible with the cluster algebra structure: cluster monomials correspond to real simple modules. In particular, the  \emph{determinantial modules} $\mathsf{M}_{w \lambda, \lambda}$ correspond to the unipotent quantum minors $D_{w \lambda, \lambda}$.

Let $\widetilde{\scr{C}}_w=\scr{C}_w  [\mathsf{M}^{-1}_{w\varpi_i, \varpi_i} \mid i \in I]$ be the localization of $\scr{C}_w$ obtained by adding the inverses of simple modules $\mathsf{M}_{w \varpi_i, \varpi_i}$, and let $\mathsf{\Omega}_w$ be the associated localization functor constructed in \cite{KKOP21}. Similarly, let $\mathsf Q_w \coloneqq \scr{C} \to \scr{C}[\mathsf{M}^{-1}_{w\varpi_i, \varpi_i} \mid i \in I]$ be the corresponding localization functor.  Then  $\widetilde{\scr{C}}_w$ and $\scr{C}[\mathsf{M}^{-1}_{w\varpi_i, \varpi_i} \mid i \in I]$ are equivalent. More precisely, we have  the following commutative diagram of functors 
\begin{equation}\label{equ_diagramincategory}
\xymatrix{
\scr{C}_w \ar@{>->}[d] \ar[rr]^{\mathsf{\Omega}_w \quad \quad \quad \quad} & & \widetilde{\scr{C}}_w \simeq \scr{C}  [\mathsf{M}^{-1}_{w\varpi_i, \varpi_i} \mid i \in I]   \\
\scr{C} \ar@{->>}[r] \ar[rru]^{\mathsf{Q}_w}  & \scr{C}/ {\rm Ker} (\mathsf Q_w) \ar[ru]_{\mathsf P_w}& } 
\end{equation}
Here, the functor $\mathsf Q_w$ factors through the Serre quotient $\scr{C}/ {\rm Ker} (\mathsf Q_w)$ of $\scr{C}$ by the kernel of $\mathsf Q_w$. The induced functor on the quotient is denoted by $\mathsf P_w$. Passing to the Grothendieck rings and specializing at $q=1$, the inclusion functor $\scr{C}_w \to \scr{C}$ and the localization functor $\mathsf \Omega_w$ yield $\pi(w)^*$ and $\Omega_w$ in \eqref{equ_maindiagram}, respectively. 

The Serre subcategory ${\rm Ker} (\mathsf Q_w)$ consists of simple modules of the form 
$$
\bigl\{ \mathsf{M} \in \scr{C} \, \vert \, [\mathsf{M}]=\mathsf{G}_{q=1}^{\mathrm{up}}(b) \,\, \text{for} \ b \in {\mathbf B}(\infty) \setminus {\mathbf B}_w(\infty) \bigr\}.
$$
Here, we use $[ \, \, ]$ to denote the class in the Grothendieck ring. Since the kernel of $\iota_w^*$ is spanned by 
$$
\bigl\{\mathsf{G}_{q=1}^{\mathrm{up}}(b) \in \C[U^-] \, \vert \,b \in \mathbf B(\infty) \setminus \mathbf B_w(\infty) \bigr\}.
$$
the quotient functor $\scr{C} \twoheadrightarrow\scr{C}/ {\rm Ker} (\mathsf Q_w)$ induces 
$$
\iota^*_w \colon \C[U^-]  \twoheadrightarrow \C[U^-\cap X_w] \simeq  \C[U^-] / {\rm ker}(\iota^*_w).
$$
We denote by $Q_w$ and $P_w$ the induced homomorphisms corresponding to $\mathsf Q_w$ and $\mathsf P_w$, respectively.

Keeping this correspondence in mind, we recall the following result.

\begin{proposition}[Proposition 5.28 in \cite{KN25}]\label{prop_KN25}
Let $\mathsf{M}$ be a simple module in $\scr{C}$ such that $\mathsf{Q}_w(\mathsf{M})$ is simple. Then there exists a dominant integral weight $$\lambda = \sum_{k \in I(w)} \lambda_k \varpi_k \in \mathsf{P}^+$$ such that the simple object $\mathsf{M}_{w \lambda, \lambda} \nabla \mathsf{M}$ belongs to  $\scr{C}_w$, where $X \nabla Y$ denotes the head of $X\circ Y$ of the simple modules in $\scr{C}$.
\end{proposition}

The key point of this proposition is that such a $\lambda \in \mathsf{P}^+$ satisfies $\lambda_k = 0$ for all $k \in I \setminus I(w)$.

\begin{proof}[Proof of Proposition~\ref{thm_dualcanonical}]
Let $\mathsf{G}^{\mathrm{up}}_{q=1}(b) \in \C[U^-]$ be a dual canonical basis element such that  ${\iota}^*_w (\mathsf{G}^{\mathrm{up}}_{q=1}(b)) \neq 0$. Let $\mathsf{M}$ be the corresponding simple module in $\scr{C}$. Recall that 
$$
\mbox{
${\iota}^*_w (\mathsf{G}^{\mathrm{up}}_{q=1}(b)) \neq 0 \Leftrightarrow Q_w (\mathsf{G}^{\mathrm{up}}_{q=1}(b)) \neq 0 \Leftrightarrow {\sf Q}_w(\sf{M})$ is simple $\Leftrightarrow b \in \mathbf{B}_w(\infty)$.
}
$$
By Proposition~\ref{prop_KN25}, there exists a weight $\lambda =\sum_{k \in I(w)} \lambda_i \varpi_k$ such that  the simple module $\widetilde{\mathsf{M}} \coloneqq \mathsf{M}_{w \lambda, \lambda} \nabla \mathsf{M}$ belongs to  $\scr{C}_w$.  
By \cite[Proposition 1.20]{KKOP23},
\begin{equation} \label{equ:QwMlaM}
\mathsf{Q}_w(\widetilde{\mathsf{M}}) = 
\mathsf{Q}_w(\mathsf{M}_{w \lambda, \lambda} \nabla \mathsf{M}) \simeq \mathsf{Q}_w(\mathsf{M}_{w \lambda, \lambda}) \circ \mathsf{Q}_w(\mathsf{M}) \quad \mbox{in $\widetilde{\scr{C}}_w$}.
\end{equation}
Let $\mathsf{G}^{\mathrm{up}}_{q=1}(\widetilde{b}) \in \C[U^-(w)]$ be the dual canonical basis element corresponding to $\sf{\widetilde{M}}$. Then, in $\C[U^-(w)]_{D_w}$, we obtain 
$$
\mathsf{G}^{\mathrm{up}}_{q=1}(\widetilde{b}) = D_{w \lambda, \lambda} \cdot \mathsf{G}^{\mathrm{up}}_{q=1}(b),
$$
which proves the claim.
\end{proof}

Note that \eqref{equ:QwMlaM} implies that 
$\widetilde{\mathsf{M}}= \mathsf{M}_{w \lambda, \lambda} \nabla \mathsf{M}$ and  
$\mathsf{M}_{w \lambda, \lambda} \circ \mathsf{M}$ are isomorphic in $\mathscr{C} /{\rm Ker} (\mathsf{Q}_w)$. Hence, we obtain
$$
\overline{\mathsf{G}^{\mathrm{up}}_{q=1}(\tilde b)} =\overline{D_{w \lambda, \lambda}} \cdot \overline{\mathsf{G}^{\mathrm{up}}_{q=1}(b)} \quad \text{in } \ \C[U^-]/{\rm Ker} (\iota^*_w).
$$ 
It follows that 
\begin{equation}
\begin{aligned}
(\bar \pi(w)^*\circ [\mathsf{P}_w]) (\overline{\mathsf{G}^{\mathrm{up}}_{q=1}(b)}) 
&=\bar \pi(w)^* \left( [\mathsf{Q}_w(\mathsf{M})]\right) 
=\bar \pi(w)^* \left( \bigl[\mathsf{Q}_w(\mathsf{M}_{w\lambda,\lambda}^{\circ -1}) \circ\mathsf{Q}_w(\mathsf{\widetilde M}) \bigr] \right) \\
&=\bar \pi(w)^* \left(D_{w \lambda, \lambda}^{-1} \cdot \mathsf{G}^{\mathrm{up}}_{q=1}(\tilde b)\right) \\
&= \bigl( \overline{D_{w \lambda, \lambda}} \bigr)^{-1} \cdot \overline{\mathsf{G}^{\mathrm{up}}_{q=1}(\tilde b)}
=\overline{\mathsf{G}^{\mathrm{up}}_{q=1}(b)}.
\end{aligned}
\end{equation}
Thus
$\bar \pi(w)^*\circ [\mathsf{P}_w] $ is equal to the localization map
$$
\C[U^-]/{\rm Ker} (\iota^*_w) \to\left( \C[U^-]/{\rm Ker} (\iota^*_w) \right)_{\overline{D_w}}.
$$
In summary, we have the following commutative diagram.
\begin{equation}\label{equ_maindiagram1}
\xymatrix@C=5em{
\C[U^-(w)]\ar@{>->}[d]_{\pi(w)^*} \ar@{>->}[rr]^{[\mathsf{\Omega}_w]} & & \C[U^-(w)]_{D_w} \ar[d]_{\bar \pi(w)^*}^{\cong}  \\
\C[U^-]  \ar@{->>}[r]_{\iota_w^*} \ar[urr]^{[{\mathsf Q}_w]} 
&\left( \C[U^-]/{\rm Ker} (\iota_w^*) \right) \ar@{>->}[r] \ar[ur]_{[{\mathsf P}_w]} 
& \left( \C[U^-]/{\rm Ker} (\iota_w^*) \right)_{\overline{D_w}}
}
\end{equation}

\section{Weak Bruhat orders and the classification of finite mutation type quivers}\label{Sec_labledextendedexch}

We compare the Newton--Okounkov polytopes of $(G/P, \mathcal{L}_\lambda^K)$ and the monotone Lagrangian tori in $(G/P, \mathcal{L}_\lambda^K)$ constructed in Theorem~\ref{theorem_NOpolyopteforXWK} and Corollary~\ref{cor_monotoneLagtoriconst}. To facilitate this comparison, we introduce marked exchange matrices together with a partial order on them, designed to encode the non-negativity in Theorem~\ref{theorem_main6}. We investigate the relationship between the weak Bruhat order on Weyl groups and the order on marked exchange matrices. Finally, we classify the cluster algebras $\mathcal{A}_P$ according to their mutation types.

\subsection{Marked exchange matrices and partial order}

We begin by recalling various notions of finiteness in the context of cluster algebras.

\begin{definition}\label{def_clufininfin}
A cluster algebra is said to be of \emph{finite type} if it has only finitely many distinct seeds under iterated mutations. Otherwise, it is said to be of \emph{infinite type}. 
\end{definition}

Let $\mcal{F}$ denote the field of rational functions over $\C$ in $m$ variables ($m \coloneqq |J|$). An initial seed $\seed_0 = (A_{\seed_0}, \varepsilon_{\seed_0})$ consists of an $m$-tuple $A_{\seed_0}$ of algebraically independent variables in $\mcal{F}$ and an extended exchange matrix $\varepsilon_{\seed_0}$. By the classification theorem of \cite{FZ03}, a cluster algebra is of finite type if and only if its exchange matrix is mutation equivalent to one whose Cartan counterpart is a finite-type Cartan matrix. Equivalently, a cluster algebra is of finite type if and only if it has finitely many distinct unfrozen cluster variables.

We next recall the notion of finite mutation type.

\begin{definition} 
A cluster algebra is said to be of \emph{finite mutation type} if it admits only finitely many distinct extended exchange matrices under iterated mutations. Otherwise, it is said to be of \emph{infinite mutation type}. 
\end{definition}

If a cluster algebra is of infinite mutation type, then it is necessarily of infinite type. However, the converse does \emph{not} hold in general, i.e., a cluster algebra may be of infinite type while still being of finite mutation type. This distinction plays a crucial role in verifying the multiplicity condition in Theorem~\ref{theorem_maincriterion}. In particular, the verification reduces to detecting subquivers that are of finite mutation type but of infinite type.

We now focus on the cluster algebras $\mcal{A}_P$ in Definition~\ref{def_clusteralgforpartialflag}. To compare their mutation types, we introduce a partial order on extended exchange matrices. Let $J$ be the index set of an extended exchange matrix $\varepsilon$, with the partition $J = J_\mathrm{uf} \sqcup J_\mathrm{fz}$. Thus $\varepsilon = (\varepsilon_{r,s})_{(r,s) \in J_\mathrm{uf} \times J}$, where the principal part $(\varepsilon_{r,s})_{(r,s) \in J_\mathrm{uf} \times J_\mathrm{uf}}$ is skew-symmetrizable. 

\begin{definition}\label{def_parorderonskewsy}
Let $\varepsilon = (\varepsilon_{r,s})$ be an extended exchange matrix of size $J_{\varepsilon, \mathrm{uf}} \times J_{\varepsilon}$, and let $\varepsilon^\prime = (\varepsilon^\prime_{r,s})$ be an extended exchange matrix of size $J_{\varepsilon^\prime, \mathrm{uf}} \times J_{\varepsilon^\prime}$. We write 
\begin{equation}\label{equ_orderonex}
\varepsilon \leq \varepsilon^\prime
\end{equation}
if there exists a one-to-one map $\phi \colon J_{\varepsilon} \to J_{\varepsilon^\prime}$ such that 
\begin{enumerate}
\item $\phi ({J_{\varepsilon,\mathrm{uf}}}) \subseteq J_{\varepsilon^\prime,\mathrm{uf}}$, and 
\item $\varepsilon_{r,s} = \varepsilon^\prime_{\phi(r), \phi(s)}$ for all $(r,s) \in J_{\varepsilon, \mathrm{uf}} \times J_{\varepsilon}$.
\end{enumerate}
\end{definition}

\begin{remark}
Let $\mcal{A}$ and $\mcal{A}^\prime$ be cluster algebras. If $\mcal{A}$ is of infinite (mutation) type and there exist seeds $\seed$ of $\mcal{A}$ and $\seed^\prime$ for $\mcal{A}^\prime$ such that $\varepsilon_\seed \leq \varepsilon_{\seed^\prime}$, then $\mcal{A}^\prime$ is also of infinite (mutation) type.
\end{remark}

Gei{\ss}--Leclerc--Schr{\"{o}}er \cite{GLS08} introduced a partial order on the set of marked Dynkin diagrams to determine whether cluster algebras on unipotent cells associated with a parabolic subgroup are of finite or infinite type. Their ordering induces the order in \eqref{equ_orderonex} when appropriate seeds are chosen. 

To compare Newton–Okounkov bodies, however, we require a refinement that also records certain frozen variables, specifically those indexed by elements in~\eqref{equ_defIw}. For this reason, we introduce \emph{marked} extended exchange matrices.

\begin{definition}\label{def_labeledexexm}
A \emph{marked extended exchange matrix} is a pair $(\varepsilon, \mathsf{L})$, where $\varepsilon$ is an extended exchange matrix of size $J_\mathrm{uf} \times J$ and $\mathsf{L} \subseteq J_\mathrm{fz} \coloneqq J \setminus J_\mathrm{uf}$ is a distinguished subset of frozen indices. Elements of $\mathsf{L}$ are called \emph{marked frozen variables}. 

We define a partial order $\leq$
\begin{equation}\label{equ_partialorderlabel}
(\varepsilon, \mathsf{L}) \leq (\varepsilon^\prime, \mathsf{L}^\prime)
\end{equation}
if there is a one-to-one map $\phi \colon J_{\varepsilon} \to J_{\varepsilon^\prime}$ satisfying conditions $(1), (2)$ of Definition~\ref{def_parorderonskewsy} and 
\begin{enumerate}
\setcounter{enumi}{2}
\item $\phi(\mathsf{L}) \subseteq \mathsf{L}^\prime$.
\end{enumerate}
\end{definition}

When the principal part $\varepsilon |_{J_\mathrm{uf} \times J_\mathrm{uf}}$ is skew-symmetric, $\varepsilon$ can be represented by a quiver $\mathsf{Q}$. A vertex $v_j$ of $\mathsf{Q}$ is called \emph{unfrozen} (resp. \emph{frozen}) if $j \in J_\mathrm{uf}$ (resp. $j \in J_\mathrm{fz}$). A pair $(\mathsf{Q}, \mathsf{L})$ is called a \emph{marked quiver}. We write
$$
(\mathsf{Q}, \mathsf{L}) \leq (\mathsf{Q}^\prime, \mathsf{L}^\prime)
$$
if there exists an embedding $\phi \colon \mathsf{Q} \to \mathsf{Q}^\prime$ of quivers that sends unfrozen vertices to unfrozen vertices and marked vertices to marked vertices.

\begin{example}
Consider two marked quivers $(\mathsf{Q}, \mathsf{L})$ and $(\mathsf{Q}^\prime, \mathsf{L}^\prime)$ as depicted in Figure~\ref{fig_twolabeledquivers}. The index sets for $\mathsf{Q}$ and $\mathsf{Q}^\prime$ are $J = [9]$, and $J^\prime = [13]$, respectively. Let $J_{\mathrm{fz}} = \{5,7,8,9\}$, $J^\prime_{\mathrm{fz}} = \{7,9,11,12,13\}$, $\mathsf{L} =\{9\}$, and $\mathsf{L}^\prime =\{13\}$. Then the map
$$
\phi \colon J \to J^\prime \quad \mbox{given by $j \mapsto j+2$.}
$$
is an embedding of extended exchange matrices satisfying conditions $(1), (2)$ in Definition~\ref{def_parorderonskewsy}. However, the map $\phi$ does \emph{not} satisfy condition $(3)$ in Definition~\ref{def_labeledexexm}. Thus.
$$
\mathsf{Q} \leq \mathsf{Q}^\prime \,\, \mbox{ while $(\mathsf{Q}, \mathsf{L}) \not\leq (\mathsf{Q}^\prime, \mathsf{L}^\prime)$}.
$$
\begin{figure}[h]
$$
\resizebox{.5\hsize}{!}{
\xymatrix{
s_1 & & 2 \ar[rrd] 	&	&	& \fbox{8} \ar[lll]	&	&\\
s_2 & 1 \ar[rd] \ar[ru]	&  	& 4 \ar[ll]	\ar@/^1pc/[dd] &6 \ar[l] \ar[ru] \ar[rd]	&	&{\xy*{9}*\cir<7pt>{}\endxy}  \ar[ll] \\
s_3 & 	& 3 \ar[rru]	&	&	&\fbox{7} \ar[lll]	\\
s_4 &	&	&  \fbox{5} &  	&  (\mathsf{Q}, \mathsf{L})	& \\
 &	&	&   &  	&	& 	
}
}
\resizebox{.5\hsize}{!}{
\xymatrix{
s_1 & &2 \ar[rrrrd]	&  &  	&	&	& \fbox{12} \ar[lllll]	& \\
s_2 & 1\ar[rd] \ar[ru]&  	& 4 \ar[ll] \ar[rrd]	&	&	& 10 \ar[lll] \ar[rd] \ar[ru] &&{\xy*{13}*\cir<7pt>{}\endxy} \ar[ll] \\
s_3 &  	& 3 \ar[rd] \ar[ru] 	&  &6 \ar[ll] \ar@/^1pc/[dd] 	& 8 \ar[l] \ar[rd] \ar[ru]	& & \fbox{11} \ar[ll]&\\
s_4 & 	& & 5 \ar[rru]	&	& 	& \fbox{9}	\ar[lll]&\\
s_5 &	&	&  &  \fbox{7}	&	&	& 	(\mathsf{Q}^\prime, \mathsf{L}^\prime)	  &  
}
}
$$
\caption{The marked quivers $(\mathsf{Q}, \mathsf{L})$ and $(\mathsf{Q}^\prime, \mathsf{L}^\prime)$.} \label{fig_twolabeledquivers}
\end{figure}
\end{example}

We now explain how to associate a marked extended exchange matrix to a parabolic subgroup of a semisimple algebraic group over $\C$. Let $I$ denote the index set for the simple roots of $\frak{g}$ and let $w_0$ be the longest element of $W$. 

For a subset $K \subseteq I$, consider the corresponding parabolic subgroup $P^K$ and the quotient $W^K \simeq W / W_K$ where $W_K$ is the parabolic Weyl group. Let $w_K$ be the minimal length representative of the coset $w_0 W_K$ in $W$. 

Fix a reduced expression $\underline{w}_K = s_{i_1} s_{i_2} \cdots s_{i_m}$ of ${w}_K$. This choice determines a marked extended exchange matrix $(\varepsilon, \mathsf{L})$. The extended exchange matrix $\varepsilon = (\varepsilon_{r,s})$ is defined by the rule in~\eqref{equ_GLSseed} corresponding to $\underline{w}_K$. For each $k \in I \setminus K$, denote by $j_k \in J_\mathrm{fz}$ the index corresponding to the frozen variable with respect to the initial seed determined by $\underline{w}_K$. We then define the marking set as 
\begin{equation}\label{equ_labelL}
\mathsf{L} = \{ j_k \in J_\mathrm{fz} \mid k \in I \setminus K \}.
\end{equation}
This choice of the markings reflects the non-negativity of components of extended $g$-vectors of dual canonical basis element, cf. Theorem~\ref{theorem_main6}.

In the special case where the principal part $\varepsilon |_{J_\mathrm{uf} \times J_\mathrm{uf}}$ is skew-symmetric, $\varepsilon$ can be represented by a quiver $\mathsf{Q}$. By abuse of notation, we also write $\mathsf{L}$ for the corresponding set of marked vertices in $\mathsf{Q}$.

\begin{definition}\label{def_labeledexcmatass}
The pair $(\varepsilon, \mathsf{L})$ is called a \emph{marked exchange matrix associated with} the parabolic subgroup $P^K$. When $\varepsilon |_{J_\mathrm{uf} \times J_\mathrm{uf}}$ is skew-symmetric, the corresponding pair $(\mathsf{Q}, \mathsf{L})$ is called the \emph{marked quiver associated with} $P^K$.
\end{definition}

\begin{example}\label{example_OG28lab}
Let $G$ be a classical complex Lie group of type $D_4$ whose Dynkin digram is depicted below.  
$$
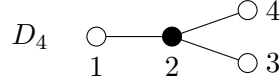
\begin{tikzpicture}[scale=0.5]
\draw (-1,0) node[anchor=east] {$D_{4}$};
\draw (0 cm,0) -- (2 cm,0);
\draw (2 cm,0) -- (4 cm,0.7 cm);
\draw (2 cm,0) -- (4 cm,-0.7 cm);
\draw[fill=white] (0 cm, 0 cm) circle (.25cm) node[below=4pt]{$1$};
\draw[fill=black] (2 cm, 0 cm) circle (.25cm) node[below=4pt]{$2$};
\draw[fill=white] (4 cm, 0.7 cm) circle (.25cm) node[right=3pt]{$4$};
\draw[fill=white] (4 cm, -0.7 cm) circle (.25cm) node[right=3pt]{$3$};
\end{tikzpicture}
$$
Take $K = \{1, 3, 4\}$, so that the corresponding partial flag variety $G/P^K$ is the orthogonal Grassmannian $\mathrm{OG}(2,8)$ of dimension $9$. A reduced expression of $w_K$ is $s_2 s_1s_3s_2s_4s_2s_3s_1s_2$. The marked quiver associated with $P^K$ is $(\mathsf{Q}, \mathsf{L})$ in Figure~\ref{fig_twolabeledquivers}. The frozen vertices are indicated by either a circle or a square, while marked frozen vertices are distinguished by circles.

Also, the marked quiver $(\mathsf{Q}^\prime, \mathsf{L}^\prime)$ in Figure~\ref{fig_twolabeledquivers} corresponds to the orthogonal Grassmannian $\mathrm{OG}(2, 10)$.
\end{example}

\subsection{Weak Bruhat orders and marked extended exchange matrices}

To compare the cluster algebras $\mcal{A}_P$ associated with parabolic subgroups, we introduced a partial order on exchange matrices in Definition~\ref{def_parorderonskewsy}. Since extended exchange matrices are determined by the Weyl group elements as in~\eqref{equ_GLSseed}, the partial order is closely related to the weak Bruhat order on the Weyl group, which we now recall. 

Let $W$ be a Weyl group, let $I$ be the index set for simple roots, and let $S = \{s_i \mid i \in I\}$ be the set of corresponding simple reflections. For $w \in W$, denote its length by $\ell(w)$. A reduced expression $\underline{w}$ then has the form 
$$
s_{i_1} s_{i_2} \cdots s_{i_m}, \, \mbox{where $m = \ell(w)$.}
$$ 

\begin{definition}
Let $v, w \in W$. The \emph{left weak Bruhat order} on $W$ is defined by
$$
v \leq_L w \quad \mbox{if $\ell(w) - \ell(v) = \ell(wv^{-1})$.}
$$
Equivalently, $w$ can be expressed as $s_{i_1} \cdots s_{i_\nu} v$ with $\ell(w) = \ell(v) + \nu$.

Similarly, the \emph{right weak Bruhat order} on $W$ is defined by
$$
v \leq_R w \quad \mbox{if $\ell(w) - \ell(v) = \ell(v^{-1}w)$.}
$$
Equivalently, $w$ can be expressed as $v s_{i_1} \cdots s_{i_\nu}$ with $\ell(w) = \ell(v) + \nu$.

Both $\leq_L$ and $\leq_R$ define partial orders on $W$.
\end{definition}

For a subset $K \subseteq I$, let 
\begin{itemize}
\item $W_{K}$ denote the parabolic subgroup generated by $\{ s_{{k}} \mid k \in K\}$, and
\item $\mathsf{D}_K$ denote the Dynkin diagram associated with $W_K$.
\end{itemize}
We also denote by
\begin{itemize}
\item $w^{K}_0$ the longest element of $W_K$, and
\item $w_{K}$ a (unique) minimal length representative of the left coset $w_0 W_{K}$ in $W$. 
\end{itemize}
Recall that $w_K = w_0 w^K_{0}$. Note that if $K = \emptyset$, then $w^{K}_0 = \mathrm{id}$ and $w_K = w_0$. 

We define an involution on $\mathsf{D}_K$ by
$$
*^K \colon \mathsf{D}_K \to \mathsf{D}_K \quad \mbox{ determined by $\alpha_{i^{*^K}} = -w_0^K \alpha_i$}
$$
For $K \subseteq M \subseteq I$, set
$$
K^{*^M} = \left\{ i^{*^M} \mid i \in K \right\}.
$$
When $M = I$, we simply write $K^{*}$ instead of $K^{*^M}$.

Although $w_0$ and $w_0^K$ do \emph{not} commute in general, we have the following lemma. 
\begin{lemma}\label{lemma_w0Kw0}
For every subset $K \subseteq I$, we have 
$$
w_0^K w_0 = w_0 w_0^{K^*}
$$
\end{lemma}

\begin{proof}
It follows from the identity $w_0 s_i w_0 = s_{i^*}$.
\end{proof}

Note that 
\begin{itemize}
\item $w_0^K w_0$ is the minimal length \emph{right} coset representative of $W_Kw_0$ in $W$, and 
\item $w_0 w_0^K$ is the minimal length \emph{left} coset representative of $w_0W_K$ in $W$. 
\end{itemize}
The next proposition is an immediate consequence of the construction of marked extended exchange matrices in Definition~\ref{def_labeledexcmatass}.

\begin{proposition}\label{proposition_weakBorderlabel}
If $w_{K} \leq_L w_{{K^\prime}}$ in the \emph{left} weak Bruhat order, then there exist reduced expressions $\underline{w}_{K}$ and  $\underline{w}_{{K^\prime}}$ such that the associated marked extended exchange matrices $(\varepsilon, \mathsf{L})$ and $(\varepsilon^\prime, \mathsf{L}^\prime)$, constructed from $\underline{w}_{K}$ and  $\underline{w}_{{K^\prime}}$ respectively, satisfy
$$
(\varepsilon, \mathsf{L}) \leq (\varepsilon^\prime, \mathsf{L}^\prime)
$$ 
in the sense of~\eqref{equ_partialorderlabel}.
\end{proposition}

Using Proposition~\ref{proposition_weakBorderlabel}, we compare the minimal length {left} coset representatives. We begin by comparing ${w}_{K_1}$ with ${w}_{{K_2}}$ within a fixed Weyl group $W$.

\begin{proposition}\label{proposition_weakbruhat}
Let $K_1, K_2 \subseteq I$.  If $K_1 \supseteq K_2$, then 
\begin{equation}\label{equ_propweakbruhat}
w^{{K_1}}_0 w_0 \leq_L w^{{K_2}}_0 w_0.
\end{equation}
Moreover,
\begin{equation}\label{equ_propweakbruhat2}
w_{K_1} \leq_L  w_{K_2}.
\end{equation}
\end{proposition}

\begin{proof}
Since $K_1 \supseteq K_2$, we have $w_{0}^{K_2} \leq_L w_{0}^{K_1}$. Observe that
$$
w^{{K_1}}_0 w_0 = w^{{K_1}}_0 w^{{K_2}}_0 w^{{K_2}}_0 w_0
$$
and compute
\begin{align*}
\ell ( w^{K_2}_{0} w_0 ) - \ell ( w^{K_1}_{0} w_0) &= \left( \ell(w_0) - \ell (w^{K_2}_{0}) \right) - \left( \ell(w_0) - \ell (w^{K_1}_{0}) \right) \\
&= \ell (w^{K_1}_{0}) - \ell (w^{K_2}_{0}) = \ell (w^{K_1}_{0} w^{K_2}_{0}) = \ell( w^{{K_1}}_0 w_0)^{-1} (w^{{K_2}}_0 w_0)).
\end{align*}
which establishes~\eqref{equ_propweakbruhat}. 

Since $K_1 \subseteq K_2$ implies $K_1^* \subseteq K_2^*$, applying~\eqref{equ_propweakbruhat} to the pair $(K_1^*, K_2^*)$ yields 
$$
w_{K_1} = w_0w_0^{K_1} = w_0^{K_1^*} w_0 \leq_L w_0^{K_2^*} w_0 = w_0 w_0^{K_2} = w_{K_2}.
$$
\end{proof}

Combining Propositions~\ref{proposition_weakBorderlabel} and~\ref{proposition_weakbruhat}, we obtain the following corollary.

\begin{corollary}\label{corollary_Dynkin1}
If $K_1 \supseteq K_2$, then there exist reduced expressions $\underline{w}_{K_1}$ and $\underline{w}_{K_2}$ such that the associated marked extended exchange matrices $(\varepsilon_1, \mathsf{L}_1)$ and $(\varepsilon_2, \mathsf{L}_2)$, constructed from $\underline{w}_{K_1}$ and $\underline{w}_{K_2}$, respectively, satisfy
$$
(\varepsilon_1, \mathsf{L}_1) \leq (\varepsilon_2, \mathsf{L}_2).
$$
\end{corollary}

We now consider two semisimple complex Lie algebras $\frak{g}_1$ and $\frak{g}_2$. For $i = 1,2$, let $I_i, W_i$, and $\mathsf{D}_i$ denote the index set, Weyl group, and Dynkin diagram, respectively. For $K_i \subseteq I_i$, let $\mathsf{D}_{K_i}$ be the corresponding Dynkin diagram determined by ${K_i}$. Set $J_i \coloneqq I_i \setminus K_i$ as in~\eqref{equ_labelL}. 

Suppose that there exists an embedding $\phi \colon \mathsf{D}_1 \to \mathsf{D}_2$ of Dynkin diagrams. Using $\phi$, any reduced expression $w = s_{i_1}\cdots s_{i_m} \in W_1$ can be viewed as a reduced  expression of $W_2$ by setting
$$
\phi(w) = s_{\phi(i_1)} s_{\phi(i_2)} \cdots s_{\phi(i_m)}.
$$
This enables us to extend the \emph{weak Bruhat order} to elements $w_1 \in W_1$, $w_2 \in W_2$ by declaring
\begin{equation}\label{equ_weakBruhatextended}
w_1 \leq_L w_2 \mbox{ (resp. $w_1 \leq_R w_2$)} \mbox{ if and only if $\phi(w_1) \leq _L w_2$} \mbox{ (resp. $\phi(w_1) \leq_R w_2$)}
\end{equation}

\begin{proposition}\label{proposition_weakbruhat2}
Let $W_i$, $I_i$, and $K_i \subseteq I_i$ be as above. Suppose that there exists an embedding of Dynkin diagrams
$$
\phi \colon \mathsf{D}_1 \to \mathsf{D}_2 \quad \mbox{ such that $\phi(J_1) \subseteq J_2$}.
$$
where $J_i \coloneqq I_i \setminus K_i$. Then, regarding $w_{K_1}$ as an element of $W_2$ via $\phi$, we have
$$
w_{K_1} \leq_L w_{K_2}.
$$
\end{proposition}

To prove Proposition~\ref{proposition_weakbruhat2}, we prepare some lemmas. For $J \subseteq K \subseteq I$, denote by $X^K_J$ the set of minimal length right coset representatives of $W_J$ in $W_K$. If $K = I$, then we write $X_J$ instead of $X^I_J$. Note that
$$
X^J_{K} = \{ w \in W_J \mid w \leq_R w^J_0 w^K_0 \}.
$$
We also set $X_J^{-1} \coloneqq \{ w^{-1} \mid w \in X_J \}$.

\begin{lemma}[Mackey decomposition]\label{lemma_Mackeydecomposition}
For $J \subseteq K \subseteq I$, we have
$$
X_K = \bigsqcup_{d \in X_K \cap X^{-1}_J} d \cdot X^J_{K^d \cap J}, \,\, \mbox{where $K^d = \{ d^{-1} s_i d \mid i \in K \}$}.
$$
\end{lemma}

\begin{proof}
See Lemma 2.1.9 in \cite{GP00}.
\end{proof}

\begin{lemma}\label{lemma_conseqMackey}
Let $J, K \subseteq I$. Then
$$
w_0^{K \cap J} w_0^J   \leq_R w_0^{K} w_0^I .
$$
In particular, if $J \subseteq I_1 \subseteq I_2$, then
$$
w_0^{I_1 \setminus J}  w_0^{I_1}  \leq_R w_0^{I_2 \setminus J}w_0^{I_2}.
$$
\end{lemma}

\begin{proof}
For $J,K \subseteq I$, observe that
$$
X_K = \bigl\{ w \in W_I \mid w \leq_R  w_0^K w_0^I \bigr\} \mbox{ and } X^J_{K \cap J} = \bigl\{ w \in W_J \mid w \leq_R w_0^{K \cap J} w_0^J \bigr\}.
$$
By the Mackey decomposition (Lemma~\ref{lemma_Mackeydecomposition}), it follows that 
$$
w_0^{K \cap J} w_0^J   \leq_R w_0^{K} w_0^I.
$$
\end{proof}

\begin{proof}[Proof of Proposition~\ref{proposition_weakbruhat2}]
Without loss of generality, after reindexing via $\phi$ if necessary, we may assume that $K_1 \subseteq I_1 \subseteq I_2$. Then, by Lemma~\ref{lemma_conseqMackey},
$$
w_0^{K_1} w_0^{I_1} = w_0^{I_1 \backslash J_1} w_0^{I_1} \leq_R w_0^{I_2 \backslash J_1} w_0^{I_2}.
$$
Taking inverses, we have
$$
w_0^{I_1} w_0^{K_1} = w_0^{I_1} w_0^{I_1 \backslash J_1}  \leq_L  w_0^{I_2} w_0^{I_2 \backslash J_1} \leq_L  w_0^{I_2} w_0^{I_2 \backslash J_2} = w_0^{I_2} w_0^{K_2}. 
$$
where the second inequality $\leq_L$ follows from Proposition~\ref{proposition_weakbruhat}.
\end{proof}

Consider the set of marked Dynkin diagrams, that is, 
\begin{equation}\label{equ_DJ}
\left\{ (\mathsf{D}, J) \mid \mbox{$\mathsf{D}$ is a Dynkin diagram of type $A, D$, or $E$, and $J \subseteq I$} \right\}.
\end{equation}
Motivated by Proposition~\ref{proposition_weakbruhat2}, we introduce a partial order on the set~\eqref{equ_DJ} by
\begin{equation}\label{equ_Dynkindia}
(\mathsf{D}_1, J_1) \leq (\mathsf{D}_2, J_2) 
\end{equation}
if there exists an embedding of Dynkin diagrams
$$
\phi \colon \mathsf{D}_1 \to \mathsf{D}_2 \quad \mbox{ such that $\phi(J_1) \subseteq J_2$}.
$$

Combining Propositions~\ref{proposition_weakbruhat} and~\ref{proposition_weakbruhat2}, we obtain the following corollary.

\begin{corollary}\label{corollary_Dynkin2}
Suppose that $(\mathsf{D}_1, J_1)$ and $(\mathsf{D}_2, J_2)$ are marked Dynkin diagrams satisfying
$$
(\mathsf{D}_1, J_1) \leq (\mathsf{D}_2, J_2)
$$ 
with respect to~\eqref{equ_Dynkindia}. Then there exist reduced expressions $\underline{w}_{K_1}$ and $\underline{w}_{K_2}$ of $w_{K_1}$ and $w_{K_2}$, respectively, such that the associated marked extended exchange matrices $(\varepsilon_1, \mathsf{L}_1)$ and $(\varepsilon_2, \mathsf{L}_2)$ satisfy
$$
(\varepsilon_1, \mathsf{L}_1) \leq (\varepsilon_2, \mathsf{L}_2).
$$
\end{corollary}

\subsection{Classification of quivers of finite mutation type}

Consider the following family of cluster algebras$\colon$
\begin{equation}\label{equ_paraboliccluade}
\{ \mathcal{A}_P \mid \mbox{$P$ is a parabolic subgroup of a simple algebraic group of type $A, D$, or $E$} \},
\end{equation} 
where $\mathcal{A}_P$ is defined in Definition~\ref{def_clusteralgforpartialflag}. In this subsection, we classify those cluster algebras in this family that are of finite/infinite mutation type.

We fix an ordering of the vertices of the Dynkin diagrams of the simply laced types. 
\begin{equation}\label{equ_DynkinAD}
\begin{tikzpicture}[scale=0.4]
\draw (-1,0) node[anchor=east] {$A_{n}$};
\draw (0 cm,0) -- (8 cm,0);
\draw[fill=white] (0 cm, 0 cm) circle (.25cm) node[below=4pt]{$1$};
\draw[fill=white] (2 cm, 0 cm) circle (.25cm) node[below=4pt]{$2$};
\draw[fill=white] (4 cm, 0 cm) circle (.25cm) node[below=4pt]{$3$};
\draw[fill=white] (6 cm, 0 cm) circle (.25cm) node[below=4pt]{$\cdots$};
\draw[fill=white] (8 cm, 0 cm) circle (.25cm) node[below=4pt]{$n$};
\end{tikzpicture}
\quad \quad \quad
\begin{tikzpicture}[scale=0.4]
\draw (-1,0) node[anchor=east] {$D_{n}$};
\draw (0 cm,0) -- (6 cm,0);
\draw (6 cm,0) -- (8 cm,0.7 cm);
\draw (6 cm,0) -- (8 cm,-0.7 cm);
\draw[fill=white] (0 cm, 0 cm) circle (.25cm) node[below=4pt]{$1$};
\draw[fill=white] (2 cm, 0 cm) circle (.25cm) node[below=4pt]{$2$};
\draw[fill=white] (4 cm, 0 cm) circle (.25cm) node[below=4pt]{$\cdots$};
\draw[fill=white] (6 cm, 0 cm) circle (.25cm) node[below=4pt]{$n-2$};
\draw[fill=white] (8 cm, 0.7 cm) circle (.25cm) node[right=3pt]{$n$};
\draw[fill=white] (8 cm, -0.7 cm) circle (.25cm) node[right=3pt]{$n-1$};
\end{tikzpicture}
\end{equation}
\begin{equation*}
\begin{tikzpicture}[scale=0.35]
\draw (0,2) node[anchor=east] {$E_{6}$};
\draw (0 cm,0) -- (8 cm,0);
\draw (4 cm,0) -- (4 cm,2 cm);
\draw[fill=white] (0 cm, 0 cm) circle (.25cm) node[below=4pt]{$1$};
\draw[fill=white] (2 cm, 0 cm) circle (.25cm) node[below=4pt]{$3$};
\draw[fill=white] (4 cm, 0 cm) circle (.25cm) node[below=4pt]{$4$};
\draw[fill=white] (4 cm, 2 cm) circle (.25cm) node[left=4pt]{$2$};
\draw[fill=white] (6 cm, 0 cm) circle (.25cm) node[below=4pt]{$5$};
\draw[fill=white] (8 cm, 0 cm) circle (.25cm) node[below=4pt]{$6$};
\end{tikzpicture}
\quad  
\begin{tikzpicture}[scale=0.35]
\draw (0,2) node[anchor=east] {$E_{7}$};
\draw (0 cm,0) -- (10 cm,0);
\draw (4 cm,0) -- (4 cm,2 cm);
\draw[fill=white] (0 cm, 0 cm) circle (.25cm) node[below=4pt]{$1$};
\draw[fill=white] (2 cm, 0 cm) circle (.25cm) node[below=4pt]{$3$};
\draw[fill=white] (4 cm, 0 cm) circle (.25cm) node[below=4pt]{$4$};
\draw[fill=white] (4 cm, 2 cm) circle (.25cm) node[left=4pt]{$2$};
\draw[fill=white] (6 cm, 0 cm) circle (.25cm) node[below=4pt]{$5$};
\draw[fill=white] (8 cm, 0 cm) circle (.25cm) node[below=4pt]{$6$};
\draw[fill=white] (10 cm, 0 cm) circle (.25cm) node[below=4pt]{$7$};
\end{tikzpicture}
\quad  
\begin{tikzpicture}[scale=0.35]
\draw (0,2) node[anchor=east] {$E_{8}$};
\draw (0 cm,0) -- (12 cm,0);
\draw (4 cm,0) -- (4 cm,2 cm);
\draw[fill=white] (0 cm, 0 cm) circle (.25cm) node[below=4pt]{$1$};
\draw[fill=white] (2 cm, 0 cm) circle (.25cm) node[below=4pt]{$3$};
\draw[fill=white] (4 cm, 0 cm) circle (.25cm) node[below=4pt]{$4$};
\draw[fill=white] (4 cm, 2 cm) circle (.25cm) node[left=4pt]{$2$};
\draw[fill=white] (6 cm, 0 cm) circle (.25cm) node[below=4pt]{$5$};
\draw[fill=white] (8 cm, 0 cm) circle (.25cm) node[below=4pt]{$6$};
\draw[fill=white] (10 cm, 0 cm) circle (.25cm) node[below=4pt]{$7$};
\draw[fill=white] (12 cm, 0 cm) circle (.25cm) node[below=4pt]{$8$};
\end{tikzpicture}
\end{equation*}

Gei{\ss}--Leclerc--Schr{\"{o}}er classified the cluster algebras of finite and infinite type in~\eqref{equ_paraboliccluade}. To reduce the number of cases to be examined, they compared certain reduced expressions and associated quivers with respect to the partial order~\eqref{equ_orderonex}. Their classification of finite type up to Dynkin diagram symmetry is summarized in \cite[Table 11.1]{GLS08}\footnote{We note that the case of type $A_5$ with $I \setminus K = \{1,3,4\}$ is missing from the table in \cite{GLS08}.}. 

\begin{table}[h]
\begin{center}
\begin{tabular}{| c | c | c | }
 \hline
 Type & $I \setminus K$ & $\mcal{A}_P$ \\ 
 \hline  \hline
$A_n (n \geq 2)$ & $\{1\}$ &  \\   
$A_n (n \geq 2)$ & $\{2\}$ & $A_{n-2}$  \\  
$A_n (n \geq 2)$ & $\{1, 2\}$ & $A_{n-1}$  \\  
$A_n (n \geq 2)$ & $\{1, n\}$ & $(A_1)^{n-1}$  \\  
$A_n (n \geq 3)$ & $\{1, n-1\}$ & $A_{2n-4}$  \\  
$A_n (n \geq 3)$ & $\{1, 2, n\}$ & $A_{2n-3}$  \\ 
\hline 
$A_4$ & $\{2,3\}$ & $D_4$  \\   
$A_4$ & $\{1,2,3\}$ & $D_5$  \\   
$A_4$ & $\{1,2,3,4\}$ & $D_6$  \\   
\hline 
$A_5$ & $\{3\}$ & $D_4$  \\   
\hline 
\end{tabular}
 \quad
\begin{tabular}{| c | c | c | }
 \hline
 Type & $I \setminus K$ & $\mcal{A}_P$ \\ 
 \hline  \hline
$A_5$ & $\{1, 3\}$ & $E_6$  \\   
$A_5$ & $\{2, 3\}$ & $E_6$  \\   
$A_5$ & $\{1, 2, 3\}$ & $E_7$  \\   
$A_5$ & $\{1, 3, 4\}$ & $E_8$  \\   
\hline 
$A_6$ & $\{3\}$ & $E_6$  \\   
$A_6$ & $\{2, 3\}$ & $E_8$  \\   
\hline 
$A_7$ & $\{3\}$ & $E_8$  \\   
\hline
$D_n (n \geq 4)$ & $\{1\}$ & $(A_1)^{n-2}$  \\   
\hline
$D_4$ & $\{n-1, n\}$ & $A_5$  \\   
\hline
$D_5$ & $\{n\}$ & $A_5$  \\   
\hline
\end{tabular}
\end{center}
\caption{The list of the cluster algebras $\mcal{A}_P$ of finite type and their Lie type}
\label{table_listfiniteparbolic}
\end{table}

The main theorem to be verified in Sections~\ref{sec_distinguishing} is stated below.

\begin{theorem}\label{theorem_divergeentries}
Every cluster algebra $\mcal{A}_P$ not listed in Table~\ref{table_listfiniteparbolic} satisfies the multiplicity property in Theorem~\ref{theorem_maincriterion}. Namely, there exist 
\begin{enumerate}
\item a sequence $\left( \seed_\ell \right)_{\ell \in \mathbb{N}}$ of seeds, 
\item an unfrozen index $r \in J_\mathrm{uf}$, and 
\item a frozen index $s \in J_\mathrm{fz}$ corresponding to the simple root $\alpha_i$ for some $i \in I \setminus K$, 
\end{enumerate}
such that the sequence of the $(r, s)$-entries of the extended exchange matrices $\varepsilon^\ell$ satisfies
$$
\lim_{\ell \to \infty} \varepsilon^\ell_{r, s} = - \infty.
$$
\end{theorem}

For our comparison of Newton--Okounkov bodies, we must also take into account the marking and the  partial order respecting this marking in~\eqref{equ_partialorderlabel}. Let $w_P$ (resp. $w_{P^\prime})$ denote the minimal length representative of the left coset $w_0 W_P$ (resp. $w_0 W_{P^\prime}$) in $W$. If $w_P \leq_L w_{P^\prime}$ and $\mcal{A}_P$ satisfies the multiplicity property, then $\mcal{A}_{P^\prime}$ also satisfies it. Following the strategy of \cite{GLS08} with respect to the new partial order~\eqref{equ_partialorderlabel}, we compare the corresponding cluster algebras. See Remark~\ref{rmk_arbitraryaddfrozen} for a discussion of issues arising from the presence of markings.

Let $\mathcal{W}$ be the set of the Weyl group elements $w_P$ associated with $\mcal{A}_P$ and consider the left weak Bruhat order on $\mathcal{W}$ as defined in~\eqref{equ_weakBruhatextended}. By Corollaries~\ref{corollary_Dynkin1} and~\ref{corollary_Dynkin2}, it suffices to compare the associated marked Dynkin diagrams with respect to the partial order~\eqref{equ_Dynkindia}. This leads to the following proposition.

\begin{proposition}\label{proposition_minimalelements}
Let $\mathscr{I}$ be the set of cluster algebras of infinite type in~\eqref{equ_paraboliccluade}. Then, up to Dynkin diagram symmetry, the entries in Table~\ref{table_minimalpara} are precisely the minimal elements of $\mathscr{I}$ with respect to the partial order~\eqref{equ_Dynkindia}.
\end{proposition}
\begin{table}[h]
\begin{center}
\begin{tabular}{| c | c |  }
 \hline
 Type & $I \setminus K$ \\ 
\hline  \hline
$A_n (n \geq 5)$ & $\{2, n-1\}$   \\   
\hline  
$A_5$ & $\{1, 3, 5\}$  \\
\hline     
$A_6$ & $\{1,3\}$   \\   
$A_6$ & $\{1,4\}$   \\   
$A_6$ & $\{3,4\}$   \\   
\hline 
$A_7$ & $\{4\}$   \\   
$A_7$ & $\{1,5\}$   \\   
$A_7$ & $\{2,3\}$   \\   
\hline
$A_8$ & $\{3\}$   \\   
\hline
\end{tabular}
\quad \quad
\begin{tabular}{| c | c |  }
 \hline
 Type & $I \setminus K$ \\ 
\hline  \hline
$D_n (n \geq 4)$ & $\{2\}$  \\   
\hline
$D_4$ & $\{1,3,4\}$  \\
\hline
$D_5$ & $\{1,5\}$  \\   
$D_5$ & $\{4,5\}$  \\   
\hline
$D_6$ & $\{6\}$  \\   
\hline
$E_6$ & $\{1\}$  \\   
$E_6$ & $\{2\}$  \\
\hline
$E_7$ & $\{7\}$  \\
\hline
$E_8$ & $\{8\}$  \\
\hline
\end{tabular}
\end{center}
\caption{Minimal cases of infinite mutation type}
\label{table_minimalpara}
\end{table}

As illustrated in Table~\ref{table_minimalpara}, the left weak Bruhat order significantly reduces the number of cases that require verification. Nevertheless, \emph{infinitely many} possibilities still remain. In the next part, we further reduce the problem to a finite number of cases by means of a series of lemmas.

We begin with type $A_n$.

\begin{lemma}\label{lemma_an1}
For each integer $n \geq 5$, let $I_n$ be the index set of simple roots for $\frak{sl}_{n+1}(\C)$, and let $K_{n} = I_n \setminus \{2, n-1\}$. Let $w_{K_n}$ be the minimal length representative of the coset $w_0 W_{K_n}$ in the Weyl group $W_n$ of $\frak{sl}_{n+1}(\C)$. Then, regarding $w_{K_{n}}$ as an element of $W_{n+1}$, 
\begin{equation}\label{exp_an1}
w_{K_{n+1}} = s_{n} s_{n+1} w_{K_n} s_{n+1} s_n.
\end{equation}
\end{lemma}

\begin{proof}
This follows from the expression
$$
w_{K_{n}} = 
\left(\begin{array}{cc|ccc|cc}
1 & 2 & 3 & \cdots & n-1 & n& n+1 \\
n  & n+1 & 3 & \cdots & n-1 & 1 & 2
\end{array}\right), \,\,
$$
where we identify the Weyl group $W_n$ of $\frak{sl}_{n+1}(\C)$ with the symmetric group $\frak{S}_{n+1}$. 
\end{proof}

By~\eqref{exp_an1}, any reduced expression of $w_{K_{n}}$ induces a reduced expression of $w_{K_{n+1}}$. Using these reduced expressions, we construct the marked quivers  $(\mathsf{Q}_n, \mathsf{L}_n)$ and $(\mathsf{Q}_{n+1}, \mathsf{L}_{n+1})$ associated with $K_{n}$ and $K_{n+1}$, respectively. Note that the marking $\mathsf{L}_n$ consists of two vertices$\colon$ one corresponding to the frozen variable associated with $s_2$ and the other corresponding to $s_n$. We introduce an alternative marking for $\mathsf{Q}_n \colon$
$$
\mathsf{L}^\prime_n \coloneqq \{ \mbox{the vertex corresponding to the frozen variable associated with $s_2$}\} \subsetneq \mathsf{L}_n.
$$
Lemma~\ref{lemma_an1} yields the following corollary.

\begin{corollary}\label{cor_an1}
Let $(\mathsf{Q}_n, \mathsf{L}^\prime_n)$ and $(\mathsf{Q}_{n+1}, \mathsf{L}^\prime_{n+1})$ be as above. Then, for each $n \geq 5$, 
$$
(\mathsf{Q}_n, \mathsf{L}^\prime_n) \leq (\mathsf{Q}_{n+1}, \mathsf{L}^\prime_{n+1})
$$ 
in the sense of~\eqref{equ_partialorderlabel}.
\end{corollary}

We now turn to type $D_n$. Let $\epsilon_i$ be the projection to the $i$-th diagonal entry. The simple roots are given by (cf.~\eqref{equ_DynkinAD})$\colon$
$$
\alpha_1 = \epsilon_1 - \epsilon_2, \cdots, \alpha_{n-1} = \epsilon_{n-1} - \epsilon_n, \alpha_n = \epsilon_{n-1} + \epsilon_n.
$$
Let $I_n$ be the index set of simple roots for $\frak{so}_{2n}(\C)$. The Weyl group of $\frak{so}_{2n}(\C)$ is denoted by $W_{n}$ isomorphic to $(\Z / 2\Z)^{n-1} \rtimes \mathfrak{S}_n$, the group of even signed permutation group. We identify $W_{n}$ with a subgroup of $\mathfrak{S}_{2n}$ via the embedding $W_{n} \hookrightarrow \mathfrak{S}_{2n}$ determined by
$$
\begin{cases}
s_j \mapsto s_j s_{2n-j} \, &\mbox{ for $j = 1, \dots, n-1$} \\
s_n  \mapsto s_n s_{n-1} s_{n+1} s_n. &
\end{cases}
$$
Under this embedding, the negative signed index $\overline{j}$ in $W_{n}$ corresponds to $2n+1-j$ in $\mathfrak{S}_{2n}$.

For each integer $n \geq 4$, let $K_{n} = I_n \setminus \{2\}$. Then the element $w_{K_{n}}$ can be written in permutation notation as
$$
w_{K_{n}} = 
\left(\begin{array}{cc|ccc|cc}
1 & 2 & 3 & \cdots & 2n-2 & 2n-1& 2n \\
2n-1  & 2n & 3 & \cdots & 2n-2 & 1 & 2
\end{array}\right). \,\,
$$
We introduce the \emph{shift map} $\sigma$ as follows. The shift map is defined by the homomorphism $\sigma \colon \frak{S}_{2n} \to \frak{S}_{2n+2}$$\colon$ for $w \in \frak{S}_{2n}$ with
$$
w = 
\left(\begin{array}{ccccc}
1 & 2 & \cdots &  2n-1& 2n \\
w(1)  & w(2)  & \cdots & w(2n-1) & w(2n)
\end{array}\right), \,\,
$$ 
we set
$$
\sigma(w) = 
\left(\begin{array}{cccccc}
1& 2 & 3 & \cdots &  2n+1& 2n+2 \\
1 & w(1) +1 & w(2) + 1 & \cdots & w(2n) +1 & 2n+2
\end{array}\right). \,\,
$$
Equivalently, $\sigma$ is determined by $\sigma(s_i) = s_{i+1}$ for all simple transpositions $s_i$. Therefore, it induces a well-defined homomorphism
$$
\sigma \colon W_{n} \to W_{{n+1}}.
$$

We then obtain the following lemma.

\begin{lemma}\label{lemma_dn1}
For each integer $n \geq 4$, with $I_n, K_n$, and $w_{K_n}$ defined as above, we have 
\begin{equation}\label{exp_dn1}
w_{K_{n+1}} = s_{2} s_{1} \sigma(w_{K_{n}}) s_{1} s_2.
\end{equation}
\end{lemma}

By~\eqref{exp_dn1}, any reduced expression of $w_{K_{n}}$ induces a reduced expression of $w_{K_{n+1}}$. Using these reduced expressions, we obtain the marked quivers  $(\mathsf{Q}_n, \mathsf{L}_n)$ and $(\mathsf{Q}_{n+1}, \mathsf{L}_{n+1})$ associated with $K_{n}$ and $K_{n+1}$, respectively. Note that the marking $\mathsf{L}_n$ consists of the single vertex, corresponding to the frozen variable associated with $s_{2}$. 

Although the quiver $\mathsf{Q}_n$ is embedded into $\mathsf{Q}_{n+1}$ via~\eqref{exp_dn1}, this embedding does \emph{not} preserve the marking$\colon$ the frozen variable associated with $s_{2}$ in $\mathsf{Q}_n$ is mapped to one associated with $s_{3}$ in $\mathsf{Q}_{n+1}$, whereas $\mathsf{L}_{n+1}$ consists of the vertex  corresponding to the frozen variable associated with $s_{2}$, (cf. Example~\ref{example_OG28lab}).

To resolve this issue, we find a suitable subquiver of $\mathsf{Q}_n$ that embeds into $\mathsf{Q}_{n+1}$ preserving the lebeling. Let $V_{\leq 3}$ denote the subset of vertices corresponding to the simple reflections $s_{1}, s_{{2}}$, or $s_{{3}}$, and set
$$
\mathsf{Q}_n^\prime \coloneqq \mathsf{Q}_n \big{|}_{V_{\leq 3}}.
$$
The following lemma provides the desired embedding.

\begin{lemma}\label{lemma_dn2}
For each integer $n \geq 5$, let $I_n, K_n$, and $w_{K_n}$ be as in Lemma~\ref{lemma_dn1}. Then the subquiver $\mathsf{Q}_n^\prime$ is isomorphic to $\mathsf{Q}^\prime_{n+1}$ and this isomorphism preserves the marking.
\end{lemma}

\begin{proof}
The Dynkin diagram of type $D_{n+1}$, when restricted to $V_{\leq 3}$, coincides with the Dynkin diagram of type $D_{n}$ restricted to the same subset. This follows from the expression
\begin{equation}\label{equ_wpn+1}
w_{K_{n+1}} =  s_{2} s_{1} s_{3} s_{2}  s_{4} s_{3} \, \sigma^{(3)} (w_{K_{n-2}}) \, s_{3}  s_{4} s_{2}  s_{3} s_{1} s_{2}
\end{equation}
which is obtained from~\eqref{exp_dn1}. In this expression, any reduced expression of $w_{K_{n-2}}$ contains none of the simple reflections $s_{1}, s_2, s_{3}$. By the construction of quivers in~\eqref{equ_GLSseed}, the subquiver $\mathsf{Q}^\prime_{n+1}$ is entirely determined by the sequence of $s_{1}, s_2$, and $s_{3}$ occurring before and after $w_{K_{n-2}}$ in~\eqref{equ_wpn+1}. Since the same pattern appears in $w_{K_n}$ for all $n \geq 5$, the conclusion follows. 
\end{proof}

By Lemma~\ref{lemma_dn2}, it suffices to show that $\mathsf{Q}_5^\prime$ has the multiplicity property in order to verify that each $\mathcal{A}_P$ of type $D_n$ with $n \geq 5$ has the multiplicity property.

In summary, thanks to Lemmas~\ref{lemma_an1} and~\ref{lemma_dn2}, it suffices to check only the finitely many cases listed in Table~\ref{table_minimalparbolic}. These cases will be examined in the next section.

\begin{table}[h]
\begin{center}
\begin{tabular}{| c | c |  }
 \hline
 Type & $I \setminus K$ \\ 
\hline  \hline
$A_5$ & $\{2, 4\}$   \\   
$A_5$ & $\{1, 3, 5\}$  \\
\hline     
$A_6$ & $\{1,3\}$   \\   
$A_6$ & $\{1,4\}$   \\   
$A_6$ & $\{3,4\}$   \\   
\hline
\end{tabular}
\quad \quad
\begin{tabular}{| c | c |  }
 \hline
 Type & $I \setminus K$ \\ 
\hline  \hline
$A_7$ & $\{4\}$   \\   
$A_7$ & $\{1,5\}$   \\   
$A_7$ & $\{2,3\}$   \\   
\hline
$A_8$ & $\{3\}$   \\   
\hline 
\end{tabular}
\quad \quad
\begin{tabular}{| c | c |  }
 \hline
 Type & $I \setminus K$ \\ 
\hline  \hline
$D_4$ & $\{2\}$  \\   
$D_4$ & $\{1,3,4\}$  \\
\hline
$D_5$ & $\{2\}$  \\   
$D_5$ & $\{1,5\}$  \\   
$D_5$ & $\{4,5\}$  \\   
\hline
\end{tabular}
\quad \quad
\begin{tabular}{| c | c |  }
 \hline
 Type & $I \setminus K$ \\ 
\hline  \hline
$D_6$ & $\{6\}$  \\   
\hline
$E_6$ & $\{1\}$  \\   
$E_6$ & $\{2\}$  \\
\hline
$E_7$ & $\{7\}$  \\
\hline
$E_8$ & $\{8\}$  \\
\hline
\end{tabular}
\end{center}
\caption{Minimal cases of infinite mutation type}
\label{table_minimalparbolic}
\end{table}

%----------------------------------------------------------------------------------------
\section{Detecting thresholds for appearance of infinite mutation type subquivers}\label{sec_distinguishing}

The goal of this section is to complete the proof of Theorem~\ref{theorem_distinguishGP}. We show that each case in Table~\ref{table_minimalparbolic} admits an arrow of arbitrarily large multiplicity. This is achieved by identifying an induced quiver of infinite mutation type that arises from a quiver of finite mutation type after adding a marked frozen vertex. 

Since we work with complex simple algebraic groups of simply laced type, the exchange matrix of the initial seed $\seed_0$ is skew-symmetric, so every extended exchange matrix $\varepsilon_{\seed}$ can be represented by a quiver. As a quiver determines a cluster algebra, we say that the quiver is of finite, infinite, finite mutation, or infinite mutation type according to the type of its associated cluster algebra.

Let $\mathsf{Q}_0$ be a quiver containing a subquiver $\mathsf{Q}^\prime$ of \emph{finite mutation type}. Let $\mathsf{V}_0$ (resp. $\mathsf{V}^\prime$) be the vertex set of $\mathsf{Q}_0$ (resp. $\mathsf{Q}^\prime$). Suppose that there is a vertex $f \in \mathsf{V}_0 \setminus \mathsf{V}^\prime$ such that the induced subquiver $\mathsf{Q} \coloneqq \mathsf{Q}_0 \big{|}_{\mathsf{V}^\prime \cup \{f\}}$ is of \emph{infinite mutation type}. Since $\mathsf{Q}^\prime$ has a finite mutation class, the mutation class of $\mathsf{Q}$ contains quivers  in which the multiplicity of arrows from $f$ to some unfrozen vertex becomes arbitrarily large. 

We summarize this observation in the following proposition.

\begin{proposition}[Lemma 6.2 in \cite{CKKP25}]\label{proposition_multiplicityquiver}
Let $\mathsf{Q}$ be a quiver with vertex set $\mathsf{V} = \mathsf{V}_\mathrm{fz} \sqcup \mathsf{V}_\mathrm{uf}$, and let $f \in \mathsf{V}$. Let $\varepsilon_\mathsf{Q}$ be its extended exchange matrix.
Set $\mathsf{V}^\prime \coloneqq \mathsf{V} - \{f\}$ and let $\mathsf{Q}^\prime = \mathsf{Q} |_{\mathsf{V}^\prime}$.
If $\mathsf{Q}$ is of infinite mutation type while $\mathsf{Q}^\prime$ is of finite mutation type, then 
\begin{equation} \label{eq:manyarrows2}
    \text{\parbox{35em}{
for each integer $\ell \geq 0$, there exists a matrix $\varepsilon_{\ell} =(\varepsilon^{\ell}_{i,j})_{ i \in \mathsf{V}^\prime,\, j \in \mathsf{V}}$ mutation equivalent to $\varepsilon_\mathsf{Q}$ such that
$- \varepsilon^{\ell}_{i_\ell,f} \geq \ell.$  
}}
\end{equation}
\end{proposition}

The main result of this section is stated below.

\begin{proposition}\label{prop_mainsection7}
Suppose that the cluster algebra $\mcal{A}_P$ associated with a parabolic subgroup $P$ of a complex simple algebraic group of simply laced type is of infinite type. Let $(\mathsf{Q}_P,\mathsf{L}_P)$ be the associated marked quiver. Then $(\mathsf{Q}_P,\mathsf{L}_P)$ has a subquiver $\mathsf{Q}$ and a marked frozen vertex $f \in \mathsf{L}_P$ satisfying ~\eqref{eq:manyarrows2}.
\end{proposition}

The proof of Proposition~\ref{prop_mainsection7} will be provided throughout this section. It is worth mentioning a corollary stating the special property of the cluster algebras $\mathcal{A}_P$.

\begin{corollary}
The cluster algebra $\mcal{A}_P$ is of finite type if and only if it is of finite mutation type. 
\end{corollary}

Assuming Proposition~\ref{prop_mainsection7}, we now complete the proof of Theorem~\ref{theorem_distinguishGP}.

\begin{proof}[Proof of Theorem~\ref{theorem_distinguishGP}]
Combining Proposition~\ref{prop_mainsection7} with Proposition~\ref{proposition_minimalelements}, we obtain Theorem~\ref{theorem_divergeentries}, which verifies the multiplicity condition of the criterion in Theorem~\ref{theorem_maincriterion}. The non-negativity condition is confirmed by Theorem~\ref{theorem_main6}. Therefore, we establish Theorem~\ref{theorem_distinguishGP} from  Theorem~\ref{theorem_maincriterion}.
\end{proof}

\begin{remark}
We conclude with a remark on the role of marking. The classification into finite or infinite type does not depend on the presence of frozen vertices. In particular, removing all frozen vertices from an infinite-type quiver $\mathsf{Q}_P$ still yields a quiver of infinite type. 

To prove Proposition~\ref{prop_mainsection7} using Proposition~\ref{proposition_multiplicityquiver}, we proceed in practice by first removing all frozen vertices and then reintroducing selected frozen vertices so as to obtain a quiver of infinite mutation type. It is important to emphasize that adding arbitrary frozen vertices does \emph{not} guarantee that the resulting quiver is of infinite mutation type. In contrast, a suitably chosen marked frozen vertex can  make the resulting quiver of infinite mutation type. See Remark~\ref{rmk_arbitraryaddfrozen} for an explicit example. 
\end{remark}

\subsection{Transition from finite to infinite mutation type by adding frozen variables}

The first ingredient in the proof of Proposition~\ref{prop_mainsection7} is the classification result of cluster algebras of finite mutation type due to Felikson--Shapiro--Tumarkin.

\begin{theorem}[Theorem 6.1 in \cite{FST12}]\label{theorem_FST61}
Any skew-symmetric matrix $(n \times n)$-matrix with $n \geq 3$ is either the adjacency matrix of triangulation of a bordered two-dimensional surface, or a matrix mutation equivalent to one of the following types$\colon$
$$
E_6, E_7, E_8, \widetilde{E}_6, \widetilde{E}_7, \widetilde{E}_8, E^{(1,1)}_6, E^{(1,1)}_7, E^{(1,1)}_8, X_6, X_7. 
$$
\end{theorem}

Our goal is to locate, within our marked quivers, a subquiver that is of finite mutation type but infinite type, cf. Proposition~\ref{proposition_multiplicityquiver}. For this purpose, the relevant cases are the following six families:
\begin{enumerate}
\item (Affine type) $\widetilde{A}_{p,q},  \widetilde{E}_6, \widetilde{E}_7, \widetilde{E}_8$ 
\item (Extended affine type) $E^{(1,1)}_7, E^{(1,1)}_8$
\end{enumerate}
Note that quivers of type $\widetilde{A}_{p,q}$ arise from triangulations of an annulus. 

A quiver of finite mutation type may become of infinite mutation type after the addition of a suitably chosen frozen vertex. Felikson--Tumarkin \cite{FT24} classified when this happens. The following lemma restates their criterion in a form adapted to our setting.

\begin{lemma}[Section 4 in \cite{FT24}] \label{lem:FTcrit}
Suppose that $\mathsf{Q}^\prime$ is a quiver of affine type. By Theorem~\ref{theorem_FST61}, the quiver $\mathsf{Q}^\prime$ is mutation equivalent to a quiver $\mathsf{Q}^\lozenge$ containing the diamond-shaped subquiver in Figure~\eqref{fig_quiverdiamond}, see $\mathsf{Q}^\lozenge$ in Figures~\ref{fig_quiverApq},~\ref{fig_quiverE6},~\ref{fig_quiverE71},~\ref{fig_quiverE72},~\ref{fig_quiverE8}. Let $\mathsf{Q}$ be the quiver obtained from $\mathsf{Q}^\lozenge$ by adding a single frozen vertex $f$ and set
\begin{align*}
&b_1 \coloneqq (\text{the number of arrows from $w_0$ to $f$}) -  \, (\text{the number of arrows from $f$ to $w_0$}), \\ 
&b_2 \coloneqq (\text{the number of arrows from $v_0$ to $f$}) -  \, (\text{the number of arrows from $f$ to $v_0$}).
\end{align*}
If either $b_1\neq -b_2$ or $b_2 > 0$, then $\mathsf{Q}$ is of mutation infinite type.
\vspace{-0.2cm}
\begin{figure}[h]
$$
\vspace{-0.1cm}
\resizebox{0.23\hsize}{!}{
\xymatrix{
&w_0 \ar[rd] 	&		\\
w_1 \ar[ru] \ar[rd]&	&v_{p-1} \ar[ld] 		\\
&v_0 \ar@<-.5ex>[uu] \ar@<.5ex>[uu] 		\\
}
}
$$
\caption{Diamond shape subquiver}
\label{fig_quiverdiamond}
\end{figure}
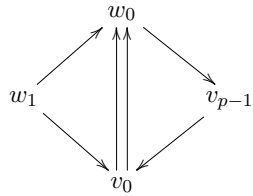
\end{lemma}

Using Lemma~\ref{lem:FTcrit}, we first deal with the affine types.

\begin{lemma}[see Lemma 6.5 in \cite{CKKP25}]\label{lemma_infinitefinite2}
For $p, q > 0$, let ${\mathsf{Q}^\prime}$ be the quiver of type $\widetilde{A}_{p,q}$ with all vertices unfrozen. Let $\mathsf{Q}$ be the quiver obtained from ${\mathsf{Q}^\prime}$ by adding a single frozen vertex $f$ and one arrow from $f$ to $w_0$ as in Figure~\ref{fig_quiverApq}. Then the following hold$\colon$
\begin{enumerate}
\item the quiver ${\mathsf{Q}^\prime}$ is of finite mutation type, and
\item the quiver $\mathsf{Q}$ is of infinite mutation type.
\end{enumerate}
\end{lemma}

In Figure~\ref{fig_quiverApq}, the quiver $\mathsf{Q}$ on the left is mutation equivalent to the quiver $\mathsf{Q}^\lozenge$ on the right.

\begin{figure}[h]
$$
\resizebox{1.03\hsize}{!}{
\xymatrix{
  \mathsf{Q} & v_0 \ar[r] \ar[ld] & v_1 \ar[dr] & \\
  w_{q-1} \ar[d]& & & v_2 \ar[d] \\
  \vdots \ar[d] & \cdots &  & \vdots \ar[d]  \\
  w_2 \ar[rd] & & & v_{p-1} \ar[ld]   \\
  & w_1 \ar[r] & w_0 & \fbox{f} \ar[l]   \\
} \quad \,
\xymatrix{
\mathsf{Q}^\lozenge &  	&		&	& 	&	&	&	&\\
&	&	&	&w_0 \ar[rd] 	& \fbox{f}\ar[l]	&	&	&\\
w_{q-1} \ar@{-}[r] &	\cdots &	w_2  \ar[r] \ar[l] 	&w_1 \ar[ru] \ar[rd]&	&v_{p-1} \ar[ld]	&v_{p-2} \ar[l] \ar[r]	&\cdots \ar@{-}[r]	& v_1	& \\
&	&	  	& 	&v_0 \ar@<-.5ex>[uu] \ar@<.5ex>[uu]    &	&	& 	&\\
}
}
$$
\caption{Quivers $\mathsf{Q}$ and $\mathsf{Q}^\lozenge$ in Lemma~\ref{lemma_infinitefinite2}} \label{fig_quiverApq}
\end{figure}

\begin{lemma}\label{lemma_infinitefinite6}
Let ${\mathsf{Q}^\prime}$ be the quiver of type $\widetilde{E}_6$ with all vertices unfrozen. Let $\mathsf{Q}$ be the quiver obtained from ${\mathsf{Q}^\prime}$ by adding one frozen vertex $f$ and one single arrow from $6$ to $f$ as in Figure~\ref{fig_quiverE6}. Then the following hold$\colon$
\begin{enumerate}
\item the quiver ${\mathsf{Q}^\prime}$ is of finite mutation type and
\item the quiver $\mathsf{Q}$ is of infinite mutation type.
\end{enumerate}
\end{lemma}

\begin{proof} 
The quiver $\mathsf{Q}^\prime$ is of type $\widetilde{E}_6$ and hence is of finite mutation type in Theorem~\ref{theorem_FST61}.
To show that $\mathsf{Q}$ is of infinite mutation type, consider the following sequence $\mu$ of mutations (applied from left to right)
\begin{equation*}\label{equ_seqofmut}
(6,1,7,3,2,1,7,3,2,4,7,5,2,4,1,6,5,7).
\end{equation*}
Let $\mathsf{Q}^\lozenge \coloneqq \mu(\mathsf{Q})$ denote the resulting quiver. Then $\mathsf{Q}^\lozenge$ contains a double arrow $3 \rightrightarrows 4$ together with an arrow $f \to 4$. By Lemma~\ref{lem:FTcrit}, $\mathsf{Q}$ is of infinite mutation type. 
\begin{figure}[h]
$$
\resizebox{0.9\hsize}{!}{
\xymatrix{
	 	\mathsf{Q}	&	&7 \ar[d]	&	\\
	 		&	&4	&	&	&		 \\
	  	1 \ar[r]	&2	&3 \ar[l]\ar[r] \ar[u]	&5	&6 \ar[l] \ar[r]	& \fbox{f} 	\\
}
\quad \quad
\xymatrix{
	 \mathsf{Q}^\lozenge	&	&4 \ar[rd] \ar[rrd] \ar[ld]	& \fbox{f}\ar[l]	&	\\
	5    	&2 \ar[l] \ar[rd]&	&1 \ar[ld]	&6 \ar[lu] \ar[lld]	&7 \ar[l]	\\
	  	& 	&3 \ar@<-.5ex>[uu] \ar@<.5ex>[uu]   &	&	\\
}
}
$$
\caption{Quivers $\mathsf{Q}$ and $\mathsf{Q}^\lozenge$ in Lemma~\ref{lemma_infinitefinite6}} \label{fig_quiverE6}
\end{figure}
\end{proof}

\begin{lemma}\label{lemma_infinitefinite1}
Let ${\mathsf{Q}^\prime}$ be the quiver of type $\widetilde{E}_7$ with all vertices unfrozen. Let $\mathsf{Q}$ be the quiver obtained from ${\mathsf{Q}^\prime}$ by adding one frozen vertex $f$ and one single arrow from $f$ to $4$ as in Figure~\ref{fig_quiverE71}. Then the following hold$\colon$
\begin{enumerate}
\item the quiver ${\mathsf{Q}^\prime}$ is of finite mutation type and
\item the quiver $\mathsf{Q}$ is of infinite mutation type.
\end{enumerate}
\end{lemma}

\begin{proof} 
Apply the following sequence $\mu$ of mutations to $\mathsf{Q}$ 
$$
(6,1,8,4,3,5,7,6,7,3,2,4,1,5,8,3,7,5,2).
$$ 
Let $\mathsf{Q}^\lozenge \coloneqq \mu(\mathsf{Q})$. By Lemma~\ref{lem:FTcrit} ($b_2 > 0$), $\mathsf{Q}$ is of infinite mutation type.
\begin{figure}[h]
$$
\resizebox{0.95\hsize}{!}{
\xymatrix{
	\mathsf{Q} 	&	&	&\fbox{f} \ar[d]&	&&\\
	 	&	&	&4	&	&&&\\
	8  	&1 \ar[l]\ar[r]	&2	&3 \ar[l]\ar[r] \ar[u]	&5	&6 \ar[l]\ar[r]	&7\\
}
\xymatrix{
	 \mathsf{Q}^\lozenge	&	&	&\fbox{f}\ar[ld] \ar[rrdd] \ar@/_1.0pc/[llldd]	&	&	&\\
	 	&	&4 \ar[rd] \ar[rrd] \ar[ld]	&	&	&	&\\
	2   \ar[r] 	&3 \ar[rd] \ar@/^1.0pc/[rruu]&	&1 \ar[ld]	&7 \ar[lld]	&5 \ar[l] \ar[r]	&6 \ar[llluu]	&\\
	  	& 	&8 \ar@<-.5ex>[uu] \ar@<.5ex>[uu] \ar[ruuu] &	&	& 	&\\
}
}
$$
\caption{Quivers $\mathsf{Q}$ and $\mathsf{Q}^\lozenge$ in Lemma~\ref{lemma_infinitefinite1}} \label{fig_quiverE71}
\end{figure}
\end{proof}

\begin{lemma}\label{lemma_infinitefinite3}
Let ${\mathsf{Q}^\prime}$ be the quiver of type $\widetilde{E}_7$ with all vertices unfrozen. Let $\mathsf{Q}$ be the quiver obtained from ${\mathsf{Q}^\prime}$ by adding one frozen vertex $f$ and one single arrow from $f$ to $7$ as in Figure~\ref{fig_quiverE72}. Then the following hold$\colon$
\begin{enumerate}
\item the quiver ${\mathsf{Q}^\prime}$ is of finite mutation type and
\item the quiver $\mathsf{Q}$ is of infinite mutation type.
\end{enumerate}
\end{lemma}

\begin{proof}
Apply the following sequence $\mu$ of mutations to $\mathsf{Q}$ 
\begin{equation*}
(6, 1, 8, 4, 3, 5, 7, 6, 7, 3, 2, 4, 1, 5, 8, 3, 7, 5, 2).
\end{equation*}
Let $\mathsf{Q}^\lozenge \coloneqq \mu(\mathsf{Q})$. By Lemma~\ref{lem:FTcrit}, $\mathsf{Q}$ is of infinite mutation type. 
\begin{figure}[h]
$$
\resizebox{1.0\hsize}{!}{
\xymatrix{
	\mathsf{Q} 	&	&	&4	&	&	&\fbox{f} \ar[d]	 \\
	8  	&1 \ar[l]\ar[r]	&2	&3 \ar[l]\ar[r] \ar[u]	&5	&6 \ar[l] \ar[r]	&7 \\
}
\quad
\xymatrix{
	 \mathsf{Q}^\lozenge	&	&4 \ar[rd] \ar[rrd] \ar[ld]	& \fbox{f}\ar[l] \ar[rrd]	&	&	&\\
	2   \ar[r]  	&3 \ar[rd]&	&1 \ar[ld]	&7 \ar[lu] \ar[lld]	&5 \ar[l] \ar[r]	&6 \ar[lllu]	&\\
	  	& 	&8 \ar@<-.5ex>[uu] \ar@<.5ex>[uu]   &	&	& 	&\\
}
}
$$
\caption{Quivers $\mathsf{Q}$ and $\mathsf{Q}^\lozenge$ in Lemma~\ref{lemma_infinitefinite3}} \label{fig_quiverE72}
\end{figure}
\end{proof}

\begin{lemma}\label{lemma_infinitefinite4}
Let ${\mathsf{Q}^\prime}$ be the quiver of type $\widetilde{E}_8$ with all vertices unfrozen. Let $\mathsf{Q}$ be the quiver obtained from ${\mathsf{Q}^\prime}$ by adding one frozen vertex $f$ and one single arrow from $f$ to $9$ as in Figure~\ref{fig_quiverE8}. Then the following hold$\colon$
\begin{enumerate}
\item the quiver ${\mathsf{Q}^\prime}$ is of finite mutation type and
\item the quiver $\mathsf{Q}$ is of infinite mutation type.
\end{enumerate}
\end{lemma}

\begin{proof}
Apply the following sequence $\mu$ of mutations to $\mathsf{Q}$ 
\begin{equation*}
(8,9,5,1,3,4,5,2,6,5,1,3,2,3,6,7,4,8,5,9,6,3,5)
\end{equation*}
Let $\mathsf{Q}^\lozenge \coloneqq \mu(\mathsf{Q})$. By Lemma~\ref{lem:FTcrit}, $\mathsf{Q}$ is of infinite mutation type. 
\begin{figure}[h]
$$
\resizebox{1.0\hsize}{!}{
\xymatrix{
	 	\mathsf{Q} 	&	&4	&	&	&	&	& \fbox{f} \ar[d] \\
	  	1 \ar[r]	&2	&3 \ar[l]\ar[r] \ar[u]	&5	&6 \ar[l] \ar[r]	&7	&8 \ar[r] \ar[l]	&9\\
}
\quad \,\,
\xymatrix{
	\mathsf{Q}^\lozenge	&	&4 \ar[rd] \ar[rrd] \ar[ld]	& \fbox{f}\ar[l]	&	&	&\\
	7    	&6 \ar[l] \ar[rd]&	&8 \ar[u] \ar[ld]	&3 \ar[lld]	&5 \ar[l] \ar[r]	&2 	&1 \ar[l] \\
	  	& 	&9 \ar@<-.5ex>[uu] \ar@<.5ex>[uu]   &	&	& 	&\\
}
}
$$
\caption{Quivers $\mathsf{Q}$ and $\mathsf{Q}^\lozenge$ in Lemma~\ref{lemma_infinitefinite4}} \label{fig_quiverE8}
\end{figure}
\end{proof}

Next, we deal with the extended affine types.

\begin{lemma}[Section 5 in \cite{FT24}] \label{lem:FTextcrit}
Suppose that $\mathsf{Q}^\prime$ be a quiver of extended affine type (e.g., Figure~\ref{fig_quiverE711},~\ref{fig_quiverE811}).
Let $\mathsf{Q}$ be the quiver obtained from $\mathsf{Q}^\prime$ by adding a single frozen vertex $f$.
If $\mathsf{Q}$ is connected, then $\mathsf{Q}$ is of infinite mutation type.
\end{lemma}

\begin{lemma}\label{lemma_infinitefinite7}
Let ${\mathsf{Q}^\prime}$ be the quiver of type ${E}^{(1,1)}_7$ with all vertices unfrozen. Let $\mathsf{Q}$ be the quiver obtained from ${\mathsf{Q}^\prime}$ by adding one frozen vertex $f$ and one single arrow from $8$ to $f$ as in Figure~\ref{fig_quiverE711}. Then the following hold$\colon$
\begin{enumerate}
\item the quiver ${\mathsf{Q}^\prime}$ is of finite mutation type and
\item the quiver $\mathsf{Q}$ is of infinite mutation type.
\end{enumerate}
\end{lemma}

\begin{proof}
Apply the following sequence $\mu$ of mutations to $\mathsf{Q}$ 
\begin{equation*}
(7, 2, 9, 4, 7, 3, 8, 1, 3, 6, 1).
\end{equation*}
Let $\mathsf{Q}^\lozenge \coloneqq \mu(\mathsf{Q})$. By Lemma~\ref{lem:FTcrit}, $\mathsf{Q}$ is of infinite mutation type. 
\begin{figure}[h]
$$
\resizebox{1.0\hsize}{!}{
\xymatrix{
\mathsf{Q}&&7 \ar[rd]&&&& \\
&3 \ar[ru] \ar[rd]&&8 \ar[ll] \ar[rd] &&& \\
2 \ar[ru] \ar[rd] &&5 \ar[rrrd] \ar[ru] \ar[ll] &&9 \ar[ll] &\fbox{f} \ar[d] & \\
&4 \ar[rrru] &&1 \ar[lu] &&6 \ar[ll] &
}
\quad
\xymatrix{
	 \mathsf{Q}^\lozenge	& &	&	& 	&		\\
	 	& &	&2 \ar[rd] \ar[rrd] \ar[ld]	& 	&  	&	& \fbox{f} \ar[d]		\\
	7 &   3 \ar[r] \ar[l]	&8 \ar[rd]&	&4 \ar[ld]	&5 \ar[lld]	&1 \ar[l] \ar[r]	&6 	\\
	  	& & 	&9 \ar@<-.5ex>[uu] \ar@<.5ex>[uu]    &	&	 	\\
}
}
$$
\caption{Quivers $\mathsf{Q}$ and $\mathsf{Q}^\lozenge$ in Lemma~\ref{lemma_infinitefinite7}} \label{fig_quiverE711}
\end{figure}
\end{proof}

\begin{lemma}\label{lemma_infinitefinite5}
Let ${\mathsf{Q}^\prime}$ be the quiver of type ${E}^{(1,1)}_8$ with all vertices unfrozen. Let $\mathsf{Q}$ be the quiver obtained from ${\mathsf{Q}^\prime}$ by adding one frozen vertex $f$ and one single arrow from $8$ to $f$ as in Figure~\ref{fig_quiverE811}. Then the following hold$\colon$
\begin{enumerate}
\item the quiver ${\mathsf{Q}^\prime}$ is of finite mutation type and
\item the quiver $\mathsf{Q}$ is of infinite mutation type.
\end{enumerate}
\end{lemma}

\begin{proof}
Apply the following sequence $\mu$ of mutations to $\mathsf{Q}$ \begin{equation*}
(1, 7, 2, 9, 4, 9, 5, 4, 10, 9, 5, 10, 7, 6, 4, 8, 10, 7, 4, 8, 5, 7, 4, 5, 4, 5, 7, 3).
\end{equation*}
Let $\mathsf{Q}^\lozenge \coloneqq \mu(\mathsf{Q})$. By Lemma~\ref{lem:FTcrit}, $\mathsf{Q}$ is of infinite mutation type. 
\begin{figure}[h]
$$
\resizebox{1.0\hsize}{!}{
\xymatrix{
	 	\mathsf{Q} &		1 \ar[d] \ar[r]	&2 \ar[d] \ar[r]	&4 \ar[d] \ar[r]	&6 \ar[d] \ar[r]	&8 \ar[d]			\\
&	  	3 \ar[r]	&5 \ar[lu] \ar[r]	&7 \ar[lu] \ar[r]	&9 \ar[lu] \ar[r]	&10 \ar[lu]	 	\\
&	  		&	&	&	&  \fbox{f} \ar[u]	 	\\
}
\quad \quad
\xymatrix{
	 \mathsf{Q}^\lozenge	&	&8 \ar[rd] \ar[rrd] \ar[r] \ar[ld]	& \fbox{f}	&	&	&\\
	7  \ar[r]    	&4 \ar[rd]&	&5 \ar[ld]	&6 \ar[lld]	&9 \ar[l] \ar[r]	&2 	&1 \ar[l] \ar[r] & 3 \\
	  	& 	&10 \ar@<-.5ex>[uu] \ar@<.5ex>[uu]    &	&	& 	&\\
}
}
$$
\caption{Quivers $\mathsf{Q}$ and $\mathsf{Q}^\lozenge$ in Lemma~\ref{lemma_infinitefinite5}} \label{fig_quiverE811}
\end{figure}
\end{proof}

\subsection{Existence of an arrow of arbitrary large multiplicity}

In this subsection, we complete the proof of Proposition~\ref{prop_mainsection7} using the lemmas established in the previous subsection.

We begin with the type $A$ case.

\begin{example}\label{example_A5K135}
Consider a classical group $G$ of type $A_5$ and let $K = \{1, 3, 5\}$. Let $P = P^K$ be the corresponding parabolic subgroup. Then $G/P \simeq \mathcal{F}\ell(2,4;6)$, which has dimension $m = 12$. The Weyl group $W$ of $G$ is isomorphic to the symmetric group $\mathfrak{S}_6$. The minimal length representative of the coset $w_0W_K$ is given by 
$$
w_{K} = 
\left(\begin{array}{cccccc}
1 & 2 & 3 & 4 & 5 & 6 \\
5  & 6 & 3 & 4 & 1 & 2 
\end{array}\right).
$$
A reduced expression of $w_{K}$ is 
$$
s_4 s_3 s_5 s_2 s_4 s_1 s_3 s_2 s_4 s_3 s_5 s_4.
$$ 
The corresponding marked quiver $(\mathsf{Q}_P, \mathsf{L}_P = \{8, 12\})$ is depicted in Figure~\ref{fig_quiverA5_K135} where the set of frozen vertices is $\mathsf{V}_\mathrm{fz} = \{6, 8, 10, 11, 12\}$.
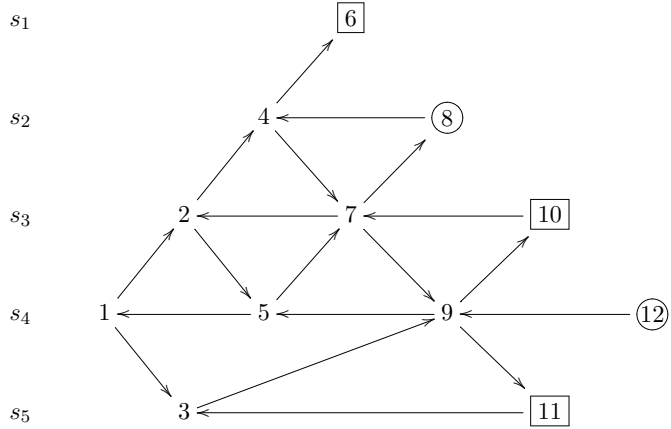
\begin{figure}[h]
$$
\resizebox{.7\hsize}{!}{
\xymatrix{
s_1 &	& 	&	&\fbox{6}	&	&	&&\\
s_2 &	& 	&4 \ar[ru] \ar[rd]	&	&{\xy*{8}*\cir<7pt>{}\endxy} \ar[ll]	&	&&&\\
s_3 &	&2 \ar[ru] \ar[rd]	&	&7 \ar[ll] \ar[ru] \ar[rd]	&	&\fbox{10} \ar[ll] &&\\
s_4 &1 \ar[ru] \ar[rd]	&	&5 \ar[ru] \ar[ll]	&	&9 \ar[ll] \ar[ru] \ar[rd]	&	&{\xy*{12}*\cir<7pt>{}\endxy} \ar[ll]\\
s_5 &	&3 \ar[rrru]	&	&	&	&\fbox{11} \ar[llll]	& 
}
}
$$
\caption{The marked quiver for $A_5$ type with $K=\{1,3,5\}$} \label{fig_quiverA5_K135}
\end{figure}

Let $\mathsf{V}^\prime = \{1,2,3,7,9\}$ be a subset of the vertex set $\mathsf{V}$ of $\mathsf{Q}_P$ and consider the induced subquiver $\mathsf{Q}^\prime = \mathsf{Q}_P |_{\mathsf{V}^\prime}$. By applying the mutation sequence $(2,3,1)$, we see that $\mathsf{Q}^\prime$ is of finite mutation type $\widetilde{A}_{3,2}$. Now take $\mathsf{V} = \mathsf{V}^\prime \cup \{12\}$ and consider the induced subquiver $\mathsf{Q} = \mathsf{Q}_P |_{\mathsf{V}}$. By Lemma~\ref{lemma_infinitefinite2}, the quiver $\mathsf{Q}$ is of infinite mutation type.
\end{example}

\begin{remark}\label{rmk_arbitraryaddfrozen}
In Example~\ref{example_A5K135}, suppose that, instead of adding $\{ 12 \}$, we add all frozen vertices corresponding to simple roots in $K$, that is,
$$
\mathsf{V} = \mathsf{V}^\prime \cup \{6, 10, 11\},
$$
then the induced subquiver $\mathsf{Q} = \mathsf{Q}_P |_{\mathsf{V}}$ remains of finite mutation type. To obtain infinite mutation type, one must add a frozen vertex corresponding to a simple root in $I \setminus K$.
\end{remark}

Using the same method, we handle the remaining cases. Table~\ref{table_minimalparboliclemmaAtype} summarizes the results. The first column lists the type of the group $G$, the second specifies the chosen subset of simple roots, the third provides the minimal length representative $w_{K}$ of the coset $w_0 W_K$, and the fourth gives a reduced expression for $w_{K}$. The fifth column describes a finite mutation type subquiver $\mathsf{Q}^\prime$ of the marked quiver $(\mathsf{Q}_P, \mathsf{L}_P)$ together with the lemma applied, and the final column indicates the frozen vertex added to $\mathsf{Q}^\prime$.

\begin{table}[h]
\begin{center}
\resizebox{.95\hsize}{!}{
\begin{tabular}{| c | c | c | c | c | c | }
 \hline 
 Type & $I \setminus K$ & $w_{K}$ & $\underline{w}_K$ & Subquiver of finite mutation type & Frozen vertex  \\ 
\hline  \hline
$A_5$ & $\{2, 4\}$ & (5,6,3,4,1,2) & $435241324354$ & $\widetilde{A}_{3,2}$-type (Lemma~\ref{lemma_infinitefinite2})& $s_4$  \\   
$A_5$ & $\{1, 3, 5\}$ & (6,4,5,2,3,1) & $3241524324135$ & $\widetilde{E}_7$-type (Lemma~\ref{lemma_infinitefinite1}) & $s_3$ \\
\hline     
$A_6$ & $\{1,3\}$ & (7,5,6,1,2,3,4) &  43524613524321 & $\widetilde{E}_7$-type (Lemma~\ref{lemma_infinitefinite3}) & $s_3$ \\   
$A_6$ & $\{1,4\}$ & (7,4,5,6,1,2,3) & 654324135246354 & $\widetilde{E}_8$-type (Lemma~\ref{lemma_infinitefinite4}) & $s_4$ \\   
$A_6$ & $\{3,4\}$ & (5,6,7,4,1,2,3) & 435246135246354 &  $\widetilde{A}_{3,3}$-type (Lemma~\ref{lemma_infinitefinite2})& $s_4$  \\   
\hline
$A_7$ & $\{4\}$ & (5,6,7,8,1,2,3,4) & 4352461357246354 & $\widetilde{A}_{3,3}$-type (Lemma~\ref{lemma_infinitefinite2})& $s_4$  \\   
$A_7$ & $\{1,5\}$ & (8,4,5,6,7,1,2,3) & 7654324135246357465 &  $\widetilde{E}_6$-type (Lemma~\ref{lemma_infinitefinite6}) & $s_5$  \\   
$A_7$ & $\{2,3\}$  & (7,8,6,1,2,3,4,5) & 54635724613524132 & ${E}^{(1,1)}_8$-type (Lemma~\ref{lemma_infinitefinite5}) & $s_2$  \\   
$A_8$ & $\{3\}$  & (7,8,9,1,2,3,4,5,6) & 657468357246135243 & ${E}^{(1,1)}_8$-type (Lemma~\ref{lemma_infinitefinite5}) & $s_3$ \\   
\hline
\end{tabular} 
\quad \quad
}
\end{center}
\caption{Minimal cases of infinite type of type $A$}
\label{table_minimalparboliclemmaAtype}
\end{table}

Next, we deal with the case of type $D_n$.

\begin{example}
Consider the classical group $G$ of type $D_4$ with $K = \{2\}$ and let $P = P^K$ be the corresponding parabolic subgroup. Then $G/P$ is isomorphic to the isotropic flag variety $\mathcal{IF}(1,3,4; 8)$ of dimension $m = 12$. The Weyl group $W$ of $G$ embeds into the symmetric group $\frak{S}_8$. The minimal length representative of the coset $w_0 W_K$ is given by 
$$
w_K = 
\left(\begin{array}{cccccccc}
1 & 2 & 3 & 4 & 5 & 6 & 7 & 8 \\
8 & 6 & 7 & 5 & 4 & 2 & 3 & 1
\end{array}\right).
$$
A reduced expression of $w_K$ is 
$$
s_1 s_3 s_2 s_3 s_4 s_2 s_1 s_3 s_2 s_3 s_4
$$ 
and the corresponding marked quiver $(\mathsf{Q}_P, \mathsf{L}_P = \{7, 10, 11\})$ is depicted in Figure~\ref{fig_quiverD4_K134}. The set of frozen variables is $\mathsf{V}_\mathrm{fz} = \{7, 9, 10, 11\}$.
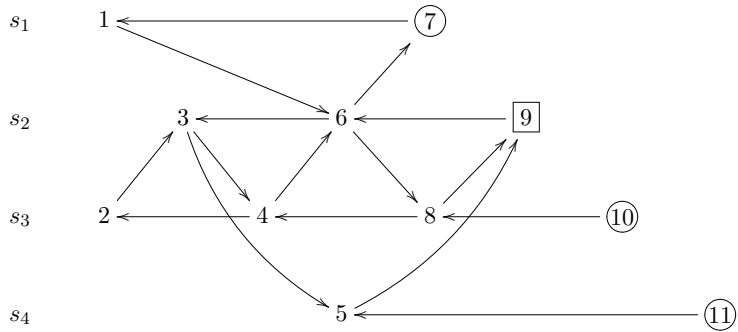
\begin{figure}[h]
$$
\resizebox{.7\hsize}{!}{
\xymatrix{
s_1 & 1 \ar[rrrd]	& 	&	&	&{\xy*{7}*\cir<7pt>{}\endxy} \ar[llll]	&	&&&\\
s_2 &	&3 \ar[rd]  \ar@/_1pc/[rrdd]	&	&6 \ar[ll] \ar[ru] \ar[rd]	&	&\fbox{9} \ar[ll] &&\\
s_3 &2 \ar[ru]	&	&4 \ar[ru] \ar[ll]	&	&8 \ar[ll] \ar[ru]	&	&{\xy*{10}*\cir<7pt>{}\endxy} \ar[ll]\\
s_4 &	&	&	& 5 \ar@/_1pc/[rruu] 	&	&	& 	& {\xy*{11}*\cir<7pt>{}\endxy} \ar[llll]
}
}
$$
\caption{The marked quiver for $D_4$ type with $K=\{1,3,4\}$} \label{fig_quiverD4_K134}
\end{figure}

Let $\mathsf{V}^\prime = V_\mathrm{uf}$. Applying the mutation at the nodes $(4, 1, 6, 2, 7, 5)$, the quiver $\mathsf{Q}_P$ contains a subquiver $\mathsf{Q}^\prime = \mathsf{Q}_P |_{\mathsf{V}^\prime}$ of finite mutation type $\widetilde{E}_{6}$. Now set $\mathsf{V} = \mathsf{V}^\prime \cup {10}$ and consider the induced subquiver $\mathsf{Q} = \mathsf{Q}_P |_{\mathsf{V}}$. By Lemma~\ref{lemma_infinitefinite6}, the quiver $\mathsf{Q}$ is of infinite mutation type.
\end{example}

Using the same method, the remaining cases can be treated similarly. The result is summarized in Table~\ref{table_minimalparboliclemmaDtype}.

\begin{table}[h]
\begin{center}
\resizebox{.95\hsize}{!}{
\begin{tabular}{| c | c | c | c | c | c | }
 \hline 
 Type & $I \setminus K$ & $w_{K}$ & $\underline{w}_{K}$ & Subquiver of finite mutation type & Frozen vertex  \\ 
\hline  \hline
$D_4$ & $\{2\}$ & (7,8,3,4,5,6,1,2) & $213242132$ & $\widetilde{A}_{2,2}$-type (Lemma~\ref{lemma_infinitefinite2})& $s_2$  \\   
$D_4$ & $\{1, 3, 4\}$ & (8,6,7,5,4,2,3,1) & $13234213234$ & $\widetilde{E}_6$-type (Lemma~\ref{lemma_infinitefinite6}) & $s_3$ \\
\hline     
$D_5$ & $\{3\}$  & (8,9,10,4,6,5,7,1,2,3) & $324351324132543$ & $\widetilde{A}_{3,3}$-type (Lemma~\ref{lemma_infinitefinite2})& $s_3$   \\   
$D_5$ & $\{1,5\}$ & (10,5,7,8,9,2,3,4,6,1) & $12435432132435$ & $\widetilde{E}_7$-type (Lemma~\ref{lemma_infinitefinite3})& $s_5$ \\   
$D_5$ & $\{4,5\}$ & (7,8,9,10,5,6,1,2,3,4) & $43532413243543$  & $\widetilde{E}_{7}$-type (Lemma~\ref{lemma_infinitefinite3})& $s_4$  \\   
\hline
$D_6$ & $\{6\}$ & (7,8,9,10,11,12,1,2,3,4,5,6) & $643546213243546$ & ${E}^{(1,1)}_7$-type (Lemma~\ref{lemma_infinitefinite7})  & $s_6$\\   
\hline
\end{tabular} 
\quad \quad
}
\end{center}
\caption{Minimal cases of infinite type of type $D$}
\label{table_minimalparboliclemmaDtype}
\end{table}

Finally, the results of type $E$ are summarized in Table~\ref{table_minimalparboliclemmaEtype}.

\begin{table}[h]
\begin{center}
\resizebox{.95\hsize}{!}{
\begin{tabular}{| c | c | c | c | c | c | }
 \hline 
 Type & $I \setminus K$ & $\underline{w}_{K}$ & Subquiver of finite mutation type & Frozen vertex  \\ 
\hline  \hline
$E_6$ & $\{1\}$ & $6543245613452431$ & $\widetilde{E}_7$-type (Lemma~\ref{lemma_infinitefinite3}) & $s_1$  \\   
$E_6$ & $\{2\}$ & $245341324565432451342$ & $\widetilde{A}_{2,3}$-type (Lemma~\ref{lemma_infinitefinite2})& $s_2$ \\
\hline     
$E_7$ & $\{7\}$  & $765432456713456245341324567$ & $\widetilde{A}_{1,2}$-type (Lemma~\ref{lemma_infinitefinite2})& $s_7$   \\   
\hline
$E_8$ & $\{8\}$ & $876543245671345624534132456787654324567134562453413245678$ &$\widetilde{E}_8$-type (Lemma~\ref{lemma_infinitefinite4})  & $s_8$\\   
\hline

\end{tabular} 
\quad \quad
}
\end{center}
\caption{Minimal cases of infinite type of type $E$}
\label{table_minimalparboliclemmaEtype}
\end{table}

\begin{proof}[Proof of Proposition~\ref{prop_mainsection7}]
We have verified that all minimal cases satisfy~\eqref{eq:manyarrows2}.  
The remaining cases follow from Corollaries~\ref{corollary_Dynkin1} and~\ref{corollary_Dynkin2}, together with Lemmas~\ref{lemma_an1} and~\ref{lemma_dn2}.  
\end{proof}

%\bibliographystyle{amsalpha}
%\bibliography{ref}

\providecommand{\bysame}{\leavevmode\hbox to3em{\hrulefill}\thinspace}
\providecommand{\MR}{\relax\ifhmode\unskip\space\fi MR }
\providecommand{\MRhref}[2]{%
  \href{http://www.ams.org/mathscinet-getitem?mr=#1}{#2}
}
\providecommand{\href}[2]{#2}

\end{document}